\documentclass[11pt]{amsart}

\usepackage[svgnames,x11names]{xcolor}
\usepackage{amssymb,amsfonts}
\usepackage[pagebackref,colorlinks,urlcolor={DarkOliveGreen4},
linkcolor={DarkSlateBlue},
citecolor={Chocolate4}]{hyperref}
\usepackage[capitalize]{cleveref}
\usepackage{amssymb,textcomp}
\usepackage{graphicx}
\numberwithin{equation}{section}
\usepackage{enumerate}

\usepackage[utf8]{inputenc}
\usepackage[T1]{fontenc}

\usepackage[mathscr]{eucal}

\usepackage[a4paper, twoside=false, vmargin={2cm,3cm}, includehead]{geometry}

\newtheorem{lemma}{Lemma}[section]
\newtheorem{theorem}[lemma]{Theorem}
\newtheorem*{theorem*}{Theorem}
\newtheorem{corollary}[lemma]{Corollary}

\newtheorem*{question*}{Question}
\newtheorem{proposition}[lemma]{Proposition}
\newtheorem*{proposition*}{Proposition}

\newtheorem{conjecture}{Conjecture}
\newtheorem{problem}{Problem}
\newtheorem*{problem*}{Problem}

\theoremstyle{definition}
\newtheorem{definition}{Definition}[section]
\newtheorem*{claim*}{Claim}

\newtheorem*{remark}{Remark}
\newtheorem*{remarks}{Remarks}

\newcommand{\C}{{\mathbb C}}
\newcommand{\E}{{\mathbb E}}
\newcommand{\D}{{\mathbb D}}

\newcommand{\N}{{\mathbb N}}
\renewcommand{\P}{{\mathbb P}}
\newcommand{\Q}{{\mathbb Q}}
\newcommand{\R}{{\mathbb R}}
\renewcommand{\S}{\mathbb{S}}
\newcommand{\T}{{\mathbb T}}
\newcommand{\Z}{{\mathbb Z}}
\newcommand{\U}{{\mathbb U}}

\newcommand{\CI}{{\mathcal I}}

\newcommand{\CM}{{\mathcal M}}

\newcommand{\CP}{{\mathcal P}}

\newcommand{\CX}{{\mathcal X}}

\newcommand{\CZ}{{\mathcal Z}}

\newcommand{\one}{\mathbf{1}}

\newcommand{\norm}[1]{\left\Vert #1\right\Vert}
\newcommand{\nnorm}[1]{\lvert\!|\!| #1|\!|\!\rvert}

\begin{document}
	
	\title[Decomposition results for   multiplicative actions and applications]{Decomposition results for multiplicative actions and applications}

	\author{Nikos Frantzikinakis}
	\address[Nikos Frantzikinakis]{University of Crete, Department of mathematics and applied mathematics, Voutes University Campus, Heraklion 71003, Greece} \email{frantzikinakis@gmail.com}

	\begin{abstract}
	Motivated by partition regularity problems of homogeneous quadratic equations, 	we prove multiple recurrence and convergence results for multiplicative measure preserving actions with iterates given by rational sequences involving polynomials that factor into products of linear forms in two variables. We focus mainly on actions that are finitely generated, and the key tool in our analysis is a  decomposition result for  any bounded measurable function into a sum of two components, one that mimics concentration properties  of  pretentious multiplicative functions and another that mimics vanishing properties of aperiodic multiplicative functions.
	 Crucial to part of our arguments are some new seminorms that are defined by a mixture of
	addition and multiplication of the iterates of the action,
	and we prove an inverse theorem that  explicitly characterizes the factor that these seminorms naturally determine.
	\end{abstract}
	
	\thanks{The author was supported  by the Hellenic Foundation for Research and Innovation	ELIDEK HFRI-NextGenerationEU-15689. }



%\date{\today}
	\maketitle

\setcounter{tocdepth}{1}	

\tableofcontents
	
	\section{Introduction}
A {\em multiplicative  measure preserving action}  is a quadruple $(X,\CX,\mu,T_n)$, where $(X,\CX,\mu)$ is a Lebesgue probability space, and $T_n\colon X\to X$, $n\in\N$, are
 invertible measure preserving transformations that satisfy $T_1=\text{id}$ and  $T_{mn}=T_m\circ T_n$ for all $m,n\in \N$. In a few cases we may also consider non-invertible actions.  We  extend the action to the positive rationals by  $T_{m/n}:=T_m\circ T^{-1}_n$ for all $m,n\in \N$.
 Following~\cite{BR22}, we say that the action is {\em finitely generated}  if the set of commuting transformations $\{T_p\colon p\in \P\}$ is finite.

Additive measure preserving actions  have been widely used to study problems in additive combinatorics concerning translation invariant patterns that occur within any
set of integers with positive upper density.
  The prototypical examples
   are Furstenberg's proof  of Szemer\'edi's theorem on arithmetic progressions~\cite{Fu77} and its polynomial extension by Bergelson and Leibman \cite{BL96}, which
   motivated  other  powerful multiple recurrence and mean convergence results in ergodic theory; see  \cite{HK18}
   for
   related recent trends.

 In complete analogy, multiplicative measure preserving actions can be used to study dilation invariant patterns that occur within any ``multiplicatively large'' set of integers, i.e., a set of integers with positive multiplicative density.
 One of our main motivations is to study the partition and density regularity (see \cref{D:prdr}) of equations of the form
 $$
 P(x,y,z)=0,
$$
where $P$ is a homogeneous quadratic polynomial; the equations $x^2+y^2=z^2$ and $x^2+y^2=2z^2$ are the best known examples. Although several such equations are expected to be partition regular,
progress has been rather scarce, with some partial results appearing in \cite{FH17,FKM23,FKM24}.\footnote{For homogeneous equations with more than three variables, or non-homogeneous quadratic equations, there are more results,  see for example  \cite{BLM21, Be87, BP17, Ch22, CLP21, Mo17}.}
 A serious limitation of the methodology developed in these works is the dependence on a representation result of Bochner-Herglotz, which  only allows us to deal with length two patterns, i.e., to guarantee that two out of the three variables
$x,y,z$ belong to the desired set.  For example, it was shown in \cite{FKM23}  that every set  of integers with positive multiplicative density contains pairs $x,y$ and $x,z$ such that $x^2+y^2=z^2$.

Our main goal here is to start developing  a methodology that will allow us to deal with patterns of length greater than two.
 Although our techniques are currently not directly applicable to   problems related with the  Pythagorean  or other diagonal equations, we will make some progress towards other homogeneous equations.
To give an explicit example (see Section~\ref{SS:PR}  for a wider range of examples),  suppose  we want to show that the equation 
$$
x^2-y^2=xz 
$$
is  density regular according to  \cref{D:prdr}  (and hence partition regular).
Note that
$$
x=k\, m^2, \quad y=k\, mn, \quad z=k\, (m^2-n^2)
$$
satisfy the equation for all $k,m,n\in \Z$.  Using a variant of Furstenberg's correspondence principle  (see \cite{Be05}), it suffices to show that for any multiplicative action $(X,\CX,\mu,T_n)$ and set $A\in \CX$ with $\mu(A)>0$, we have
\begin{equation}\label{E:m2mn}
\mu(T^{-1}_{m^2}A\cap T^{-1}_{mn}A\cap T^{-1}_{m^2-n^2}A)>0
\end{equation}
for some $m, n\in \N$ with $m>n$. Although we are currently unable to prove this multiple recurrence property for general multiplicative actions, we will prove it  for all finitely generated ones.
 In fact, in this case we show that  the closely related multiple ergodic averages (our averaging notation is explained at the end of this section)
\begin{equation}\label{E:m2mn'}
\E_{m,n\in[N], m>n} \,  T_{m/n}F_1\cdot T_{(m^2-n^2)/(mn)}F_2
\end{equation}
 converge in $L^2(\mu)$ as $N\to\infty$ for all $F_1,F_2\in L^\infty(\mu)$,  and by analyzing this limit we are also able to obtain optimal lower bounds for the multiple intersections in \eqref{E:m2mn}, roughly of the form $(\mu(A))^3$. For a much more general multiple recurrence  and convergence result,
 which  also covers  not necessarily commuting multiplicative actions, see \cref{T:MainB}.

The  methodology we develop also allows  us to verify  multiple recurrence and convergence results when all the iterates are given by several pairwise independent linear forms. For example, we show that if $(X,\CX,\mu,T_n)$ is a finitely generated multiplicative action and   $\mu(A)>0$, then for every $\varepsilon>0$, for a set of $m,n\in \N$ with positive lower density\footnote{A subset $E$ of $\N^2$ has {\em positive lower density} if
	$
	\liminf_{N\to\infty}\frac{|E \cap ([N]\times [N])|}{N^2}>0.
	$} we have
$$
\mu(A\cap T^{-1}_mA\cap T^{-1}_nA\cap T^{-1}_{m+n}A\cap\cdots\cap  T^{-1}_{m+\ell n}A)\geq (\mu(A))^{\ell+3}-\varepsilon.
$$
We also show that  the  closely related multiple ergodic averages
\begin{equation}\label{E:linit}
\E_{m,n\in[N]}\,  T_{m}F_0\cdot  T_{m+n}F_1\cdots T_{m+\ell n}F_{\ell}
\end{equation}
converge in $L^2(\mu)$ for all $F_0,F_1,\ldots, F_{\ell}\in L^\infty(\mu)$. Both results fail for some infinitely generated actions.
For a more general statement involving several pairwise independent linear forms and  different multiplicative actions, see \cref{T:MainA}. As a consequence, in \cref{C:Omega} we get that if $T_0,\ldots, T_\ell$ are ergodic, not necessarily commuting, measure preserving transformations acting on a probability space $(X,\CX,\mu)$, then
$$
\lim_{N\to\infty} \E_{m,n\in [N]} \prod_{j=0}^\ell T_j^{\Omega(m+jn)}F_j =\prod_{j=0}^\ell \int F_j\, d\mu
$$
in $L^2(\mu)$ for all $F_0,\ldots, F_\ell \in L^\infty(\mu)$, where $\Omega(n)$ denotes the number of prime divisors of $n$ counting multiplicity. See also \cref{C:additive}, which covers    recurrence and convergence  results for more general completely  additive sequences on the integers.

The proof of these results depends crucially on  decomposition results for multiplicative actions that we believe to be  of independent interest, see Theorems~\ref{T:Decomposition} and \ref{T:DecompositionGeneral}. Roughly, they state that a function $F\in L^\infty(\mu)$ can be decomposed as
\begin{equation}\label{E:Fpu}
F=F_p+F_a,
\end{equation}
where the spectral measures of $F_p$ and $F_a$ are supported on pretentious and aperiodic multiplicative functions, respectively.
This allows us to get our hands on  a variety of  results in analytic number theory and relate them directly
to properties of the
components  $F_p$ and $F_a$.
For example, for  finitely generated multiplicative actions, we show that
$$
\lim_{k\to\infty} \limsup_{N\to\infty} \E_{n\in [N]} \norm{T_{k!n+b}F_p-T_bF_p}_{L^2(\mu)}=0
$$
holds for every $b\in \N$,
and for general multiplicative actions we show that
$$
\lim_{N\to\infty} \norm{\E_{m,n\in [N]} \, T_{P(m,n)}F_a}_{L^2(\mu)} =0
$$
for any polynomial $P$ that is a product of  pairwise independent linear forms.

Finally, we note that  in the case of several linear iterates, as in  \eqref{E:linit},
our averages are controlled by some seminorms $\nnorm{\cdot}_{U^s}$ defined  on $L^\infty(\mu)$ using a mixture of addition and multiplication of the iterates $T_n$ (see \cref{D:mixedsem}). For finitely generated actions, in \cref{T:Inverse}, we prove a clean inverse theorem for these seminorms, which states that
$$
\nnorm{F}_{U^s}=0 \, \,  \forall s\in \N \Leftrightarrow  \nnorm{F}_{U^2}=0 \Leftrightarrow \lim_{N\to\infty} \norm{\E_{n\in[N]}\, T_{an+b}F}_{L^2(\mu)}=0 \quad  \forall a,b\in \N,
$$
and we deduce that
$
\nnorm{F_a}_{U^s}=0$ for every $s\in \N$.
  It is an interesting open problem to find an analogous  inverse theorem for general multiplicative actions (see \cref{Pr:Inverse} in \cref{S:Problems}).
  We suspect that such   results will play a  role in the analysis of even more delicate multiple recurrence and convergence results of general multiplicative actions;  for example,  \cref{Pr:Pythagorean} in \cref{S:Problems} is related to
  Pythagorean triples.

  \subsection*{Recent developments}
  
  We mention some developments closely related to this article that appeared
  after its first version was posted in March 2025. The works \cite{FM25,LMM26}
  make use of notions and techniques introduced here, whereas 
  \cite{CKT26} resolves an open problem posed in this article.
  \begin{itemize}
  	\item In \cite{FM25}, the author and Mountakis established multiple recurrence results
  	for pretentious multiplicative actions along generalized Pythagorean triples.

  	\item In \cite{CKT26}, Charamaras, Koutsogiannis, and Tsinas proved norm convergence
  	of multiple ergodic averages for commuting finitely generated multiplicative actions,
  	thereby solving the last problem in \cref{S:Problems}.
  	
  	\item In \cite{LMM26}, Leng, McNamara, and Mountakis proved that every set of positive
  	multiplicative upper density contains $x,y,z,w$ satisfying
  	$x^2+y^2w^2=z^2$, and established more general multiple recurrence results for
  	multiplicative actions.
  	
  \end{itemize}

\subsection*{Notation and conventions} \label{SS:notation}
We let  $\N:=\{1,2,\ldots\}$,  $\Z_+:=\{0,1,2,\ldots \}$, $\R_+:=[0,+\infty)$,  $\Q_+:=\Q\cap \R_+$,  $\Q^*_+:=\Q\cap (0,+\infty)$, $\S^1$ be the unit circle, and $\U$ be  the closed complex unit disk.
For $t\in \R$, $z\in \C$, we let $e(t):=e^{2\pi i t}$, $\exp(z):=e^z$, $\Re(z)$
be the real
part of $z$.
%respectively.

With $\P$ we
denote the set of primes and  we use the letter $p$ to denote primes.
We write $a\mid b$ if the integer $a$ divides the integer $b$ and we use a similar notation for polynomials.

For  $N\in\N$, we let $[N]:=\{1,\dots,N\}$.
 A {\em 2-dimensional grid} is a subset $\Lambda$ of $\Z^2$ of the  form $
\Lambda=\{(a_1m+b_1,a_2n+b_2)\colon m,n\in\Z\}
$, where  $a_1,a_2\in \N$ and $b_1,b_2\in \Z$.

%We often denote sequences $a\colon \N\to \U$
%by  $(a(n))$, instead of $(a(n))_{n\in\N}$.

If $A$ is a finite non-empty set and $a\colon A\to \C$, we let
$$
\E_{n\in A}\, a(n):=\frac{1}{|A|}\sum_{n\in A}\, a(n).
$$

%A subset $E$ of $\N^d$ has {\em positive lower density} if
%$
%\liminf_{N\to\infty}\frac{|E \cap [N]^d|}{N^d}>0.
%$

We write $a(n)\ll b(n)$ if for some $C>0$ we have $a(n)\leq C\, b(n)$ for every $n\in \N$.

Throughout this article, the letter $f$ is typically  used for multiplicative functions,  the letter $\chi$ for Dirichlet characters, and the letter $F$ for measurable functions.

 Whenever we consider commuting, possibly non-invertible measure
preserving transformations or actions, we implicitly pass to a common
invertible natural extension, identifying functions on the original
space with their pullbacks.

\subsection*{Acknowledgement}
	 I would like to thank B.~Host,  B.~Kuca, J.~Moreira,  A.~Mountakis, and F.~Richter,  for a number of useful comments on a preliminary version of this article, and J.~Leng for pointing out
	 an error in a previous version of the manuscript, which led to a modification of the weight
	 sequences used in the proofs.

	\section{Main results and applications}
	
	 \subsection{Multiple recurrence and convergence  for multiplicative actions}
 We describe here our two main results about multiple recurrence and convergence of multiplicative actions. 
 In \cref{SS:conj}  we state two conjectures  that guide our main results
 and our first main result verifies these conjectures
 when the rational polynomials are given by pairwise independent linear forms.
Henceforth, we say that $(X,\CX,\mu,T_{1,n}, \ldots, T_{\ell,n})$ is a {\em (finitely generated) multiplicative action} if
 	$(X,\CX,\mu,T_{j,n})$, $j\in[\ell]$, is a (finitely generated)
 	multiplicative action. Commutativity is not assumed unless explicitly stated.

 \begin{theorem}[Several linear forms]\label{T:MainA}
	Let $(X,\CX,\mu,T_{1,n},\ldots, T_{\ell,n})$ be a  finitely generated multiplicative action and $L_1,\ldots, L_\ell\in\Z[m,n]$ be pairwise independent\footnote{Two linear forms are {\em independent} if they are non-zero and their quotient is not constant.} linear forms with non-negative coefficients. Then the following properties hold:
	\begin{enumerate}
		\item  \label{I:MainA1}  For all $F_1,\ldots, F_\ell\in L^\infty(\mu)$  and every $2$-dimensional grid  $\Lambda$, the averages
		\begin{equation}\label{E:MainA0}
			\E_{m,n\in [N]} \, {\bf 1}_\Lambda(m,n)\cdot T_{1,L_1(m,n)} F_1\cdots T_{\ell,L_\ell(m,n)}F_\ell
		\end{equation}
		converge in $L^2(\mu)$ as $N\to\infty$. Furthermore,  if all actions are aperiodic (see \cref{SSS:aperiodic}) and $\Lambda=\Z^2$, then the limit is equal to $\int F_1\, d\mu\cdots \int F_\ell\, d\mu$.
		
		\item 	 \label{I:MainA2}  Suppose that there exist $m_0,n_0\in \Z$ such that $L_1(m_0,n_0)$ is either $0$ or $1$  and $L_j(m_0,n_0)=1$ for $j=2,\ldots, \ell$.  Then for every $A\in \CX$ and $\varepsilon>0$, the set
		\begin{equation}\label{E:MainA2}
	\{(m,n)\in \N^2\colon 		\mu(A\cap T^{-1}_{1,L_1(m,n)}A\cap \cdots \cap T^{-1}_{\ell,L_\ell(m,n)}A) \geq (\mu(A))^{\ell+1}
			-\varepsilon\}
		\end{equation}
		has positive lower density.
	\end{enumerate}
\end{theorem}
%\begin{remarks}
%$\bullet$ Our argument shows that the result still %holds if we use affine l%inear forms $L'_j:=L_j+c_j$, %where $c_j\in \Z$, instead of $L_j$, as long as we %insist that $L_1,\ldots, L_\ell$ are pairwise %independent, and %in part~\eqref{I:MainA2} the  %assumptions are imposed on $L'_j$.
%$\bullet$
Both parts fail for general multiplicative actions.
For   part~\eqref{I:MainA1},  take  $\ell=1$ and  consider a multiplicative rotation (defined in \cref{SS:examples}) by $n^i$, or some $1$-pretentious oscillatory multiplicative function as in example~\eqref{I:Ex3} of \cref{SS:pretentious}.   Also, the action by dilations $T_nx:=nx\pmod{1}$, $n\in \N$, acting on $\T$ with the Haar measure,  is aperiodic, 
but for $F_1(x)=F_2(x):=e(x), F_3(x):=e(-x)$, we have $T_mF_1\cdot T_nF_2\cdot T_{m+n}F_3=1$ for every $m,n\in \N$. 
  For   part~\eqref{I:MainA2}, consider again the previous action by dilations
and let $A:=[1/3,2/3)$. Then  $\mu(A)>0$ 	and  $\mu(T^{-1}_m A\cap T^{-1}_n A\cap T^{-1}_{m+n}A)=0$ for every $m,n\in \N$.
Using the same action we get that
for every $\ell\in \N$ we can define   a set $B=B_\ell$ as in the proof of \cite[Theorem~2.1]{BHK05}, such that $\mu(B)>0$ and  $\mu(T^{-1}_mB\cap T^{-1}_{m+n}B\cap T^{-1}_{m+2n}B)\leq (\mu(B))^\ell/2$ for all $m,n\in \N$.

By combining an elementary intersectivity lemma  (see \cite[Theorem~1.1]{Be85}) and Szemer\'edi's theorem on arithmetic progressions~\cite{Sz75},
Bergelson  showed in  \cite[Theorem~3.2]{Be05}  that for arbitrary multiplicative actions, for every $\ell\in \N$ there exist $m,n\in \N$ such that $\mu(T^{-1}_mA\cap T^{-1}_{m+n}A\cap\cdots\cap  T^{-1}_{m+\ell n}A)>0$.
  %However,  it is not clear if a similar recurrence %property holds for  $\ell+1$ multiplicative %actions; see part~\eqref{I:P2b} of \cref{Pr:linear} %in \cref{S:Problems}.

Our second result  is more conveniently stated for rational rather than integer polynomials, a class of sequences that we define next.
\begin{definition}
	We say that $R(m,n)$ is a {\em rational polynomial that factors linearly} if it can be represented in the form
	$
	R(m,n)=c\, \prod_{j=1}^sL^{k_j}_j(m,n),
	$ where $c\in \Q^*_+$, $k_j\in \Z$,  for $j\in [s]$, and  $L_j(m,n)=\alpha_jm+\beta_jn$, $j\in [s]$, are pairwise independent linear forms with  $\alpha_j,\beta_j\in \Z_+$.\footnote{Our assumption that $\alpha_j,\beta_j\geq 0$ is made for technical convenience. It  can be easily removed in the recurrence results of this article by making appropriate substitutions, and saves us unnecessary technicalities  in the convergence results. }
	The {\em degree of $R$} is $\deg(R):=\sum_{j=1}^s k_j$.
	We say that    $(a,b)$ is a {\em simple zero} of $R$
	if  $L_j(a,b)=0$  for some $j\in [s]$ with $k_j=1$, and $L_i(a,b)\neq 0$ for every other occurring factor $i\neq j$.
\end{definition}
We are going to  deal with two rational polynomials that factor linearly.
We are particularly interested in these cases because  the obtained  multiple recurrence results are related to  partition and density regularity of homogeneous quadratic equations in three variables, see the discussion in \cref{SS:PR}.
Again,  in our setting we do not need to  impose  any commutativity assumptions on the  actions.

\begin{theorem}[Two rational polynomials]\label{T:MainB}
	Let $(X,\CX,\mu,T_{1,n}, T_{2,n})$ be a  finitely generated multiplicative action, and  $R_1,R_2$ be rational polynomials that factor linearly. Suppose that   $R_2(m,n)=c\,  L_1(m,n)^{k}\cdot L_2(m,n)^{l}$ for some  independent linear forms $L_1,L_2$, and $c\in \Q^*_+,k,l\in \Z$, and $\gcd(|k|,|l|)=1$.\footnote{For nonconstant $R_2$, both convergence and recurrence reduce
		to the primitive case by replacing $T_{2,n}$ with $T_{2,n^d}$,
		where $d:=\gcd(|k|,|l|)$.}
	 Suppose also that  $R_1$ is not of the form $c'\, L_1^{k'}L_2^{l'}R^r$ for any  $c'\in \Q_+,k',l'\in \Z$, rational polynomial $R$, and $r\geq 2$.   Then the following properties hold:
	\begin{enumerate}
	\item  \label{I:MainB1}  For all $F_1, F_2\in L^\infty(\mu)$ and  every $2$-dimensional grid $\Lambda$,   the averages
		\begin{equation}\label{E:R1R2Converge}
		\E_{m,n\in [N]} \, {\bf 1}_\Lambda(m,n)\cdot T_{1,R_1(m,n)} F_1\cdot T_{2,R_2(m,n)}F_2
	\end{equation}
	converge in $L^2(\mu)$ as $N\to \infty$.

	\item 	 \label{I:MainB2}
Suppose that  there exist $m_0,n_0\in \Z_+$ such that
$R_2(m_0,n_0)=1$ and either $R_1(m_0,n_0)=1$ or  $(m_0,n_0)$ is a simple zero of $R_1$.  Then for every  $A\in \CX$ and  $\varepsilon>0$, the set 
\begin{equation}\label{E:MainB2'}
\{(m,n)\in \N^2\colon 	\mu(A\cap T^{-1}_{1,R_1(m,n)}A\cap  T^{-1}_{2,R_2(m,n)}A) \geq (\mu(A))^3
	-\varepsilon\}
\end{equation}
has positive lower density.

\item  \label{I:MainB3}
If  $\deg(R_1)=0$  and there exist $m_0,n_0\in \Z_+$ such that
$R_j(m_0,n_0)=1$ for $j=1,2$,  then the multiple  recurrence   property  \eqref{I:MainB2} holds even if the
action 	$(X,\CX,\mu,T_{1,n})$ is infinitely generated.
\end{enumerate}
\end{theorem}
In particular, this vastly extends  the recurrence and convergence results in \eqref{E:m2mn} and \eqref{E:m2mn'} in the introduction.
Regarding part \eqref{I:MainB3}, the multiple recurrence property \eqref{E:MainB2'} fails in several ways if we allow both actions to be infinitely generated, even when $T_{1,n}=T_{2,n}$ for every $n\in \N$.
Indeed,    the remarks following \cref{T:MainA} show that the
lower bound in \eqref{E:MainB2'}   fails
% even if $R_1(m_0,n_0)=R_2(m_0,n_0)=1$
when  $R_1(m,n):=(m+n)/m$,  $R_2(m,n):=(m+2n)/m,$
although $R_1(1,0)=R_2(1,0)=1$. Moreover,   we may even have non-recurrence when   $R_1(m,n):=n/(m+n)$, $R_2(m,n):=m/(m+n),$ although  $R_1(1,0)=0, R_2(1,0)=1$.

See  also \cite[Theorem~1.5]{CMT24} for  recurrence results of expressions  $\mu(T^{-1}_{an+b}A\cap T^{-1}_{cn+d }A)$.
\subsection{Decomposition results for multiplicative actions}
Crucial to the 	 proof of Theorems~\ref{T:MainA} and \ref{T:MainB}
are some decomposition results for  multiplicative actions that are of independent interest.
Here is the  statement for finitely generated actions
(the mixed seminorms $\nnorm{\cdot}_{U^s}$ are defined in Section~\ref{S:mixedseminorms}, $X_p$ is  defined in \cref{D:pa}, and  \cref{C:factor} establishes that it is a factor):

\begin{theorem}\label{T:Decomposition}
	Let  $(X,\CX,\mu,T_n)$ be  a finitely generated multiplicative action and
	$F\in L^\infty(\mu)$. Then  we have the decomposition
	$$
	F=F_p+F_a, \, \quad \text{where  } \,  F_p=\E(F|\CX_p) \, \text{ and } \,  F_a\bot X_p,
	$$
	and $F_p, F_a\in L^\infty(\mu)$ satisfy the following properties:
	\begin{enumerate}
		\item \label{I:Decomposition1} For every $b\in \N$ we have
		$$ \lim_{K\to\infty}\limsup_{N\to\infty} \max_{Q\in \Phi_K}\E_{n\in[N]}\norm{T_{Qn+b}F_p-T_bF_p}_{L^2(\mu)}=0,
		$$
		where $\Phi_K$ is as in \eqref{E:PhiK};
		
		\item \label{I:Decomposition2}  $\nnorm{F_a}_{U^s}=0$ for every $s\in \N$;
		
		\item \label{I:Decomposition3}  $\lim_{N\to\infty}\norm{\E_{m,n\in [N]} T_{R(m,n)}F_a}_{L^2(\mu)}=0$ whenever $R(m,n)$ is a rational polynomial that factors linearly, and is  not of the form $ c\, R'^r$ for any  $c\in \Q_+$, rational polynomial $R'$, and $r\geq 2$.
	\end{enumerate}
\end{theorem}
For finitely generated multiplicative actions
Charamaras in \cite[Theorem~1.28]{Cha24} proved a decomposition result of a similar spirit but with different information about the component functions, for more details
see  the remark after \cref{D:pa} below.
We also stress that the mixed seminorms $\nnorm{\cdot}_{U^s}$ are not the analogues of the  Host-Kra seminorms \cite{HK05} for the multiplicative action $T_n$ (see the discussion in \cref{SS:mixeddefinition}). The definition of $\nnorm{\cdot}_{U^s}$ uses   a mixture of addition and multiplication;  this combination is better suited for our purposes, since the ergodic averages we aim to study also involve such a mixture.

Next, we give a decomposition result that works for general multiplicative actions. An important difference in this case is that we have to restrict our averaging in the  concentration estimates to sets of the form
 
\begin{equation}\label{E:Sd}
	S_{\delta,N}:=\{n\in\N\colon e^{-\delta} N <n\leq N\},
\end{equation} 		
where $\delta>0$ and then take $\delta\to 0^+$ (for each  $\delta>0$ we have $|S_{\delta,N}|/N\to   1 - e^{-\delta} > 0$). This is necessary because multiplicative rotations by $n^{it}$ (in which case $X_p=L^2(\mu)$) do not exhibit any concentration unless we restrict our averaging. Another difference is that in this case we cannot claim concentration at $T_bF_p$, since the
averages $\E_{n\in[N]}\, T_{Qn+b}F_p$ in general may not even converge in $L^2(\mu)$.
Here is the exact statement:
\begin{theorem}\label{T:DecompositionGeneral}
	Let  $(X,\CX,\mu,T_n)$ be  a  multiplicative action and
	$F\in L^\infty(\mu)$. Then  we have the decomposition
	$$
	F=F_p+F_a, \, \quad \text{where  } \,  F_p=\E(F|\CX_p) \, \text{ and } \,  F_a\bot X_p,
	$$
	and $F_p, F_a\in L^\infty(\mu)$ satisfy the following properties:
	\begin{enumerate}
		\item\label{I:DecompositionGeneral1} For every $b\in \N$ we have
		$$
		\lim_{\delta\to 0^+}	\limsup_{K\to\infty}\limsup_{N\to\infty} \max_{Q\in \Phi_K}\E_{n\in S_{\delta,N}}\norm{T_{Qn+b}F_p-A_{\delta, Q,N,b}}_{L^2(\mu)}=0,
		$$
		where $\Phi_K$ is as in \eqref{E:PhiK}, $S_{\delta,N}$ are as in \eqref{E:Sd},
		and 	$A_{\delta, Q,N,b}:=\E_{n\in S_{\delta,N} }\, T_{Qn+b}F_p$;

		\item\label{I:DecompositionGeneral2} $\lim_{N\to\infty}\norm{\E_{m,n\in [N]} T_{R(m,n)}F_a}_{L^2(\mu)}=0$ whenever $R(m,n)$ is a rational polynomial that factors linearly, and is  not of the form $c\, R'^r$ for any   $c\in \Q_+$, rational polynomial $R'$, and $r\geq 2$.
	\end{enumerate}
\end{theorem}
Note that in contrast to the finitely generated case, we  cannot claim that $\nnorm{F_a}_{U^s}=0$, even for $s=2$. To see this, consider the multiplicative action defined by $T_nx:=nx\pmod{1}$, $n\in \N$,  on $\T$ with the Haar measure $m_\T$. In this case  it is easy to show that $X_p$ is trivial and so for    $F(x):=e(x)$ we have  $F=F_a$ but $\nnorm{F}_{U^2}=1$.

In both decomposition results, depending on the application we have in mind, we plan to also   use more refined properties about
the components $F_p$ and $F_a$  that are given  in Propositions~\ref{P:Concentrationfg},  \ref{P:aperiodicP0}, \ref{P:ConcRfg},  \ref{P:ConcRgeneral}.

\subsection{Connections with partition and density regularity}\label{SS:PR}
 A \emph{multiplicative F\o lner sequence in $\N$} is a sequence $\Phi=(\Phi_K)_{K=1}^\infty$ of finite subsets of $\N$ that is asymptotically invariant under dilation, in the sense that
$$ \lim_{K\to\infty}\frac{\big|\Phi_K \cap (x\cdot \Phi_K)\big|}{|\Phi_K|}=1 \quad \text{for every } x\in \N.$$
An example of a multiplicative F\o lner sequence is
$$
\Phi_K	:=\Big\{\prod_{p\leq K}p^{a_p}\colon a_K\leq  a_p\leq b_K\Big\}, \quad K\in \N,
$$
where $a_K,b_K\in \N$ are such that $b_K-a_K\to \infty$ as $K\to \infty$.
We say that $E\subset\N$ has {\em positive multiplicative density}
if
$$\limsup_{K\to\infty}\frac{\big|E\cap \Phi_K\big|}{|\Phi_K|}>0
$$
 for some multiplicative F\o lner sequence $(\Phi_K)_{K=1}^\infty$.
\begin{definition}\label{D:prdr}
	If  $P\in \Z[x,y,z]$, we say that the equation
	$P(x,y,z)=0$ is
	\begin{enumerate}
		\item {\em partition regular} if for every finite partition of $\N$ there exist distinct  $x,y,z$ in the same cell  that satisfy the equation;
		
		\item {\em density regular} if for every subset $E$ of $\N$  with positive multiplicative density there exist distinct  $x,y,z\in E$ that satisfy the equation.
	\end{enumerate}	
\end{definition}
%We require that $x,y,z$  be distinct, otherwise %equations like %$x^2+y^2=2z^2$ would be trivially %partition and density  %regular.
Note that  a finite partition of $\N$ always contains a monochromatic cell with positive multiplicative density, hence density regularity implies partition regularity. However, the converse is not true; for $a,b,c\in \N$,  it is known that the equation
$$
ax+by=cz
$$
is partition regular if and only if  $(a,b,c)$ is a {\em Rado triple}, i.e., if either $a=c$, or $b=c$, or $a+b=c$ (see \cite{R33}), but it is density regular only if $a+b=c$ (the sufficiency follows from \cite[Theorem~3.2]{Be05} and the necessity by an example of Bergelson, see remarks after \cite[Theorem~1.2]{FKM23}).
A difficult and well known problem
is to find  a similar characterization for the partition regularity of the equation $P(x,y,z)=0$ when $P$ is a homogeneous quadratic polynomial, and
the diagonal equations
$$
ax^2+by^2=cz^2
$$ have received  the most attention.
Currently, we  have only partial results that    allow us to decide in some cases when two of the three variables belong to the same partition cell or set of positive multiplicative density (see \cite{FH17,FKM23,FKM24}).  Non-diagonal quadratic equations present similar challenges, but as we will show next, in some cases we can make progress that we cannot currently reproduce in the diagonal setting.

For example, suppose we want to prove partition or density regularity of   the equation
\begin{equation}\label{E:xyz}
	ax^2+by^2= dxy+exz+fyz \quad \text{when } a+b=d.\footnote{More generally, with a bit more effort our approach works if $d^2-4ab$ is a square.}
	%\footnote{%The important feature is that %$(0,0,1)$ %and $(1,1,0)$ are solutions, this %allows us to find solutions with special features %of the form  described in %\cref{C:MainRecurrence}.}
\end{equation}
As a typical example, the reader can keep in mind   the equation
\begin{equation}\label{E:xyxz}
x^2-y^2=xz
\end{equation}
with solutions
\begin{equation}\label{E:xyxzpar}
x:=k\, m^2,\quad y:=k\, mn,\quad z:=k\, (m^2-n^2), \quad k,m,n\in\Z.
\end{equation}

%To deal with \eqref{E:xyz} we start with some %reductions.
%We assume that $a,b$ are not both $0$ and $e,f$ are not %both $0$.\footnote{If $a=b=0$, then \eqref{E:xyz} reduces to
%	 $f/x+e/y+d/z=0$, which is known to be partition regular if %and only if $fx+ey+dz=0$ is partition regular \cite{L91}, and %this equation is covered by Rado's result~\cite{R33}. Similarly %if $e=f=0$, since we are led  to $(x-y)(ax-by)=0$. }
%% $P(x,y,z)=0$ is PR if and only if %%$P(x^{-1},y^{-1},z^{-1})=0$ is PR. Because PR in $\Z$ is %equivalent with PR %in $\Q$ and in $\Q$ the map $x\mapsto %x^{-1}$ is a bijection. }	
%Without loss of generality we can assume that $a\in \N$.
%Also, since $P(x,y,z)=0$ is partition regular  on $\Z$ if and %only if $P(-x,-y,-z)=0$ is partition regular on $\Z$, we can also %assume that either $e$ or $f$ is positive;  let us assume that %$e\in \N$, the other case is similar. Finally, we can assume that %$e+f\neq 0$, since  if $e+f=0$ we are led to the equation %$(x-y)(ax-by-ez)=0$, in which case it is easy to decide %whether it is partition or density regular using Rado's %result~\cite{R33}.

More generally, under a few  technical assumptions on the coefficients (as  in \cref{C:MainRecurrence}),
it seems likely that  the equations in \eqref{E:xyz}  are density regular.
%\footnote{To give some  credence to this claim, we note that using the methodology of \cite{FKM23}, it is possible to show that these equations are density regular for all three pairs $(x,y)$, $(x,z)$, $(y,z)$.}
Using a variant of  Furstenberg's correspondence principle \cite{Be05},  it  suffices to prove a
 multiple recurrence property for arbitrary multiplicative actions.
It is   a consequence of  part~\eqref{I:MainB2} of \cref{T:MainB}, that  this multiple recurrence property holds for finitely generated  actions: 
\begin{corollary}\label{C:MainRecurrence}
Let $(X,\CX, \mu,T_n)$ be a finitely generated multiplicative action and $A\in \CX$ with $\mu(A)>0$. Let also  $a,e\in \N$, $b,d,f\in \Z$,  so that  $a+b=d$, $a>b$, and  $e+f> 0$. Then  there exist distinct
 $x,y,z\in \N$, not all of them equal,  that  satisfy \eqref{E:xyz} and
\begin{equation}\label{E:Txyz}
\mu(T^{-1}_xA\cap T^{-1}_yA\cap T^{-1}_zA)>0.
\end{equation}
%Furthermore,  $x,y,z$ can be chosen to be different, %unless
%$a=d=e=-f,b=0$ (in which case \eqref{E:xyz} reduces %to $(x-y)(x-z)=0$).
\end{corollary}
Let us first see why  \cref{C:MainRecurrence} is a consequence of part~\eqref{I:MainB2} of \cref{T:MainB}
in our working example of the equation \eqref{E:xyxz}. We start with the solutions \eqref{E:xyxzpar} and  perform the substitution $m\mapsto m+n,  n\mapsto n$ to get
  the solutions
$(m+n)^2, \, (m+n)n, \, m(m+2n)$
 with  non-negative coefficients. Using these solutions in the place of $x,y,z$,   the recurrence property in \eqref{E:Txyz} can be rewritten as (after factoring out $T^{-1}_{(m+n)^2}$)
$$\mu(A\cap T^{-1}_{m(m+2n) (m+n)^{-2}}A\cap T^{-1}_{n(m+n)^{-1}}A )>0.
$$
Note that  $R_1(m,n):=m(m+2n) (m+n)^{-2}$ and $R_2(m,n):=n(m+n)^{-1}$ are rational polynomials that factor linearly, $R_2$ is a quotient of two independent linear forms,
$R_1$ is not of the form $c\, n^k(m+n)^lR^r$ for any $c\in \Q_+$,  rational polynomial $R$,  $r\geq 2$, and
 $(0,1)$ is a simple zero of $R_1$ and $R_2(0,1)=1$.
Thus, part~\eqref{I:MainB2} of \cref{T:MainB} applies for $T_{1,n}=T_{2,n}:=T_n$ and gives the required multiple recurrence property.

To deal with   \eqref{E:xyz},  we will use the
following solutions:
$$
x:=m( em+f n), \quad y:=n(e m+fn), \quad z:=(m-n)(am-bn).
$$
Since $a>b$ and $e+f>0$, the substitution $m\mapsto m+n$ gives
the following positive solutions to \eqref{E:xyz} for $m,n\in\N$
\begin{equation}\label{E:mln}
 x=(m+n)(em+(e+f)n), \qquad y=n(em+(e+f)n), \qquad 
		z=m(am+(a-b)n).
\end{equation}
%Suppose that $l \in \N$ is such that  $le+f>0$ and %$la-b\geq 0$, then by  substituting $m\mapsto m+ln, %n\mapsto n,$ we get the solutions (with non-negative %coefficients)
%\begin{equation}\label{E:mln}
%x= (m+ln)(e m+(le+f) n), \quad  y=n(e m+(le+f) n), \quad  %z=(m+ (l-1)n)(am+(la-b)n).
%\end{equation}
Therefore, to prove \cref{C:MainRecurrence} it  suffices to establish the following result:
\begin{proposition}\label{P:MainRecurrence}
	Let $(X,\mathcal{X},\mu, T_n)$ be a finitely generated  multiplicative action and $A\in \CX$ with
	$\mu(A)>0$. Suppose that the linear forms
	$L_1,L_2,L_3, L_4, L_3-L_4$ have non-negative coefficients and
each of the three pairs $(L_3,L_4)$, $(L_1,L_3-L_4)$, $(L_2, L_3-L_4)$ consists of  independent linear forms.
Suppose also that $L_3-L_4$ is a positive multiple of $m$ or $n$.
	Then
		$$
	\liminf_{N\to \infty}\E_{m,n\in [N]}\,\mu(T^{-1}_{
	%\gamma
	 L_1(m,n)\cdot (L_3-L_4)(m,n)}A \cap T^{-1}_{L_2(m,n)\cdot L_3(m,n)} A\cap T^{-1}_{L_2(m,n)\cdot L_4(m,n)}A)>0.
	$$
\end{proposition}
\begin{remark}
 Using the parametrization \eqref{E:mln}, we get that
 \cref{C:MainRecurrence} follows by taking
 $L_1(m,n):=am+(a-b)n$, $L_2(m,n):=em+(e+f)n$,
 $L_3(m,n):=m+n$, and $L_4(m,n):=n$.
 Then $L_3-L_4=m$, all coefficients are non-negative, and the
 three independence conditions follow from $a-b>0$ and $e+f>0$.
 Moreover, $x>y$ on $\N^2$,  and the zero sets of the polynomials $z-x$ and $z-y$ have
 density zero, so the recurrence pairs can be chosen with
 $x,y,z$ pairwise distinct.
\end{remark}
\begin{proof}
To see how this follows from part~\eqref{I:MainB2} of \cref{T:MainB}, note that the asserted recurrence property can be rewritten as
$$
\liminf_{N\to \infty}\E_{m,n\in [N]}\,\mu(A\cap T^{-1}_{R_1(m,n)}A \cap T^{-1}_{R_2(m,n)} A)>0,
$$
where $
R_1(m,n):= L_1(m,n)\cdot (L_3-L_4)(m,n)\cdot L_2^{-1}(m,n)\cdot L_4^{-1}(m,n)$, and  $R_2(m,n):=L_3(m,n)\cdot L_4^{-1}(m,n)$.
Since  $L_3-L_4$ is not a rational multiple of $L_j$ for $j=1,2,3,4$, we  get    that $R_1$ is not of the form $c\,  L_3^k\cdot L_4^l\cdot R^r$ for any   $c\in \Q_+$,  $k,l\in \Z$,  rational polynomial $R$, and $r\geq 2$.  
If $L_3-L_4$ is a positive multiple of $m$, we  choose
$(m_0,n_0)=(0,1)$; otherwise we choose $(m_0,n_0)=(1,0)$.
The independence assumptions imply that all four forms $L_j$ are
positive at this point. Thus $R_2(m_0,n_0)=1$, $(m_0,n_0)$ is a simple
zero of $R_1$, and the exponents of $R_2=L_3/L_4$ are $1,-1$,
so they are relatively prime.

 Hence, the assumptions of part~\eqref{I:MainB2} of \cref{T:MainB}
 are satisfied,  giving the claimed recurrence property.
\end{proof}
\subsection{Multiple recurrence and convergence  for additive actions}\label{SS:additive}
We say that the sequence  $a\colon \N\to \Z$ is {\em completely additive} if $a(mn)=a(m)+a(n)$ for every $m,n\in \N$, and {\em finitely generated} if the set $\{a(p)\colon p\in \P\}$ is finite.
We  extend the sequence to $\Q^*_+$ by letting $a(m/n):=a(m)-a(n)$ for all $m,n\in\N$.
We make the following observation, the proof of which is rather straightforward, so we omit it (see also \cite[Corollary~1.19]{BR22} for a related observation).
\begin{lemma}\label{L:fgchar}
	Let $(X,\CX,\mu,T_n)$ be a multiplicative action. Then  the action is finitely generated if and only if there exist $\ell\in \N$, commuting  invertible measure preserving transformations $S_1,\ldots, S_\ell\colon X\to X$, and finitely generated  completely additive sequences $a_1,\ldots, a_\ell\colon \N\to \Z$, such that
	\begin{equation}\label{E:Tnproduct}
			T_n=S_1^{a_1(n)}\cdots S_\ell^{a_\ell(n)}, \quad n\in \N.
	\end{equation}
\end{lemma}
A special case of interest is when  $T_n=T^{\Omega(n)}$, where $\Omega(n)$ is the number of prime factors of $n$ counting multiplicity.
%,  and, more generally,
%when $T_n=T^{a(n)}$, where
%$a\colon \N\to \Z$ is an arbitrary finitely %generated completely additive sequence.

 %\begin{proof}
 %The converse direction is immediate, so we prove the forward direction.
   	
 %Suppose  that the action $(X,\mu,T_n)$ is finitely generated.  Then  %there exist invertible measure preserving transformations $S_1,\ldots, %S_\ell\colon X\to X$ such that
 %	$
 %	\{T_p\colon p\in \P\}=\{S_1,\ldots, S_\ell\}.
 %	$
 %	Since $T_pT_q=T_qT_p$ for every $p,q\in \P$, we get that the %transformations $S_1,\ldots, S_\ell$ commute.  	
%If  	$o_p(n)$ denotes the highest power of $p$ that divides $n$, we %let
 %		$$
 %	a_j(n):=\sum_{p\in \P\colon T_p=S_j}o_p(n),  \quad j\in [\ell].
 %	$$
%Using that $T_mT_n=T_nT_m=T_{mn}$ for all $m,n\in \N$,   it is easy %to verify
%that the sequences $a_1,\ldots, a_\ell\colon \N\to \Z_+$ are  completely %additive sequences and
%\eqref{E:Tnproduct} holds. Finally, since $a_j(p)\in \{0,1\}$ for $j\in %[\ell]$  and $p\in \P$, they are also  finitely generated.
% \end{proof}
%Conjecture~\ref{Con1} seems equally challenging in the following two special %cases:
%$(X,\mu,T_1,\ldots, T_\ell)$ is   an (additive) system, $a_1,\ldots, %a_\ell\colon \N\to \Z_+$ are (finitely generated) completely additive %sequences, and
%$T_{j,n}=T_j^{a_j(n)}$, $j=1,\ldots, \ell$.  In this case we
%we deal with averages of the form
%$$
%\E_{m,n\in [N]} \, T_1^{a_1(P_1(m,n))} F_1\cdots %T_\ell^{a_\ell(P_\ell(m,n))}F_\ell
%$$
	Using \cref{L:fgchar} and
	Theorems~\ref{T:MainA} and \ref{T:MainB},  we  derive    the following multiple recurrence and mean convergence result for (additive) measure preserving systems:
	\begin{corollary}\label{C:additive}
		Let $(X,\CX,\mu)$ be a probability space and $T_1,\ldots, T_\ell\colon X\to X$ be invertible measure preserving transformations (not necessarily commuting). Let also $a_1,\ldots, a_\ell\colon \N\to \Z$ be
		finitely generated completely additive sequences.
		\begin{enumerate}
			\item 	\label{I:additive1}  If $R_1,\ldots, R_\ell$  satisfy the assumptions of the convergence results in  \cref{T:MainA} or \cref{T:MainB}, then for every %$2$-dimensional grid $\Lambda$  and
			$F_1,\ldots, F_\ell\in L^\infty(\mu)$ the averages
			$$
			\E_{m,n \in [N]}\, %{\bf 1}_\Lambda(m,n)\cdot
			 T_1^{a_1(R_1(m,n))}\, F_1\cdot \ldots\cdot T_\ell^{a_\ell(R_\ell(m,n))}F_\ell
			$$
			converge in $L^2(\mu)$. Furthermore,   under the assumptions of \cref{T:MainA},
			 if all actions $T_j^{a_j(n)}$, $j\in [\ell]$, are aperiodic, then  the  limit is equal to $\int F_1\, d\mu\cdots \int F_\ell\, d\mu$.
			
			\item  	\label{I:additive2}  If $R_1,\ldots, R_\ell$  satisfy the assumptions of the recurrence results in  \cref{T:MainA} or \cref{T:MainB},  then
			for every $A\in \CX$  and   $\varepsilon>0$, the set 		
				$$
		\{(m,n)\in \N^2\colon 	 \mu(A\cap T_1^{-a_1(R_1(m,n))}A\cap \cdots\cap T_\ell^{-a_\ell(R_\ell(m,n)) }A)\geq (\mu(A))^{\ell+1}-\varepsilon\}
			$$
		has positive lower density.
	%	\footnote{Using Furstenberg's Correspondence Principle~\cite{Fu77},
		%	we get  applications related to configurations that can be found within subsets of the integers (or $\Z^d$) with positive upper density.}
			% these consequences are  rather standard and  %straightforward to derive, so we omit them.}
		\end{enumerate}
	\end{corollary}
	Using Furstenberg's Correspondence Principle~\cite{Fu77},
	part~\eqref{I:additive2} gives applications related to configurations that can be found within subsets of $\Z^\ell$ with positive upper density.
	For instance, we get that if $\Lambda$ is a subset of $\Z$ with positive upper density and  $a_1,a_2, a_3, a_4\colon \N\to \Z$ are
	finitely generated completely additive sequences, then there exist $x\in \Z$, $m,n\in \N$ for which 
	$$
	x,\, x+a_1(m),\, x+a_2(n),\,  x+a_3(m+n),\,  x+a_4(m+2n)\in \Lambda. 
	$$
	
Combining part~\eqref{I:additive1} and our remarks in \cref{SSS:aperiodic} below, we get the following:
\begin{corollary}\label{C:Omega}
	Let $T_1,\ldots, T_\ell$ be ergodic measure preserving transformations acting on a probability space $(X,\CX,\mu)$ and $L_1,\ldots,L_\ell\in \Z[m,n]$ be pairwise independent linear forms with non-negative coefficients. Then
	\begin{equation}\label{E:TjOm}
		\lim_{N\to\infty} \E_{m,n\in [N]} \prod_{j=1}^\ell T_j^{\Omega(L_j(m,n))}F_j =\prod_{j=1}^\ell \int F_j\, d\mu
	\end{equation}
	in $L^2(\mu)$ for all $F_1,\ldots, F_\ell \in L^\infty(\mu)$.\footnote{Our method also gives that for general  systems we have convergence to $\prod_{j=1}^\ell \E(F_j|\CI_{T_j})$ in  \eqref{E:TjOm}.  }
\end{corollary}
 If $\ell=1$, the system is uniquely ergodic, and $F_1\in C(X)$,  then ~\cite{BR22} gives that \eqref{E:TjOm}  holds pointwise.
It seems likely that  a similar property holds for all $\ell\in \N$.

Identity \eqref{E:TjOm} is highly non-trivial even if all the $T_j$'s are rotations on $\{-1,1\}$; it recovers some known uniformity properties of  the Liouville function~\cite[Proposition~9.1]{GT09}.

 We get another interesting application of \cref{C:additive}
by
 letting $T_jx:=b_jx\pmod{1}$, $j\in [\ell]$, on $\T$ with $m_\T$, where $b_1,\ldots, b_\ell\in \N$ satisfy $b_j\geq 2$ and are not necessarily distinct.  As noted in \cref{SSS:aperiodic}, all the actions
 $T_j^{\Omega(n)}$, $j\in [\ell]$, are
  aperiodic, so part~\eqref{I:additive1} of \cref{C:additive} easily gives that every sequence $N_k\to \infty$ has a subsequence $N_k'\to \infty$ such that for almost every $x\in \T$ the following holds:
If  $L_1,\ldots, L_\ell$ are pairwise independent linear forms with non-negative coefficients, then for all $F_1,\ldots, F_\ell\in C(\T)$ we have
$$
\lim_{k\to \infty} \E_{m,n\in [N'_k]} \, F_1(b_1^{\Omega(L_1(m,n))}x)\cdots F_\ell(b_\ell^{\Omega(L_\ell(m,n))}x)=\int F_1\, dm_\T\cdots \int F_\ell\, dm_\T.
$$
We deduce that  if  $\text{dig}_b(x;n)$ denotes the $n$-th digit (after the decimal point) of $x\in[0,1)$ in base $b$, and $d_{\bf N'}(E)$ denotes the density of
a subset $E$ of $\N^2$  along the squares $[N_k']\times [N_k']$, then
for almost every $x\in [0,1)$ and for all  $c_j\in \{0,\ldots, b_j-1\}$, $j\in [\ell]$,
we have 
\begin{equation}\label{E:digits}
	d_{\bf N'}\big((m,n)\in \N^2\colon \text{dig}_{b_j}\big(x;\Omega(L_j(m,n))+1\big)=c_j, j\in [\ell]\big)=(b_1\cdots b_\ell)^{-1}.
\end{equation}
The shift by $1$ follows from
$
\{b^k x\}\in[c/b,(c+1)/b)
\quad\Longleftrightarrow\quad
\text{dig}_b(x;k+1)=c.
$

Finally, using part~\eqref{I:additive2} of \cref{C:additive}, we can use finitely generated completely additive sequences $a_j\colon \N\to \Z_+$ instead of $(\Omega(n))$. For example, we get that if $b:=\max(b_1,\ldots, b_\ell)$, then for almost every $x\in [0, b^{-1})$ (then $\text{dig}_{b_j}\big(x;1)=0$, $j\in[\ell]$)
 if the pairwise independent linear forms $L_1,\ldots, L_\ell$ satisfy the assumptions of part~\eqref{I:MainA2} in  \cref{T:MainA}, then
the  following set has positive upper density (with respect to $[N]\times [N]$)
$$
\big\{(m,n)\in\N^2\colon  \text{dig}_{b_1}\big(x;a_1(L_1(m,n))+1\big)=\cdots =\text{dig}_{b_\ell}\big(x;a_\ell(L_\ell(m,n))+1\big)=0\big\}.
 $$
% has positive upper density (with respect to the squares %$[N]\times [N]$).
	
	\section{Proof strategy and main ideas}
	Here we briefly sketch the main ideas of the proof of our main results.
	We deal with some special cases in order to avoid unnecessary technicalities.
		\subsection{Proof sketch of Theorems~\ref{T:Decomposition} and \ref{T:DecompositionGeneral}}\label{SS:proofsketchDecomposition}
		We start with  \cref{T:Decomposition} and  explain the main differences in the proof  of \cref{T:DecompositionGeneral}  at the end of this subsection.

		Let $(X,\CX, \mu,T_n)$ be a finitely generated multiplicative action and $F\in L^2(\mu)$. The  spectral measure $\sigma_F$ of $F$ is supported on the compact space of completely multiplicative functions with values on the unit circle and can be decomposed into two components $\sigma_p$ and $\sigma_a$,
		which are supported on the complementary Borel sets  of pretentious and aperiodic multiplicative functions.  The spectral theory of unitary operators gives that there exist  functions  $F_p$ and $F_a$ with spectral measures $\sigma_p$ and $\sigma_a$ respectively and $F=F_p+F_a$.
		Pretentious multiplicative functions are known to satisfy various concentration estimates  (as in \cref{T:concestlinear}), which we  show  are inherited  by  the iterates $T_nF_p$ (as in \cref{P:Concentrationfg}).
		We  also  show using \cref{P:Concentrationgeneral} that the subspace $X_p\cap L^\infty(\mu)$ of  bounded  functions  with spectral  measures supported on pretentious multiplicative functions is  a conjugation closed algebra, hence it defines a factor, and
		$F_p=\E(F|\CX_p)$.
		
		Finally, we need to verify the two  vanishing  properties for $F_a$.
		 Since the spectral measure of $F_a$ is supported on aperiodic multiplicative functions, the vanishing property of part~\eqref{I:Decomposition3} follows from the corresponding property of aperiodic multiplicative functions  \cite[Theorem~2.5]{FH17}. This property  holds  for general  multiplicative actions as well.

		The fact that $\nnorm{F_a}_{U^s}=0$ for every $s\in \N$ requires  more work and uses in an essential way that the action is finitely generated.  		
		The key step  is to establish the following inverse theorem: If   $F\in L^\infty(\mu)$ satisfies
		$\lim_{N\to\infty}\norm{\E_{n\in[N]}\,  T_{qn+r}F}_{L^2(\mu)}=0$ for every    $q,r\in \N$, then
				 $\nnorm{F}_{U^s}=0$ for every $s\in \N$.  The assumption is easily shown to be satisfied by $F_a$, since its spectral measure is supported on aperiodic multiplicative functions.
			To prove the inverse theorem, we roughly argue as follows: We can assume that $\norm{F}_\infty\leq 1$. If the conclusion fails, then for  some  sequence along which a mixed seminorm stays positive  there is a further
		subsequence $N'_k\to\infty$  and  a set $E\in \CX$ with $\mu(E)>0$, such that
		\begin{equation}\label{E:FTnU}
		  \limsup_{k\to\infty} \norm{F(T_nx)}_{U^s[N'_k]}>0 \quad \text{for all } x\in E.
		\end{equation}
	By a further diagonal extraction, we may assume that, for almost
	every $x\in X$, the averages
	$\E_{n\in[\alpha N'_k]}F(T_{qn+r}x)$ converge to zero as $k\to\infty$,
	simultaneously for all $q\in\N$, $r\in\Z_+$, and rational
	$\alpha\in(0,1]$.
		Now the key point is that since the action $T_n$ is finitely generated, there exists a set of primes $P_0$ with positive upper relative density in $\P$ such that $T_p=T_{p'}$ for all $p,p'\in P_0$, and we deduce that for almost every $x\in X$ the sequence $a_x(n):=F(T_nx)$ satisfies
		\begin{equation}\label{E:axp}
		a_x(pn)=  a_x(p'n) \quad \text{for all } p,p'\in P_0,\,  n\in \N,
		\end{equation}
which can be thought of  as a form of multiplicative structure. 		At this point the argument used to prove  \cite[Theorem~2.5]{FH17}  applies with only minor changes. It enables us to deduce from \eqref{E:FTnU} and \eqref{E:axp}, that along this subsequence, there exist   $q_0, r_0\in \N, \alpha_0\in \Q\cap (0,1]$,  such that
		$$
		\limsup_{k\to\infty} |\E_{n\in[\alpha_0 N'_k]}\,
		F(T_{q_0n+r_0}x)|>0
		$$
		 for all $x$ on a positive measure subset $E_0$ of $E$.
		This contradicts  the almost everywhere convergence to zero along  the chosen subsequence and completes the proof.

 The decomposition result of \cref{T:DecompositionGeneral}  covers  general multiplicative actions and can be treated in a similar way. There are a few differences though.  The first is that
we  restrict our averaging to the sets
$S_{\delta, N} $ defined in \eqref{E:Sd}
and  then take $\delta\to 0^+$, the reason being that multiplicative rotations by $n^{it}$, $t\in \R$,  are ``pretentious actions'' that do not  satisfy any useful unrestricted concentration results. The second is that  even for multiplicative rotations by $1$-pretentious multiplicative functions (such as the one in part~\eqref{I:Ex3} of \cref{SS:pretentious}),   we  have concentration around the  average $\E_{n\in[N]}\,  T_{Qn+r}F_p$, which may not be convergent in $L^2(\mu)$.  Finally, unlike the finitely generated case, the mixed seminorms of $F_a$ do not always vanish, see the comment after \cref{T:DecompositionGeneral}.

The details of these arguments appear in \cref{S:mixedseminorms}.
	
	\subsection{Proof sketch of \cref{T:MainA}}\label{SS:proofsketchA}
	Given a  finitely generated multiplicative action $(X,\CX, \mu,T_n)$ and $F_1,F_2,F_3, F_4\in L^\infty(\mu)$,
	suppose we want to show that the averages
	\begin{equation}\label{E:F123}
	\E_{m,n\in[N]} \, T_{m}F_1\cdot T_{n}F_2\cdot T_{m+n}F_3 \cdot T_{m+2n}F_4
	%%\cdot T_{m+2n}F_4
	\end{equation}
	converge in $L^2(\mu)$. Using a pointwise
	estimate, we get in \cref{P:Characteristic1} that  the $L^2(\mu)$ norm of these averages is  controlled by the mixed seminorms $\nnorm{F_j}_{U^3}$ for $j=1,2,3,4$, in the sense  that if one of these seminorms is zero, then the averages \eqref{E:F123} converge to $0$ in $L^2(\mu)$.
Using this fact and  the decomposition result of  \cref{T:Decomposition}
(in particular the vanishing  property of part~\eqref{I:Decomposition2}) we get that it suffices to show convergence of the averages \eqref{E:F123} when each of the functions $F_j$ is replaced by $F_{j,p}:=\E(F_j|\mathcal{X}_p)$.
In this case, for highly divisible values of $Q\in \N$,  if we ignore a negligible error,  we can split the average  over $m,n\in [N]$ into subprogressions  $Qm+a,Qn+b$, where $(a,b)$ belong to a suitable subset of $[Q]\times [Q]$,  along which
the term $T_{m+2n}F_{4,p}$ gets concentrated around the function
$F_{a,b}:=\lim_{N\to\infty} \E_{n\in [N]}\,  T_{Qn+a+2b}F_{4,p}$, where crucially  the last limit can be shown to exist. The concentration estimates we use require some uniformity on the $(a,b)\in [Q]\times [Q]$ and are given in part~\eqref{I:clii} of \cref{T:concestlinear}.
 Likewise, we get similar concentration results  for the other terms $T_{m}F_{1,p}$,  $T_{n}F_{2,p}$,  $T_{m+n}F_{3,p}$. We easily deduce from the above that  the averages \eqref{E:F123}	converge in $L^2(\mu)$.

To prove  part~\eqref{I:MainA2} of  \cref{T:MainA} for the linear forms $m,n,m+n,m+2n$, arguing as above, it suffices to show that  for every $A\in \CX$ and  $\varepsilon>0$ there exists $Q\in \N$ such that
\begin{equation}\label{E:wantedQ}
\liminf_{N\to\infty} 	\E_{m,n\in[N]} \int F\cdot  T_{Qm+1}F_p\cdot T_{Qn}F_p\cdot T_{Q(m+n)+1}F_p \cdot T_{Q(m+2n)+1}F_p\, d\mu\geq (\mu(A))^5-\varepsilon,
\end{equation}
where $F:={\bf 1}_A$ and  $F_p:=\E({\bf 1}_A|\CX_p)$. For highly divisible values of $Q$, using the concentration estimates  in part~\eqref{I:Decomposition1} of \cref{T:Decomposition}, which  give concentration around the function $F_p$, we get that the previous limit  is approximately equal to
$$
\int  F\cdot F_p \cdot T_Q\tilde{F}_p \cdot F_p \cdot F_p\, d\mu,
$$
where $\tilde{F}_p:=\lim_{N\to\infty} \E_{n\in [N]}\,  T_nF_p$ and this last limit can be shown to exist.  The last expression has no obvious positiveness property because of the term
$T_Q\tilde{F}_p$. To deal with  this problem, we adapt the ``$Q$-trick'' from   \cite{FKM23} to our ergodic setting. We average over $Q$ along a multiplicative F\o lner  sequence, and using the mean ergodic theorem for multiplicative actions, we arrive at the limit ($\CI$ is as in \eqref{E:I})
$$
\int  F\cdot F_p \cdot  \E(F|\CI) \cdot F_p \cdot F_p\, d\mu \geq (\mu(A))^5,
$$
where the last bound follows from \cref{T:Chu} since  $F_p$ is also given by a conditional expectation. Combining the above, we get that there exists $Q\in \N$ such that \eqref{E:wantedQ} holds.

 The details of these arguments can be found in \cref{S:ProofThmA}.

		\subsection{Proof sketch of \cref{T:MainB}}\label{SS:proofsketchB}
		Given a  finitely generated multiplicative action $(X,\CX, \mu,T_n)$ and $F,G\in L^\infty(\mu)$,
	suppose we want to show that the averages
	\begin{equation}\label{E:F12}
		\E_{m,n\in[N]} \,  T_{(m+n)(m+2n)}F \cdot T_{mn}G
		%%\cdot T_{m+2n}F_4
	\end{equation}
	converge in $L^2(\mu)$ (one of the iterates could have more than two linear terms but not both).
	 We first reduce to the case where
	both functions are measurable with respect to the pretentious factor $X_p$. To do this
	we   use  a  two-dimensional  variant of the orthogonality criterion of
	 Daboussi-K\'atai (see \cref{L:Katai}). To  take advantage  of the fact that the
 action is finitely generated we use that if $|G|=1$, then for a set of primes $P_0$ with positive upper  relative density in $\P$ we have that $T_p$ is constant  for $p\in P_0$,  hence
	\begin{equation}\label{E:fgcancelF}
	T_{pqmn}G\cdot T_{p'q'mn}\overline{G}=1 \quad \text{for all } p,q,p',q'\in P_0, \, m,n\in\N.
	\end{equation}
	We deduce from \cref{L:Katai} that the averages \eqref{E:F12} converge to $0$ in $L^2(\mu)$, provided that
	\begin{equation}\label{E:fgcancel}
\lim_{N\to\infty} 	\E_{\substack{m\in[N/\max(p,p')]\\ n\in[N/\max(q,q')]}} \,  \int T_{(pm+qn)(pm+2qn)}F \cdot  T_{(p'm+q'n)(p'm+2q'n)}\overline{F}\, d\mu=0
	\end{equation}
for all  $p,q,p',q'\in P_0$ such that $p/q\neq p'/q'$ (if we did not have \eqref{E:fgcancelF}, then the last integral would have four instead of two terms). We use the decomposition result of  \cref{T:Decomposition} to write  $F=F_p+F_a$ where $F_p\in X_p$ and $F_a\in X_a$. Property \eqref{I:Decomposition3} of this result implies that the  needed vanishing property holds when  $F$ is replaced by $F_a$. Hence, the limiting behavior of the averages  in \eqref{E:F12} does not change if we replace $F$ by $F_p$.
Next,
	for highly divisible values of $Q\in \N$,  if we ignore a negligible error,  we can split the average  over $m,n\in [N]$ into subprogressions  $Qm+a,Qn+b$, where $(a,b)$ belong to a suitable subset of $[Q]\times [Q]$, and use   concentration estimates to replace the iterates of $ T_{(m+n)(m+2n)}F_p$ in \eqref{E:F12} by a constant function. This simplifies our setting substantially, since we only have  to deal with the iterates of $G$, in which case we can use
the decomposition result 	of \cref{T:Decomposition} and the  vanishing property of $G_a$ given in  part~\eqref{I:Decomposition3} of that result, to deduce that the limiting behavior of the averages  in \eqref{E:F12} does not change when  we replace $G$ by $G_p$.
Now that we are able to replace both $F$ and $G$ in \eqref{E:F12} by $F_p$ and $G_p$, the rest of the argument follows the proof of  the convergence part of \cref{T:MainA}.

If  none of the iterates is $mn$, as is the case when   $R_1(m,n):=(m+n)(m+2n)$ and $R_2(m,n):=(m+3n)(m+4n)$, a direct use of  \cref{L:Katai}  gives no simplification. In this case we  make an appropriate substitution that allows us to replace $R_1$ and $R_2$ by  $\alpha \, L_1(m,n)\cdot L_2(m,n)$ and $\beta \, mn$ respectively, where $\alpha,\beta$ are positive rational numbers and $L_1,L_2$ are integer linear forms. Then we can argue as before.

To deal with recurrence, we argue as in the proof of \cref{T:MainA}, but the technical aspects are more delicate.  We use the  concentration estimates   of  part~\eqref{I:concseclinfinB} in   \cref{P:ConcRfg},   for the function $F_p$ where $F={\bf 1}_A$ and  choose our averaging grid to be
  $\{(Qm+m_0, Qn+n_0)\colon m,n\in \N\}$. 
In the case where $(m_0,n_0)$ is a simple zero of $R_1$
we also have to use the $Q$-trick in a similar   way as  described before.

 % or    $\{(Q^2m-Q+m_0,Q^2n+n_0)\colon m,n\in \N\}$. The %second option is used  in order to avoid  concentration %estimates along  progressions $Qn+r$  with $r<0$, since %they are not conveniently  expressible in the ergodic  %setting.

Finally, as an example for our proof strategy for  part~\eqref{I:MainB3} of \cref{T:MainB},
 suppose we want to prove positivity  for expressions of the form
$$
\mu(A\cap T_{1,m/(m+n)}A\cap T_{2,(m+2n)/(m+3n)}A)
$$
when only   the second action is finitely generated, in which case \cref{P:Characteristic2} still applies.
  In our analysis we   use the decomposition result of  \cref{T:DecompositionGeneral}   and the concentration result    of \cref{P:ConcRgeneral},
both of which deal with    general multiplicative actions, and  two additional difficulties arise.
 The first is that  simple ``pretentious actions'', such as  multiplicative rotations by $n^{it}$, do not obey any useful unrestricted concentration.
 We address this difficulty by restricting our averaging to the sets
 $S_{\delta,R}$ defined in~\eqref{E:SdeltaR} below, which in the present
 example take the form
 $
 \big\{(m,n)\in\mathbb \N^2\colon
 \big|\log \big( m/(m+n)\big)\big|\leq \delta\big\},
 $
 for small $\delta>0$.
A more substantial problem is  that iterates of ``pretentious'' functions  concentrate around functions that depend on $N$ and may not converge in $L^2(\mu)$ as $N\to \infty$. This non-convergence causes serious technical issues that were not present in the finitely generated case. Fortunately, in the case of
 rational polynomials of degree $0$, we show in \cref{P:ConcRgeneral} that for suitable concentration results, these oscillatory factors conveniently cancel, thus alleviating the
 problem.

 The details of these arguments can be found in \cref{S:ProofThmB1}.

\section{Background and preliminary results}
In this section  we gather some necessary background notions, examples, results, and consequences from number theory and ergodic theory, which will be used later on.
\subsection{Pretentious and aperiodic multiplicative functions}
A function $f\colon \N\to \U$, where $\U$ is the  unit disk,  is called {\em multiplicative} if
$$
f(mn)=f(m)\cdot f(n)  \quad \text {  whenever  }  (m,n)=1,
$$
and  {\em completely multiplicative} if the previous equation holds for all $m,n\in\N$.
We let
$$
\CM:=\{f\colon \N\to \S^1 \text{ is a completely multiplicative function}\}.
$$
We endow $\CM$ with the topology of pointwise convergence, which makes it a compact metric space.

\begin{definition}\label{D:preap}
A multiplicative function $f\colon \N\to\U$ is called {\em aperiodic} if
$$
\lim_{N\to\infty}\E_{n\in[N]}\, f(an+b)=0 \quad \text{for all } \, a\in \N, \, b\in\Z_+.
$$
It is called {\em pretentious} if it is not aperiodic.
It is called {\em finitely generated} if the set $\{f(p)\colon p\in \P\}$ is finite.

We denote  with $\CM_p$ and
 $\CM_{a}$ the complementary sets of  pretentious and aperiodic completely multiplicative functions $f\colon \N\to \S^1$.
 \end{definition}
 It is easy to verify that $\CM_p$ and $\CM_a$ are  Borel subsets of $\CM$   \cite[Lemma~3.6]{FKM23}.
%%We also denote by
 %% $\CM^{fg}$  the set of finitely generated completely %%multiplicative functions.

 \subsection{Pretentious multiplicative functions} \label{SS:pretentious}
 We give some examples of pretentious completely multiplicative functions; the first three are the most basic  and the last  is more sophisticated
 but useful to  keep in mind.
  \begin{enumerate}
 	\item \label{I:Ex1} (Dirichlet characters and their modifications)
 	A {\em Dirichlet character $\chi$} is  a non-zero periodic completely multiplicative function. Then  $\chi(Qn+1)=1$ whenever $Q$ is a multiple of its smallest period $q$ and $\chi(p)=0$ whenever $p\mid q$. We let $\tilde{\chi}\colon \N\to \S^1$ be the completely multiplicative function  defined by $\tilde{\chi}(p):=\chi(p)$ if $p\nmid q$ and
 	 $\tilde{\chi}(p):=1$ when  $p\mid q$, and call it a {\em modified Dirichlet character}.
 	
 	\smallskip
 	
 		\item \label{I:Ex1'} (Finitely supported) If  $f(p)=1$ for all but finitely many primes we say that $f$ has {\em finite support}.

 		\smallskip
 	\item\label{I:Ex2}  (Archimedean characters) An {\em Archimedean character} is a completely multiplicative function of the form  $n\mapsto n^{it}$, $n\in \N$, where $t\in \R$. Then
 	$\E_{n\in [N]} \, n^{it}$ is asymptotically     equal to $N^{it}/(1+it)$.
 	
 	\smallskip
 	
 	\item\label{I:Ex3}  ($1$-pretentious oscillatory)  Take  $f(p):=e((\log\log{p})^{-1})$, $p\in \P$. Then $f$ is $1$-pretentious  in a sense that we will explain below,  but  it turns out that $f$ does not have a mean value; in fact, $\E_{n\in [N]}\, f(n)$ is asymptotically equal to $c\, e(w(N))$, where
 	$w(N)=\log\log\log N$ and $c$ a  non-zero complex number.
 \end{enumerate}
Note that the first two examples  are finitely generated, and the last two  are not (with $t\neq 0$ in example~\eqref{I:Ex2}).

% It is a surprising and very useful  fact that   these are the only obstructions to aperiodicity.
A completely multiplicative function that is not aperiodic is, in a sense that will be made precise in \cref{T:DH} below, close to
 a product of a Dirichlet character and an Archimedean character. To make sense of this principle we use the following notions
 introduced by Granville and Soundararajan~\cite{GS07,GS08}:
 \begin{definition}\label{D:Pretentious}
 	If  $f,g\colon \N\to \U$ are multiplicative functions, we define their distance  as
 	\begin{equation}\label{E:D}
 		\D^2(f,g):=\sum_{p\in \P} \frac{1}{p} \big(1-\Re(f(p)\cdot \overline{g(p)})\big).
 		%%	\D^2(f,g):=\sum_{p\in \P} \frac{1-\Re(f(p)\cdot \overline{g}(p))}{p}.
 	\end{equation}
 We say that  {\em $f$ pretends to be $g$} and write $f\sim g$  if  $\D(f,g)<\infty$.
% 	(Note that for $f,g\colon \N\to  \S^1$ we have
 %	$
 %	\D^2(f,g)=\frac{1}{2}\sum_{p\in \P} \frac{1}{p} |f(p)- g(p)|^2.
 %	$)
 \end{definition}	
 It can be shown (see e.g.  \cite{GS08}) %or \cite[Section~2.1.1]{GS23})
 that $\D$ satisfies the triangle inequality
 \begin{equation}\label{E:triangle}
 \D(f, g) \leq \D(f, h) + \D(h, g)
 \end{equation}
 for all  $f,g,h\colon \N\to \U$.
 %Also, for all  $f_1, f_2, g_1, g_2\colon \N\to \U$,  we have (see
 %\cite[Lemma~3.1]{GS07})
 %\begin{equation}\label{E:Df1f2}
% 	\D(f_1f_2, g_1g_2) \leq \D(f_1, g_1) + \D(f_2, g_2).
% \end{equation}

 %It follows from \eqref{E:Df1f2} that if $f_1\sim g_1$ and $f_2\sim %g_2$, then $f_1f_2\sim g_1g_2$.

 The next result is a simple  consequence of   the
 Wirsing-Hal\'asz mean value theorem \cite{Hal68}, which  characterizes completely multiplicative functions that have mean value  $0$.
 \begin{proposition}\label{T:DH}
 	A  completely multiplicative function  $f\colon \N\to \U$ is  pretentious if and only if   there exist
 	$t\in \R$ and a Dirichlet character $\chi$ such that  $f\sim \chi\cdot n^{it}$.
 \end{proposition}
Using this characterization it is possible to show that pretentious multiplicative functions satisfy strong concentration properties that we describe next.

\subsubsection{Concentration estimates for pretentious multiplicative functions}
If $f\sim \chi\cdot n^{it}$ for some Dirichlet character $\chi$ and $t\in \R$ and
$K\in \N$, $N\in [K,\infty)$, we let
\begin{equation}\label{E:FNQ}
	F_N(f,K):=\sum_{K< p\leq N} \frac{1}{p}\,\big(f(p)\cdot \overline{\chi(p)}\cdot p^{-it} -1\big),
\end{equation}
\begin{equation}\label{E:PhiK}
	\Phi_K	:=\Big\{\prod_{p\leq K}p^{a_p}\colon K< a_p\leq 2K\Big\},
\end{equation}
\begin{equation}\label{E:defQKSK}
 Q_K:=\prod_{p\leq K}p^{2K}, \quad
 S_K:=\Big\{  a\in [Q_K]\colon    p^K\nmid a \text{ for all } p\leq K\Big\}.
\end{equation}
%Since $\D(f,\chi\cdot n^{it})<\infty$ we have
%\begin{equation}
%\lim_{K\to\infty}\sup_{N\geq K}|\Re(F_N(f,K))|=0, \quad %\lim_{N\to\infty} |F_{CN}(f,K)-F_N(f,K)|=0 .
%\end{equation}
Note   that for $K\to \infty$, the set    $S_K$   contains ``almost all''  values in $[Q_K]$.  We prove a more general fact that will be needed later.
\begin{lemma}\label{L:SKL}
Let $L_1,\ldots, L_\ell\in \Z[m,n]$ be non-trivial linear forms, $Q_K,S_K$ as in \eqref{E:defQKSK}, and
 	 \begin{equation}\label{E:defSKL} 
 	 S_{K;L_1,\ldots,L_\ell}
 	 := \{(a,b)\in[Q_K]^2 \colon
 	 p^K\nmid L_j(a,b)
 	 \text{ for all } p\in\P,\ p\le K,\ j\in[\ell]\}.
 	 \end{equation}
 	 Then
\begin{equation}\label{E:SKQKL}
	\lim_{K\to\infty} |S_{K;L_1,\ldots, L_\ell}|/Q^2_K=1.
\end{equation}
In particular, taking $\ell=1$ and $L_1(m,n):=m$,  we get
\begin{equation}\label{E:SKQK}
	\lim_{K\to\infty} |S_K|/Q_K=1.
\end{equation}
\end{lemma}
\begin{proof}
Note that if
$(a,b)\in [Q_K]^2\setminus S_{K;L_1,\ldots, L_\ell}$,
then  $p^K\mid L_j(a,b)$ for some prime $p\leq K$ and  $j\in [\ell]$.
For every fixed prime $p\leq K$ we have\footnote{For $L(a,b)=\alpha a+\beta b$, since $p^K\mid Q_K$,
		the number of pairs in $[Q_K]^2$ for which $p^K\mid L(a,b)$
		equals $Q_K^2\gcd(\alpha,\beta,p^K) p^{-K}
		\ll_L Q_K^2 p^{-K}$.}
$$
|\{(a,b)\in [Q_K]^2\colon p^K\mid L_j(a,b)\text{ for some } j\in [\ell]\}| \leq C Q_K^2/p^K
$$
for some $C:=C(L_1,\ldots, L_\ell)$.
 It follows that
$$
\frac{|[Q_K]^2\setminus S_{K; L_1,\ldots, L_\ell}|}{Q^2_K}\leq C \sum_{p\leq K}\frac{1}{p^{K}}\leq
C\sum_{p\leq K}\frac{1}{2^{K}}\leq \frac{CK}{2^{K}}\to 0
$$
as $K\to \infty$, which proves \eqref{E:SKQKL}.
\end{proof}
Next, we give some concentration estimates for pretentious multiplicative functions that will be crucial for our arguments. We will use  different concentration estimates depending
on whether we are  dealing with
 recurrence results, in which case an appropriately  chosen congruence class $a \! \! \mod{Q}$ suffices (as in \eqref{E:fconc}),  or
convergence results, in which case we want to cover almost all congruence classes $a \! \! \mod{Q}$ (as in \eqref{E:fconca}). The next result appears in \cite[Lemma 2.5]{KMPT21}:
\begin{proposition}[\cite{KMPT21}]\label{T:concestlinear}
	Let $f\colon \N\to \U$ be a pretentious multiplicative function such that $f\sim \chi\cdot n^{it}$ for some Dirichlet character $\chi$ with period $q$ and $t\in \R$. Then
	\begin{enumerate}
		\item \label{I:cli}
	 For every $b\in \N$ we have
	\begin{equation}\label{E:fconc}
		\lim_{K\to\infty} 	\limsup_{N\to\infty} \max_{Q\in \Phi_K} \E_{n\in[N]}\big|f(Qn+b)-  %\epsilon_{a,f}\cdot 
		f(b)\cdot (Qb^{-1}n)^{it}\cdot \exp\big(F_N(f,K)\big)\big|=0,
	\end{equation}	where
 $F_N(f,K)$ and 	 $\Phi_K$  are as in \eqref{E:FNQ} and \eqref{E:PhiK}    respectively.
 %, and
%	 $\epsilon_{b,f}:=1$,  unless $b<0$ and %$\chi(q-1)=-1$, in which case $\epsilon_{b,f}=-1$.

\item  \label{I:clii}
If $t=0$  (i.e. $f\sim \chi$) and $Q_K,S_K$ are as in \eqref{E:defQKSK}, then
\begin{equation}\label{E:fconca}
	\lim_{K\to\infty} 	\limsup_{N\to\infty} \max_{b\in S_K} \E_{n\in[N]}\big|f(Q_Kn+b)-  f((Q_K,b))\cdot \chi(b/(Q_K,b))\cdot \exp\big(F_{N}(f,K)\big)\big|=0
\end{equation}
and
\begin{equation}\label{E:fconca'}
	\lim_{K\to\infty} 	\limsup_{N\to\infty} \max_{b\in S_K} \E_{n\in[N]}\big|f(Q_Kn+b)-  \E_{n\in[N]}\, f(Q_Kn+b)\big|=0.
\end{equation}
\end{enumerate}
\end{proposition}
%The result 	in part~\eqref{I:cli}  is proved in  %\cite{KMPT21} for  $b=1$; we explain next how to cover %more general $b\in \Z^*$. We  treat only the case $b<0$, %the case $b>0$ is similar (in fact, slightly easier).  %Clearly, it suffices to verify that  \eqref{E:fconc} %holds with $n-b$ instead of  $n$. Note that %$\lim_{n\to\infty}((n-b)^{it}-n^{it})=0$,
%	$f(Q(n-b)+b)= f(Q(n+|b|)-|b|)$, and
%	\begin{multline*}
%	|f(Q(n+|b|)-|b|)-f(|b|)\cdot \chi(q-1)\cdot %(Q|b|^{-1}n)^{it}\cdot \exp\big(F_N(f,K)\big)|=\\
%	|f(Q\, |b|^{-1}\,n+Q-1)-\chi(Q-1)\cdot %(Q|b|^{-1}n)^{it}\cdot \exp\big(F_N(f,K)\big)|,
%	\end{multline*}
%	where we used that $|f(|b|)|=1$, and  since $\chi$ has %period $q$ we have  $\chi(q-1)=\chi(Q-1)$ for all $Q\in %\Phi_K$ and $K$ large enough so that $q\mid Q$ and %$|b|\mid Q$.
%	Since $Q-1$ is positive and relatively prime to %$Q/|b|$, we deduce from  \cite[Lemma 2.5]{KMPT21}  that
%$$		\lim_{K\to\infty} 	\limsup_{N\to\infty} %\max_{Q\in \Phi_K} \E_{n\in[N]}\big|f(Q\, |b|^{-1}\, %n+Q-1)-  \chi(Q-1)\cdot  (Q|b|^{-1}n)^{it}\cdot %\exp\big(F_N(f,K)\big)\big|=0.
%$$	
%	Combining the previous facts, we get that
%	\eqref{E:fconc} holds for all negative $b$.

 %The result 	in part~\eqref{I:clii}
%is proved in \cite{KMPT21} with $c_K=d_K=1$. To get  the more general %statement, we use
 %that  $ \lim_{N\to\infty} |F_{cN}(f,K)-F_{N}(f,K)|=0$ holds  for all $c\in %(0,\infty)$, $K\in\N$, which follows from
 %the Cauchy-Schwarz inequality,  our assumption   $\D(f,\chi)<\infty$,
%and the bound $\sup_{N\in \N} \sum_{N\leq p\leq N^2}\frac{1}{p}<\infty$.

The argument of \cite[Lemma~2.5]{KMPT21} for $b=1$ applies after
factoring out $b$. Indeed, for sufficiently large $K$ and $Q\in\Phi_K$,
we have $b\mid Q$ and $\gcd(b,(Q/b)n+1)=1$, so
$f(Qn+b)=f(b)f((Q/b)n+1)$. The same concentration proof applies
to $Q/b$, whose small-prime exponents still tend to infinity;
its error tends to zero uniformly over $Q\in\Phi_K$.	
	It also gives the following identity that will be used later:
	 \begin{equation}\label{E:fconc'}
	 	\lim_{k\to\infty} 	\limsup_{N\to\infty} \E_{n\in[N]}\big|f(k!n+1)-  (k!n)^{it}\cdot \exp\big(F_N(f,k)\big)\big|=0.
	 \end{equation}	
% and for every $k\in
 %\Z, b\in \Z^*$ we have
 %\begin{equation}\label{E:fconcQ2}
% 	\lim_{K\to\infty} 	\limsup_{N\to\infty} %\max_{Q\in \Phi_K} \E_{n\in[N]}\big|f(Q^2n+kQ+b)-  %\epsilon_{a,f}\cdot f(|b|)\cdot %(Q^2|b|^{-1}n)^{it}\cdot %\exp\big(F_N(f,K)\big)\big|=0.
 %\end{equation}
%For the last identity, we use the case $b=1$ and %argue as before to cover
% the case $b\in \Z^*$.
\begin{corollary}\label{C:concentrationfg}
	If $f\colon \N\to \S^1$ is  a finitely generated pretentious multiplicative function, then
	\begin{equation}\label{E:fconcfingen}
	\lim_{K\to\infty} 	\limsup_{N\to\infty} \max_{Q\in \Phi_K} \E_{n\in[N]}\big|f(Qn+1)- 1\big|=0.
\end{equation}	
More generally, if $b\in \N$ and $f\sim \chi$ for some Dirichlet character $\chi$ with period $q$, then
	\begin{equation}\label{E:fconcfingenr}
	\lim_{K\to\infty} 	\limsup_{N\to\infty} \max_{Q\in \Phi_K} \E_{n\in[N]}\big|f(Qn+b)- %\epsilon_{b,f}\cdot 
	f(b)\big|=0.
\end{equation}	
%where $\epsilon_{b,f}:=1$, unless $b<0$ and %$\chi(q-1)=-1$, in which case $\epsilon_{b,f}:=-1$. %Moreover, \eqref{E:fconcfingenr}  still holds if  we %replace $Qn+b$ by  $Q^2n+kQ+b$  for any $k\in \Z$, $b\in %\Z^*$.
\end{corollary}
 \begin{proof}
 	  We use \cref{T:concestlinear} and the fact that if $f$ is pretentious and finitely generated, then $f\sim \chi$ for some Dirichlet character $\chi$ (see for example \cite[Lemma~B.3]{Cha24}), and $\lim_{K\to\infty}\sup_{N>K}|F_N(f,K)|=0$, which follows from \cite[Lemma~B.4]{Cha24}.
 	   Hence, $\lim_{K\to\infty} \sup_{N>K}\big|\exp\big(F_{N}(f,K)\big)-1\big|=0$.
 \end{proof}

\begin{corollary}\label{C:concentrationgeneral}
		Let $f\colon \N\to \S^1$ be a  pretentious multiplicative function and suppose that
		$f\sim \chi\cdot n^{it}$  for some Dirichlet character $\chi$  with period $q$  and $t\in \R$.
Let also 	$S_{\delta,N}$,  $F_N(f,K)$,  	$\Phi_K$  be as in \eqref{E:Sd}, \eqref{E:FNQ},   \eqref{E:PhiK},  respectively.
\begin{enumerate}

\item  \label{I:cconc4a}   For every  $b\in \N$ we have
\begin{equation}\label{E:fconcgendelta}
\lim_{\delta\to 0^+}	\limsup_{K\to\infty} 	\limsup_{N\to\infty} \max_{Q\in \Phi_K} \E_{n\in S_{\delta,N}}\big|f(Qn+b)-  %\epsilon_{b,f}\cdot 
f(b)\cdot (Qb^{-1}N)^{it}\cdot \exp\big(F_N(f,K)\big)\big|=0,
\end{equation}
%	where
%$\epsilon_{b,f}:=1$,  unless $b<0$ and $\chi(q-1)=-1$, in %which case $\epsilon_{b,f}=-1$, 
and
\begin{equation}\label{E:fconcgendeltaAverage}
	\lim_{\delta\to 0^+}	\limsup_{K\to\infty} 	\limsup_{N\to\infty} \max_{Q\in \Phi_K} \E_{n\in S_{\delta, N}}\big|f(Qn+b)-  \E_{n\in S_{\delta, N}} \, f(Qn+b)\big|=0.
\end{equation}

\item \label{I:cconc2}
We have
\begin{equation}\label{E:fconcabsol}
	\lim_{\delta\to 0^+}	\liminf_{k\to\infty} 	\liminf_{N\to\infty}|\E_{n\in S_{\delta, N}}\,  f(k!n+1)|=1
\end{equation}		
and
\begin{equation}\label{E:fconcabsol'}
	\liminf_{k\to\infty} 	\liminf_{N\to\infty}|\E_{n\in [N]}\,  f(k!n+1)|>0.
\end{equation}
\end{enumerate}
\end{corollary}
\begin{proof}	
	We prove  \eqref{I:cconc4a}. First note that for $n\in S_{\delta,N}$ we have $|\log(n/N)|\leq \delta$, and hence
	\begin{equation}\label{E:Sdtd}
		\sup_{n\in S_{\delta,N}}
		\left|\left(\frac{n}{N}\right)^{it}-1\right|
		\leq |t|\delta.
	\end{equation}
	We  prove \eqref{E:fconcgendelta}. Let
	$$
	e_{K,Q,N}(n):=
	\left|f(Qn+b)-%\epsilon_{b,f}
	f(b)
	(Qb^{-1}n)^{it}\exp(F_N(f,K))\right|.
	$$
	By \eqref{E:fconc} of \cref{T:concestlinear},
	$$
	\lim_{K\to\infty}\limsup_{N\to\infty}
	\max_{Q\in\Phi_K}
	\E_{n\in[N]}\, e_{K,Q,N}(n)=0.
	$$
	For fixed $\delta>0$ we have 
	$$
	\E_{n\in S_{\delta,N}}\, e_{K,Q,N}(n)
	\leq
	\frac{N}{|S_{\delta,N}|}
	\E_{n\in[N]}e_{K,Q,N}(n),
	$$
	and by \eqref{E:Sd},
	$
	N/|S_{\delta,N}|\to
	(1-e^{-\delta})^{-1}<\infty.
	$
	Thus
	$$
	\lim_{K\to\infty}\limsup_{N\to\infty}
	\max_{Q\in\Phi_K}
	\E_{n\in S_{\delta,N}}e_{K,Q,N}(n)=0.
	$$
		Let
	$$
	c_{K,Q,N,b}:=
	%\epsilon_{b,f}
	f(b)
	(Qb^{-1}N)^{it}\exp(F_N(f,K)).
	$$
	Since $\Re(F_N(f,K))\leq 0$, we have
	$|\exp(F_N(f,K))|\leq1$. Hence, using \eqref{E:Sdtd}, we have 
	$$
	|f(Qn+b)-c_{K,Q,N,b}|
	\leq
	e_{K,Q,N}(n)+
	\left|\left(\frac nN\right)^{it}-1\right|
	$$
	for $n\in S_{\delta,N}$. Therefore,
	$$
	\limsup_{K\to\infty}\limsup_{N\to\infty}
	\max_{Q\in\Phi_K}
	\E_{n\in S_{\delta,N}}
	|f(Qn+b)-c_{K,Q,N,b}|
	\leq |t|\delta .
	$$
	Letting $\delta\to0^+$ proves \eqref{E:fconcgendelta}.
	
	Finally, for every scalar $c$ and every non-empty finite set $S$, we have 
	$$
	\E_{n\in S}
	\left|u_n-\E_{r\in S}u_r\right|
	\leq 2\E_{n\in S}|u_n-c|.
	$$
	Taking $S:=S_{\delta,N}$, $u_n:=f(Qn+b)$, and
	$c=c_{K,Q,N,b}$ proves \eqref{E:fconcgendeltaAverage}.

	We prove \eqref{I:cconc2}. For the first identity we use
	\eqref{E:fconc'} in \cref{T:concestlinear} and make two
	observations. First, for every $k,N\in\N$, we have 
	$
	|\exp(F_N(f,k))|=\exp(\Re(F_N(f,k))),
	$
	and since $f\sim\chi\cdot n^{it}$, we have
	$\lim_{k\to\infty}\sup_{N\geq k}|\Re(F_N(f,k))|=0$. Hence,
	\begin{equation}\label{E:=1}
		\lim_{k\to\infty}\sup_{N\geq k}
		\big||\exp(F_N(f,k))|-1\big|=0.
	\end{equation}
	For fixed $\delta>0$, since
	$|S_{\delta,N}|/N\to 1-e^{-\delta}>0$, restricting the averages in
	\eqref{E:fconc'} to $S_{\delta,N}$ only costs a bounded factor. Thus,
	using also \eqref{E:=1}, we get
	$$
	\limsup_{k\to\infty}\limsup_{N\to\infty}
	\bigg|
	\Big|\E_{n\in S_{\delta,N}}f(k!n+1)\Big|
	-
	\Big|\E_{n\in S_{\delta,N}}(k!n)^{it}\Big|
	\bigg|=0.
	$$
	Next, for every $k,N\in\N$, we have 
	$$
	\Big|\E_{n\in S_{\delta,N}}(k!n)^{it}\Big|
	=
	\Big|\E_{n\in S_{\delta,N}}(n/N)^{it}\Big|,
	$$
	and by \eqref{E:Sdtd}, we have 
	$$
	\bigg|
	\Big|\E_{n\in S_{\delta,N}}(n/N)^{it}\Big|-1
	\bigg|
	\leq |t|\delta.
	$$
	Combining the last three identities and letting $\delta\to0^+$ gives
	\eqref{E:fconcabsol}.

Finally, we prove  \eqref{E:fconcabsol'}.
 It suffices to show that
 %\begin{equation}\label{E:fconcabsol''}
 	$$
 	\lim_{k\to\infty} 	\limsup_{N\to\infty}\big||\E_{n\in [N]}\,  f(k!n+1)| -(1+t^2)^{-\frac{1}{2}}\big|=0.
 $$
% \end{equation}
 Using \eqref{E:fconc'} of \cref{T:concestlinear} and    \eqref{E:=1}, it suffices to show that
 $$
 \lim_{N\to\infty}|\E_{n\in [N]}\, n^{it}|= (1+t^2)^{-\frac{1}{2}}.
 $$
 This follows immediately from the identity
 %% \begin{equation}\label{E:nit}
 	%%	\E_{n\in[N]} \, n^{i t}=N^{ it}/(1+i t) +t\cdot  o_{N\to \infty}(1).
 	%%\end{equation}
 	\begin{equation}\label{E:nit}
 			\lim_{N\to\infty}\big(	\E_{n\in[N]} \, n^{i t}-N^{ it}/(1+i t)\big)=0,
 \end{equation}
 	completing the proof.
\end{proof}

 \subsection{Examples of multiplicative actions}\label{SS:examples}
  We give some examples of multiplicative actions
   that appear frequently in this article.

 \subsubsection{Multiplicative rotations}
 Let $f\colon \N\to \S^1$ be a completely multiplicative function, $X$ be the closure of the range of
 $f$, which  is either $\S^1$ or a finite subgroup of $\S^1$, and $m_X$ be the Haar measure on $X$. For $n\in \N$  we consider the action defined by the maps  $T_n\colon X\to X$ as $T_nx:=f(n)x$ for $x\in X$, and   call it a {\em multiplicative rotation by $f$}.
 The action is finitely generated if and only if $f$ is finitely generated.
 For example, if $\lambda$ is the Liouville function (i.e., $\lambda(n)=(-1)^{\Omega(n)}$), then $X=\{-1,1\}$ and $m_X=(\delta_{-1}+\delta_1)/2$. If $f$ is as in examples~\eqref{I:Ex2} (for $t\neq 0$) or \eqref{I:Ex3} of \cref{SS:pretentious}, then $X=\S^1$ and $m_X=m_{\S^1}$.

 \subsubsection{Multiplicative dilations} Let $k\in \N$. On  $\T$ with the Borel $\sigma$-algebra and the Haar measure $m_\T$, we define  for $n\in\N$ the maps $T_n\colon \T\to \T$ by  $T_nx:=n^kx\pmod{1}$. We call the resulting multiplicative action {\em a dilation by $k$-th powers on  $\T$}. This action is infinitely generated. More generally, given a completely multiplicative function $\phi\colon \N\to \N$, we can define a multiplicative {\em action by dilations}  as follows
 $T_nx:=\phi(n)x\pmod{1}$, $n\in \N$, on $\T$ with $m_\T$. The resulting multiplicative action is finitely generated if and only if the set $\{\phi(p)\colon p\in \P\}$ is finite. One such example is given by  $T_nx=2^{\Omega(n)}x\pmod{1}$, $n\in\N$.
 %More examples of finitely  generated multiplicative actions were given in \cref{L:fgchar}

 \subsubsection{Aperiodic multiplicative actions}\label{SSS:aperiodic}
 We say that a multiplicative action $(X,\CX,\mu,T_n)$ is {\em aperiodic} if for all $a\in \N, b\in \Z_+$, and $F\in L^2(\mu)$ we have $\lim_{N\to\infty}\E_{n\in [N]}
 \, T_{an+b}F=\int F\, d\mu$ in $L^2(\mu)$.
   See~\cite[Corollary~1.17]{Cha24} for a characterization of aperiodicity in the case of finitely generated  multiplicative actions that uses the notion of pretentious rational eigenfunctions; we will not  use it  in this article though.

 It is shown in \cite[Corollary~1.23 (ii)]{Cha24} that a multiplicative rotation by a finitely generated multiplicative function $f$ is aperiodic if and only if $f^k$ is aperiodic for all $k\in \N$ with $ k<|X|$, where  $|X|$ may be infinite. In particular, the multiplicative rotation by the Liouville function $\lambda$ is aperiodic.
  It is also  shown in \cite[Corollary~1.24 (ii)]{Cha24} that if  $(X,\CX,\mu,T)$ is an ergodic measure preserving system, then   the action defined by  $T_nx:=T^{\Omega(n)}x$, $x\in X$, $n\in\N$,  is aperiodic. Taking $Tx:=2x\pmod{1}$ on $\T$ with $m_\T$, we get that the action by dilations  $T_nx=2^{\Omega(n)}x\pmod{1}$, $n\in\N$, is aperiodic. %Taking $T\colon \{-1,1\} \to \{-1,1\}$ be the map that exchanges $-1$ %and $1$, we get that the multiplicative rotation by $\lambda$ is %aperiodic.
Finally,  using  \eqref{E:spectralid}, Weyl's equidistribution result for polynomials,  and the bounded convergence theorem,
 	it is easy to show that for every $k\in \N$ actions by dilations by $k$-th powers are aperiodic.

\subsection{Spectral theory of multiplicative actions}
Recall that
  $\CM$  denotes the set of completely multiplicative functions with values on the unit circle.  Since completely multiplicative functions on the positive rationals are uniquely determined by their values on the positive integers,
we can identify $\CM$ with the Pontryagin dual of the (discrete) group of positive rational numbers under multiplication.

Let $(X,\CX,\mu,T_n)$ be a multiplicative action and $F\in L^2(\mu)$. Note that the map
$r/s\mapsto \int T_rF\cdot T_s\overline{F}\, d\mu$, $r,s\in \N$, from $(\Q^*_+,\times)$ to $\C$ is well defined and positive definite. Using a theorem of Bochner-Herglotz, we get that  there exists a finite Borel measure $\sigma_F$ on $\CM$ such that
$$
\int T_rF\cdot T_s\overline{F}\, d\mu=\int_\CM f(r)\cdot \overline{f(s)}\, d\sigma_F(f) \quad \text{for all } r,s\in\Q^*_+.
$$
We  easily  deduce from this the following identity that we will use frequently
\begin{equation}\label{E:spectralid}
\norm{\sum_{k=1}^l c_k\, T_{r_k}F}_{L^2(\mu)}=\norm{\sum_{k=1}^l c_kf(r_k)}_{L^2(\sigma_F(f))} \quad \text{for all } l\in \N, \, c_1,\ldots, c_l\in \C, \, r_1,\ldots, r_l\in \Q^*_+.
\end{equation}
Recall that  if $r=m/n$, $m, n\in \N$,  we have  $T_r=T_m\circ T^{-1}_n$
and $f(r)=f(m)/f(n)=f(m)\cdot \overline{f(n)}$, and also $T_1=\text{id}$, $f(1)=1$.
In particular, we have
%\begin{equation}\label{E:spectralid1}
%\norm{T_nF-T_mF}_{L^2(\mu)}=\norm{f(n)-f(m)}_{L^2(\sigma_F)}, \quad %m,n\in \N,
%\end{equation}
%and
\begin{equation}\label{E:spectralid2}
\norm{T_nF-F}_{L^2(\mu)}=\norm{f(n)-1}_{L^2(\sigma_F(f))}, \quad n\in \N.
\end{equation}

We will use the following basic facts about spectral measures.
If $\mu,\nu$ are two finite Borel measures  on $\CM$, we write
$\mu\bot \nu$ if they are mutually singular  and $\mu\ll \nu$ if $\mu$ is absolutely continuous with respect to $\nu$. If two functions $F,G\in L^2(\mu)$ are orthogonal, we write $F\bot G$.  Lastly, we say that the measure $\sigma$ on $\CM$ is supported on a Borel subset $A$ of $\CM$ if $\sigma (\CM\setminus A)=0$.  

\begin{lemma}\label{L:spectral}
	If $(X,\CX,\mu,T_n)$ is a multiplicative action, the following statements hold:
	\begin{enumerate}
		\item \label{I:spectral-1} $\sigma_{T_rF}=\sigma_F$, $\sigma_{cF}=|c|^2\sigma_F$, for all $r\in \Q^*_+, c\in \C$.
		
		\item \label{I:spectral0} If for every $n\in\N$ the spectral measures of $F_n\in L^2(\mu)$ are supported on a Borel set $A\subset \CM$, and $\lim_{n\to\infty}\norm{F_n-F}_{L^2(\mu)}=0$ for some $F\in L^2(\mu)$, then the spectral measure of $F$ is also supported on $A$.
		
		\item \label{I:spectral0.5}If $F,G\in L^2(\mu)$, then $\sigma_{F+G}\ll \sigma_F+\sigma_G$.
		
		\item \label{I:spectral1} If $F,G\in L^2(\mu)$ and  $\sigma_F\bot \sigma_G$, then $F\bot G$.
		
		\item  \label{I:spectral2} If $F\in L^2(\mu)$  and  $\sigma_F=\sigma_1+\sigma_2$ where $\sigma_1,\sigma_2$ are Borel measures such that  $\sigma_1\bot \sigma_2$, then there exist $F_1,F_2\in L^2(\mu)$ such that $F=F_1+F_2$ and  $\sigma_1=\sigma_{F_1}$, $\sigma_2=\sigma_{F_2}$.
	\end{enumerate}
\end{lemma}
Part~\eqref{I:spectral-1} is trivial.
Parts~\eqref{I:spectral0}-\eqref{I:spectral2} follow from
 Corollary~2.2, Corollary 2.1, Proposition~2.4, and Corollary~2.6 of \cite{Qu10}, respectively.
 All the arguments in \cite{Qu10} are given
in the case of $(\Z,+)$-actions, but the same arguments apply to $(\Q^*_+,\times)$-actions.

  Recall that  a multiplicative function $f\colon \N\to \S^1$ is  finitely generated if the set $\{f(p)\colon p\in \P\}$ is finite.
We will also   use the following fact from \cite[Lemma~2.6]{Cha24}:
\begin{proposition}[\cite{Cha24}]\label{L:finspectral}
	Let $(X,\CX,\mu,T_n)$ be a finitely generated multiplicative action. Then for every $F\in L^2(\mu)$ the spectral measure of $F$ is supported on the set  of finitely generated completely multiplicative functions with values on the unit circle.\footnote{See \cite[Lemma~7.1]{CMT24} for a proof  that this set is a Borel subset of $\CM$.}
	\end{proposition}

\subsection{Two convergence results}
We let
 \begin{equation}\label{E:I}
\CI:=\{F\in L^2(\mu)\colon T_nF=F \text{ for every } n\in\N\}.
\end{equation}
The mean ergodic theorem for multiplicative actions gives the following identity:
\begin{proposition}\label{L:multaverid}
	Let $(X,\CX,\mu,T_n)$ be a multiplicative action and $F\in L^2(\mu)$. Then for every multiplicative F\o lner  sequence $(\Phi_N)_{N\in\N}$ on $\N$ we have
	$$
	\lim_{N\to\infty}\E_{n\in\Phi_N} T_nF=\E(F|\CI)
	$$
	 in $L^2(\mu)$.
\end{proposition}
To prove the convergence properties in Theorems~\ref{T:MainA} and \ref{T:MainB}  we will use the following result, which  can be extracted from \cite[Theorem~1.14]{Cha24}:
\begin{proposition}[\cite{Cha24}] \label{P:convergelinear}
If  $(X,\CX,\mu,T_n)$ is a finitely generated action, then  for every   $F\in L^2(\mu)$ and $a\in \N$, $b\in \Z_+,$ the averages
$$
\E_{n\in[N]}\, T_{an+b}F
$$
converge in $L^2(\mu)$ as $N\to\infty$.
\end{proposition}
The  result  fails  for general multiplicative actions, such as multiplicative rotations by the pretentious multiplicative functions given in  examples \eqref{I:Ex2}  and \eqref{I:Ex3}  in \cref{SS:pretentious}.

\subsection{Two elementary estimates}
In the proofs of Theorems~\ref{T:MainA} and \ref{T:MainB}
we will use the following estimate from \cite[Lemma~1.6]{Chu11}:
\begin{proposition}[\cite{Chu11}]\label{T:Chu}
	Let $(X,\CX,\mu)$ be a probability space, $\CX_1,\ldots, \CX_\ell$ be sub-$\sigma$-algebras of $\CX$, and $F\in L^\infty(\mu)$ be non-negative.
	Then
	$$
	\int F\cdot \E(F|\CX_1)\cdots \E(F|\CX_\ell)\, d\mu \geq \Big(\int F\, d\mu\Big)^{\ell+1}.
	$$
\end{proposition}

We shall use the following easy to prove variant of \cite[Lemma~3.1]{FKM23}:
\begin{lemma}\label{L:lN}
	Let $V$ be a normed space, $v\colon \N\to V$ be a $1$-bounded sequence  and $l_1,l_2\in \Z_+$, not both of them $0$. Suppose that for some $\varepsilon>0$ and for some $1$-bounded sequence $(v_N)_ {N\in \N}$ of elements in $V$  we have
	$$
	\limsup_{N\to\infty} \E_{n\in [N]} \norm{v(n) -v_N}\leq \varepsilon.
	$$
	Then
	$$
	\limsup_{N\to\infty} \E_{m,n\in [N]} \norm{v(l_1m+l_2n)-v_{N}}\leq 4(l_1+l_2)\varepsilon.
	$$
%	where $l:=l_1+l_2$.
\end{lemma}
The estimate is given in \cite{FKM23} with $v_{lN}$ instead of $v_N$  and $2l\varepsilon $ instead of $4l\varepsilon$ in the conclusion where $l:=l_1+l_2$. But using our  assumption with $lN$
instead of $N$, we deduce that
$\limsup_{N\to\infty} \norm{v_{lN}-v_N}\leq  (l+1) \varepsilon$, which implies the asserted bound.
% (with $3l+1$ instead of $4l$).

\section{Decomposition results for multiplicative actions}
Our goal in this section is to prove the decomposition results of  Theorems~\ref{T:Decomposition} and \ref{T:DecompositionGeneral}. We do this with the exception of  the proof of property \eqref{I:Decomposition2} in Theorem~\ref{T:Decomposition}, which  will be completed in \cref{S:mixedseminorms}.  For a proof sketch see \cref{SS:proofsketchDecomposition}.

{\em In what follows, when we say that a measure is supported on a certain set, we mean that the measure of its complement is $0$.}

\subsection{Pretentious and aperiodic functions-Decomposition result}
Given a multiplicative action   we define two subspaces that will be used in  our study.
\begin{definition}\label{D:pa}
	Let $(X,\CX,\mu,T_n)$ be a multiplicative action. We let
	$$
	X_p:=\{F\in L^2(\mu)\colon \sigma_F \text{ is supported in } \CM_{p} \}
	$$
	and call any element of $X_p$ a {\em pretentious} function.
	We also let
		$$
	X_a:=\{F\in L^2(\mu)\colon \sigma_F \text{ is supported in } \CM_{a} \}
	$$
	and call any element of $X_{a}$ an {\em aperiodic } function.
\end{definition}
\begin{remark}
	In the case where the multiplicative action $(X,\CX,\mu,T_n)$ is finitely generated, Charamaras defined in \cite{Cha24}
	a subspace $X'_p$
	spanned by the ``pretentious rational eigenfunctions'' and a subspace $X_a'$ spanned by those functions  satisfying  \eqref{E:abzero}. It  turns out that   for finitely generated actions  we have
	$X_p=X'_p$ and $X_a=X'_a$. Indeed,  \cref{P:aperiodic0}
		implies that $X_a=X'_a$, and then \cref{P:Decomposition}
		 and  \cite[Theorem~1.28]{Cha24} imply that $X_p=X_p'$. However, this is no longer true for general multiplicative actions, see the remark after \cref{P:aperiodic0}.
	Moreover,
	 for our purposes, even for finitely generated actions,  it is more convenient   to use the defining properties of $X_p$ and $ X_a$  given above.
\end{remark}

From parts~\eqref{I:spectral-1}-\eqref{I:spectral0.5} of \cref{L:spectral}  we immediately deduce the following:
\begin{lemma}\label{L:Xpa}
	The spaces $X_p$ and $X_a$ are closed  $T_n$-invariant subspaces of $L^2(\mu)$.
\end{lemma}
It is less obvious that  $X_p\cap L^\infty(\mu)$ is an algebra. We show this later  in Corollary~\ref{C:factor}.

\begin{definition}
A multiplicative action  $(X,\CX,\mu,T_n)$ is {\em pretentious} if $X_p=L^2(\mu)$ and
{\em aperiodic} if $X_a=L^2_0(\mu)$,  where $L^2_0(\mu)$ denotes the zero-mean functions in $L^2(\mu)$.
\end{definition}
 \begin{remark}
 Part~\eqref{I:aper1}	of \cref{P:aperiodic0} below shows that this definition of aperiodicity  and the one in  \cref{SSS:aperiodic} are consistent.
\end{remark}

Looking at the examples in \cref{SS:examples},   a multiplicative rotation by a pretentious multiplicative function is a pretentious action, while a multiplicative rotation by the Liouville function  is an aperiodic action.
As explained in \cref{SSS:aperiodic}, other examples of aperiodic actions include   multiplicative rotations by finitely generated multiplicative functions $f$ such that $f^k$ is aperiodic for all  $k\in \N$ strictly smaller than the cardinality of the range of $f$,  and    actions by dilations by $k$-th powers.
On the other hand,  if we   consider the completely multiplicative function defined by $f(p)=-1$ for all primes $p\neq 2$ and $f(2)=i$,  then $f$ is aperiodic but $f^2$ is non-trivial and pretentious. A multiplicative rotation by such an $f$  acts on the space $X:=\{\pm 1,\pm i\}$, and this action exhibits  mixed behavior, in the sense that it is neither pretentious nor aperiodic.
% that $X_p\neq L^2(\mu)$ and  $X_a\neq L^2_0(\mu)$.
Indeed, if  $F(x):=x$, $x\in X$, then $F$  and $F^2$ have zero mean, and it can be shown that  $F\in X_a$ but $F^2\in X_p$. Therefore,   the zero-mean function  $G:=F+F^2$  is neither in $X_p$ nor in $X_a$, but it can be decomposed as the sum of two functions, one in $X_p$ and the other in $X_a$.

For general multiplicative actions we have the following decomposition result:
\begin{proposition}\label{P:Decomposition}
	Let  $(X,\CX,\mu,T_n)$ be  a multiplicative action and
	 $F\in L^2(\mu)$. Then  there  exist $F_p,F_a\in L^2(\mu)$ such that
	\begin{equation}\label{E:apdecomposition}
	F=F_p+F_a \quad \text{ and } \quad F_p\in X_p,\, \,  F_a\in X_a, \,\,    F_p\bot F_a.
	\end{equation}
Hence, $X_{a}=X_{p}^{\bot}$.
	\end{proposition}
\begin{remarks}
$\bullet$ For finitely generated actions, a similar result was   proved by Charamaras in~\cite[Theorem~1.28]{Cha24} for the spaces $X'_p$, $X'_a$ described in the remark after \cref{D:pa}.

$\bullet$ It  follows from \cref{C:factor} below that $X_p$ defines a factor $\CX_p$ and  $F_p=\E(F|\CX_p)$.
\end{remarks}
\begin{proof}
Let  $F\in L^2(\mu)$ and $\sigma_F$ be the spectral measure of $F$. Then
$$
\sigma_F=\sigma_{p}+\sigma_{a},
$$ where $\sigma_{p}$ and $\sigma_{a}$ are the restrictions of $\sigma_F$ on  the Borel subsets $\CM_{p}$ and $\CM_{a}$, respectively. Since the measures $\sigma_p$ and $\sigma_a$ are mutually singular, we get by  part~\eqref{I:spectral2} of \cref{L:spectral}  that there exist $F_p,F_a\in L^2(\mu)$ such that
$$
F=F_p+F_a \quad \text{and} \quad \sigma_{F_p}=\sigma_p, \,  \sigma_{F_a}=\sigma_a.
$$
 Then $F_p\in X_p$ and $F_a\in X_a$.
Furthermore, since $\sigma_{F_p}$ and $\sigma_{F_a}$ are mutually singular, we get by part~\eqref{I:spectral1} of \cref{L:spectral} that $F_{p}\bot F_{a}$. This completes the proof.
	\end{proof}

\subsection{Characterization of pretentious functions-Concentration property}
Next, we establish some approximate periodicity properties and related consequences for iterates of pretentious functions   that will be crucial in the proofs of our main results.
\begin{proposition}\label{P:Concentrationfg}
	Let   $(X,\CX,\mu,T_n)$  be a  finitely generated multiplicative action and $F\in L^2(\mu)$. 	
%\begin{enumerate}
	%\item \label{I:pre1}
	%If the action is finitely generated,
	 Then  $F\in X_{p}$
	if and only if
	\begin{equation}\label{E:concentration}
		\lim_{K\to\infty}\limsup_{N\to\infty} \max_{Q\in \Phi_K}\E_{n\in[N]}\norm{T_{Qn+1}F-F}_{L^2(\mu)}=0.
	\end{equation}
Moreover,  if $F\in X_p$ and  $Q_K,S_K$ are as in \eqref{E:defQKSK}, then
%%for all $c_K>0$, $K\in \N$,  we have
\begin{equation}\label{E:fconcfingen''}
	\lim_{K\to\infty} 	\limsup_{N\to\infty} \max_{b\in S_K} \E_{n\in[N]}\norm{T_{Q_Kn+b}F-A_{Q_K,b}(F)}_{L^2(\mu)}=0,
\end{equation}
where  for $Q\in \N$ and $b\in \N$ we let $A_{Q,b}(F):=\lim_{N\to\infty}\E_{n\in [N]}\,  T_{Qn+b}F$ and the limit exists in $L^2(\mu)$ by \cref{P:convergelinear}.
\end{proposition}
\begin{remark}
	Identity \eqref{E:concentration} also implies that
	\begin{equation}\label{E:concentrationa}
		\lim_{K\to\infty}\limsup_{N\to\infty} \max_{Q\in \Phi_K}\E_{n\in[N]}\norm{T_{Qn+b}F-T_bF}_{L^2(\mu)}=0
	\end{equation}
for every $b\in \N$ (but this is false for $b<0$).
This follows directly from \eqref{E:fconcfingenr}, the spectral identity,
and  Fatou's lemma. For every fixed $C\in\N$, (5.3) remains valid with $S_K$ replaced by
$\{b\in[CQ_K]:p^K\nmid b \text{ for every prime } p\le K\}$.
Indeed, write $b=r+jQ_K$ with $r\in S_K$ and $0\le j<C$.
Then $A_{Q_K,b}(F)=A_{Q_K,r}(F)$, and shifting $n$ by $j$
changes the averages by $O_C\big(\|F\|_{L^2(\mu)}/N\big)$.
	\end{remark}
\begin{proof}
	We prove the first part. Suppose that  $F\in X_p$, in which case $\sigma_F$ is supported on $\CM_p$.
	Using \eqref{E:spectralid2} we get
	$$
	\norm{T_{Qn+1}F-F}_{L^2(\mu)}=\norm{f(Qn+1)-1}_{L^2(\sigma_F(f))}
	$$	
	for every $Q,n\in \N$. Using \cref{L:finspectral} and applying
	 Fatou's lemma twice, it suffices to show that for every $f\in \CM_p$ that is finitely generated we have
	$$
	\lim_{K\to\infty}\limsup_{N\to\infty} \max_{Q\in \Phi_K}\E_{n\in[N]}|f(Qn+1)-1|=0.
	$$
This follows from the identity \eqref{E:fconcfingen} in   \cref{C:concentrationfg}.
	
	To prove the converse, suppose that $F\in L^2(\mu)$ satisfies \eqref{E:concentration}.
	Then for any $Q_K\in \Phi_K$, $K\in \N$,  we have
	$$
	\lim_{K\to\infty}\limsup_{N\to\infty} \norm{\E_{n\in  [N]}\, T_{Q_K n+1}F}_{L^2(\mu)}=\norm{F}_{L^2(\mu)}.
	$$
	Let $F=F_p+F_a$ be the decomposition of  \cref{P:Decomposition} with $F_p\in X_p$, $F_a\in X_a$, and $F_p\bot F_a$. Since $F_a\in X_a$, part~\eqref{I:aper1} of  \cref{P:aperiodic0} gives that
	for every $K\in \N$ we have
	$$
	\lim_{N\to\infty}\norm{\E_{n\in  [N]}\, T_{Q_K n+1}F_a}_{L^2(\mu)}=0.
	$$
	By the already established forward direction we get that $F_p$ satisfies \eqref{E:concentration}, hence
	$$
	\lim_{K\to\infty}\limsup_{N\to\infty} \norm{\E_{n\in  [N]}\, T_{Q_K n+1}F_p}_{L^2(\mu)}=\norm{F_p}_{L^2(\mu)}.
	$$
	Since $F=F_p+F_a$, from the previous three identities we deduce that
	$
	\norm{F_p+F_a}_{L^2(\mu)}=\norm{F_p}_{L^2(\mu)}.
	$
	Since $F_p\bot F_a$ this implies that $F_a=0$. Hence, $F=F_p\in X_p$.
	
To prove \eqref{E:fconcfingen''}, put
$B_{Q,b,N}(F):=\E_{r\in[N]}T_{Qr+b}F$.
By \eqref{E:spectralid},
\begin{equation}\label{E:TQKn}
	\big\|T_{Q_Kn+b}F-B_{Q_K,b,N}(F)\big\|_{L^2(\mu)}
	=
	\big\|f(Q_Kn+b)-\E_{r\in[N]}f(Q_Kr+b)\big\|_{L^2(\sigma_F)}.
\end{equation}
Since $F\in X_p$, \cref{L:finspectral} implies that $\sigma_F$
is supported on pretentious finitely generated functions, which satisfy
$f\sim\chi$ for some Dirichlet character $\chi$ by
\cite[Lemma~B.3]{Cha24}.
Thus \eqref{E:fconca'}, the Cauchy-Schwarz inequality, and Fatou's lemma give
\eqref{E:fconcfingen''} with $B_{Q_K,b,N}(F)$ in place of $A_{Q_K,b}(F)$.
For fixed $K$, the set $S_K$ is finite and
$B_{Q_K,b,N}(F)\to A_{Q_K,b}(F)$ in $L^2(\mu)$ uniformly over
$b\in S_K$, by \cref{P:convergelinear}.
The triangle inequality therefore gives \eqref{E:fconcfingen''}
as stated.
	\end{proof}
We next give a variant of the previous result for general multiplicative actions.
Part~\eqref{I:pre2} is used to prove part~\eqref{I:pre2'}, which in turn is used in the proof of  \cref{C:factor}, and
part~\eqref{I:pre3} is used to prove \cref{P:aperiodic0}.
\begin{proposition}\label{P:Concentrationgeneral}
		Let   $(X,\CX,\mu,T_n)$  be a  general  multiplicative action, $S_{\delta,N}$ be as in \eqref{E:Sd},  and $F\in L^2(\mu)$. 	
\begin{enumerate}
	\item \label{I:pre2}   We have  $F\in X_{p}$
	if and only if
	\begin{equation}\label{E:concegeneral}
		\lim_{\delta\to 0^+}\liminf_{k\to\infty}\liminf_{N\to\infty} \norm{\E_{n\in S_{\delta,N}}\, T_{k!n+1}F}_{L^2(\mu)}=\norm{F}_{L^2(\mu)}.
	\end{equation}

	 \item \label{I:pre3}  If $F\in X_p$ and $F\neq 0$, then
	\begin{equation}\label{E:concegeneral'}
		\liminf_{k\to\infty}\liminf_{N\to\infty} \norm{\E_{n\in  [N]}\, T_{k!n+1}F}_{L^2(\mu)}>0.
		\end{equation}
		
		\item \label{I:pre2'}   We have  $F\in X_{p}$
	if and only if
	\begin{equation}\label{E:concegeneral2}
		\lim_{\delta\to 0^+}\limsup_{k\to\infty}\limsup_{N\to\infty} \E_{n\in S_{\delta, N}}\, \norm{T_{k!n+1}F -A_{\delta,k,N}}_{L^2(\mu)}=0,
	\end{equation}
	 for some $A_{\delta,k,N}\in L^2(\mu)$
(in which case we can take $A_{\delta,k,N}= \E_{n\in S_{\delta, N}}\, T_{k!n+1}F$).

	\item \label{I:pre4}
If  $F\in X_{p}$,
 then for every  $b\in \N$ we have
	\begin{equation}\label{E:concegeneral2A}
\lim_{\delta\to 0^+}	\limsup_{K\to\infty}\limsup_{N\to\infty} \max_{Q\in \Phi_K}\E_{n\in S_{\delta, N}}\norm{T_{Qn+b}F-A_{\delta, Q,N,b}}_{L^2(\mu)}=0,
\end{equation}
 	where  for $Q,N\in \N$ and $b\in \N$ we let 	$A_{\delta, Q,N,b}:=\E_{n\in S_{\delta, N}}\, T_{Qn+b}F$.
\end{enumerate}
\end{proposition}
\begin{proof}
			We prove \eqref{I:pre2}. Suppose that $F\in X_p$. Using \eqref{E:spectralid} we get
	$$		
\norm{\E_{n\in S_{\delta, N}}\, T_{k!n+1}F}_{L^2(\mu)}
	=\norm{\E_{n\in S_{\delta, N}}\, f(k!n+1)}_{L^2(\sigma_F(f))}
	$$
	for every $k,N\in \N$ and $\delta>0$. Note also that
	$
	\norm{ F}_{L^2(\mu)}=
	\norm{1}_{L^2(\sigma_F(f))}
	$
	and $\sigma_F$ is supported in $\CM_p$.
	Using the previous facts,  Fatou's lemma three times and  identity \eqref{E:fconcabsol} of \cref{C:concentrationgeneral},  we get that the left side of \eqref{E:concegeneral} (with $\liminf_{\delta\to 0^+}$ in place of $\lim_{\delta\to 0^+}$) is at least
	$	\norm{1}_{L^2(\sigma_F(f))}=\norm{ F}_{L^2(\mu)}$. On the other hand,
$	\norm{\E_{n\in S_{\delta, N}}\, T_{k!n+1}F}_{L^2(\mu)}\leq\norm{F}_{L^2(\mu)}$ for every $k,N\in \N$, $\delta>0$. Combining these facts we get that
	 \eqref{E:concegeneral} holds.
	
		To prove the converse, suppose that $F\in L^2(\mu)$  satisfies \eqref{E:concegeneral}. Let $F=F_p+F_a$ be the decomposition given by \cref{P:Decomposition}. Since $F_a\in X_a$,  part~\eqref{I:aper2} of \cref{P:aperiodic0} below  implies that
	for every $k\in \N$ and $\delta>0$ we have
	$$
	\lim_{N\to\infty}
	\norm{\E_{n\in S_{\delta, N}}\, T_{k!n+1}F_a}_{L^2(\mu)}=0.
	$$
	Since $F_p\in X_p$ we get by the already established forward direction,  that $F_p$ satisfies \eqref{E:concegeneral}. We deduce that
	$
	\norm{F}_{L^2(\mu)}=\norm{F_p+F_a}_{L^2(\mu)}=\norm{F_p}_{L^2(\mu)}.
	$
	Since $F_p\bot F_a$, this implies that $F_a=0$. Hence, $F=F_p\in X_p$.
	
		We prove \eqref{I:pre3}.  Using \eqref{E:spectralid} and Fatou's lemma twice
	we get that
	$$
	\liminf_{k\to\infty}\liminf_{N\to\infty} \norm{\E_{n\in  [N]}\, T_{k!n+1}F}_{L^2(\mu)}\geq 	\norm{A(f)}_{L^2(\sigma_F(f))},
	$$
	where
	$$
	A(f):=\liminf_{k\to\infty} 	\liminf_{N\to\infty}|\E_{n\in [N]}\,  f(k!n+1)|.
	$$
	By \eqref{E:fconcabsol'} in \cref{C:concentrationgeneral} we get that
	$A(f)>0$ for every $f\in \CM_p$.
 Using this, and since  $\sigma_F$ is supported on $\CM_p$ and  $\sigma_F\neq 0$ (since $F\neq 0$),  we deduce that
	$\norm{A(f)}_{L^2(\sigma_F(f))}>0$.
	%This completes the proof.

		We prove \eqref{I:pre2'}. Suppose that $F\in X_p$. For $k,N\in \N$ and $\delta>0$, let $A_{\delta,k, N}:= \E_{n\in S_{\delta, N}}\, T_{k!n+1}F$. After expanding the square below, we find that
		$$
		\E_{n\in S_{\delta, N}}\, \norm{T_{k!n+1}F -A_{\delta,k,N}}_{L^2(\mu)}^2=
\norm{F}_{L^2(\mu)}^2 -
\norm{A_{\delta,k,N}}^2_{L^2(\mu)}.
		$$
	Using this and \eqref{E:concegeneral}, we deduce that \eqref{E:concegeneral2} holds.
	
	Conversely, suppose that  $F\in L^2(\mu)$  satisfies \eqref{E:concegeneral2}. Let $F=F_p+F_a$ be the decomposition given by \cref{P:Decomposition}. Since $F_a\in X_a$,  part~\eqref{I:aper2} of \cref{P:aperiodic0} below and \eqref{E:spectralid} imply  that
	for every $k\in \N$ and $\delta>0$ we have
	$$
	\lim_{N\to\infty}
	\norm{\E_{n\in S_{\delta, N}}\, T_{k!n+1}F_a}_{L^2(\mu)}=0.
	$$
	We deduce from this, \eqref{E:concegeneral2}, and the triangle inequality, that
	$$
		\lim_{\delta\to 0^+}\limsup_{k\to\infty}\limsup_{N\to\infty} \norm{\E_{n\in S_{\delta, N}}\,T_{k!n+1}F_p -A_{\delta, k,N}}_{L^2(\mu)}=0.
		$$
It follows that 	 \eqref{E:concegeneral2} holds with
$B_{\delta, k,N}:=\E_{n\in S_{\delta, N}}\,T_{k!n+1}F_p$ instead  of $A_{\delta, k,N}$, i.e.,
\begin{equation}\label{E:concegeneral2'}
\lim_{\delta\to 0^+}\limsup_{k\to\infty}\limsup_{N\to\infty} \E_{n\in S_{\delta, N}}\, \norm{T_{k!n+1}F -B_{\delta, k,N}}_{L^2(\mu)}^2=0.
\end{equation}
 Note that $B_{\delta, k,N}\in X_p$,  since $F_p\in X_p$ and $X_p$ is a $T_n$-invariant subspace. Hence,
$$
T_{k!n+1}F_a\bot \, T_{k!n+1}F_p- B_{\delta, k,N} \quad \text{for every } k,N\in \N, \, \delta>0.
$$
Since $F=F_p+F_a$, we deduce from this,  \eqref{E:concegeneral2'},  and the Pythagorean theorem, that
$$
	\lim_{\delta\to 0^+}\limsup_{k\to\infty}\limsup_{N\to\infty} \E_{n\in S_{\delta, N}}\, \norm{T_{k!n+1}F_p -B_{\delta, k,N}}^2_{L^2(\mu)}+\norm{F_a}^2_{L^2(\mu)}=0.
$$
Hence, $F_a=0$, which implies that $F\in X_p$.

	Finally, 	part~\eqref{I:pre4} follows by combining \eqref{E:spectralid} with  \eqref{E:fconcgendeltaAverage}
		%and \eqref{E:fconcabsol}
		in \cref{C:concentrationgeneral}, and applying Fatou's lemma multiple times.
\end{proof}	

\begin{corollary}\label{C:factor}
	Let $(X,\CX,\mu,T_n)$ be a multiplicative action. Then   $X_p\cap L^\infty(\mu)$ is a conjugation closed  subalgebra.
\end{corollary}
\begin{proof}
	We know from \cref{L:Xpa} that $X_p$ is a $T_n$-invariant subspace,
	and it easily follows from the definition of $X_p$ that it is conjugation
	closed.  Thus, it remains to show
	that $X_p\cap L^\infty(\mu)$ is closed under products. Let
	$F_1,F_2\in X_p\cap L^\infty(\mu)$.
	
	  For $j=1,2$ we use \eqref{E:concegeneral2} in \cref{P:Concentrationgeneral},
	 with $F_j$ instead of $F$, and some functions
	 $A_{j,\delta, k,N}\in L^\infty(\mu)$ instead of $A_{\delta, k,N}$.
	We also note that for $j=1,2$ we have
		$\norm{A_{j,\delta, k,N}}_{L^\infty(\mu)}\leq \norm{F_j}_{L^\infty(\mu)}$,
for every $\delta>0$ and $k,N\in \N$. Using these facts and the triangle inequality, 		
		 we get that
\eqref{E:concegeneral2} is satisfied with $F_1\cdot F_2$ instead of $F$ and
$A_{1,\delta, k,N}\cdot A_{2,\delta, k,N}$ instead of $A_{\delta, k,N}$.
	From the converse direction in part~\eqref{I:pre2'} of \cref{P:Concentrationgeneral}, it follows  that $F_1\cdot F_2\in X_p$. 
	This completes the proof.
	(For finitely generated actions, we could  instead use \eqref{E:concentration} in \cref{P:Concentrationfg}  to carry out this argument.)
\end{proof}

\subsection{Characterization of   aperiodic  functions-Vanishing property}
The next result gives a useful characterization of $X_a$ that works for general multiplicative actions.
\begin{proposition}\label{P:aperiodic0}
	Let   $(X,\CX,\mu,T_n)$ be  a multiplicative action and $F\in L^2(\mu)$.
	\begin{enumerate}
\item\label{I:aper1}	We have
	     $F\in X_{a}$ if and only if
	     	\begin{equation}\label{E:abzero}
	     	\lim_{N\to\infty}\norm{\E_{n\in  [N]}\,  T_{an+b}F}_{L^2(\mu)}=0
	     	\quad  \text{for every  } a\in \N,\,  b\in\Z_+.
	     \end{equation}
\item \label{I:aper2} If $F\in X_a$, then
	\begin{equation}\label{E:abzero'}
		\lim_{N\to\infty}\norm{\E_{n\in S_{\delta,N} }\,  T_{an+b}F}_{L^2(\mu)}=0
		\quad  \text{for every  } \delta>0, \,  a\in \N,\,  b\in\Z_+,
	\end{equation}
where $S_{\delta, N}$ is as in \eqref{E:Sd}.
\end{enumerate}
\end{proposition}
\begin{remark}
	 If we relax the mean convergence condition in \eqref{E:abzero}  to weak convergence, i.e.,
$$
\lim_{N\to\infty}\E_{n\in [N]} \int T_{an+b}F\cdot G\, d\mu=0 \quad \text{for every  } a\in \N, \, b\in \Z_+, \, G\in L^2(\mu),
$$
then  the converse direction no longer holds.  Let us see
how to prove  this when $G:=\overline{F}$.  Consider a multiplicative action $(X,\CX,\mu,T_n)$ and $F\in L^2(\mu)$ such that
$$
\int T_nF\cdot T_m\overline{F}\, d\mu=\int_0^1 n^{2\pi it}\cdot m^{-2\pi it}\, dt \quad \text{for all } m,n\in \N.
$$
(For example, on $\T^2$  with $m_{\T^2}$, let   $T_n(t,x):=(t,x+\{t\}\log{n})$ and $F(t,x):=e(x)$.)
Then,  $\sigma_F$ is supported on the set $\{(n^{it})\colon t\in [0,2\pi]\}$, hence   $F\in X_p$. However,
   for every $a\in \N$, $b\in \Z_+$, we have 
$$
\lim_{N\to\infty}\E_{n\in [N]} \int T_{an+b}F\cdot \overline{F}\, d\mu=
%\lim_{N\to\infty}\int_0^1 \E_{n\in[N]} \, (an)^{2\pi i t}\, dt=
\lim_{N\to\infty} \int_0^1 (aN)^{2\pi it}\cdot (1+2\pi i t)^{-1}\, dt=0,
$$
where the first identity follows by first using \eqref{E:spectralid} and then \eqref{E:nit}, and the second  from a variant of the Riemann-Lebesgue lemma.

The same example can be used to show that   if $X'_p$ is the closed subspace  spanned by the pretentious eigenfunctions with eigenvalues pretentious multiplicative functions (defined as in \cite{Cha24} but for general actions)  and $F$ is as before, then $F\bot X'_p$ (although $F\in X_p$, hence $X'_p\neq X_p$) and \eqref{E:abzero} fails even when $a=1$.

Lastly, using a multiplicative rotation by $n^i$, it can  be seen that
the converse implication in \eqref{I:aper1}
fails if we use logarithmic averages.
\end{remark}
\begin{proof}
	We prove \eqref{I:aper1}.
	Suppose that $F\in X_a$. Then the spectral measure of $F$ is supported on $\CM_a$.
Let $a\in \N$, $b\in \Z_+$. Using	\eqref{E:spectralid} we get
	$$
\norm{\E_{n\in  [N]}\,  T_{an+b}F}_{L^2(\mu)}=\norm{\E_{n\in[N]}\, f(an+b)}_{L^2(\sigma_F(f))}
		$$
		for every $N\in \N$.
		Since $\sigma_F$ is supported on $\CM_{a}$ and for $f\in \CM_{a}$ we have
		$$
		\lim_{N\to\infty} \E_{n\in[N]}\, f(an+b)=0,
		$$
		the needed vanishing property follows from the bounded convergence theorem.
	
	For the converse direction, it suffices to show that if \eqref{E:abzero} holds for $F$ and $F=F_p+F_a$ is the decomposition given by \cref{P:Decomposition}, then $F_p=0$.
	Note first  that since \eqref{E:abzero} holds for $F$ and $F_a$, we get that it also holds for $F_p$, so for every $k\in \N$ we have
	$$
	\lim_{N\to\infty}  \norm{\E_{n\in [N]}\, T_{k!n+1}F_p}_{L^2(\mu)}=0.
	$$
	Using part~\eqref{I:pre3} of \cref{P:Concentrationgeneral}, we deduce that $F_p=0$, completing the proof.

We prove \eqref{I:aper2}.  Let $F\in X_a$,  $\delta>0$, $a\in\N$,
$b\in\Z_+$, and 
$
M_N:=[ e^{-\delta}N].
$
For all sufficiently large $N$ we have
$$
\E_{n\in S_{\delta,N}} T_{an+b}F
=
\frac{N}{N-M_N}\E_{n\in[N]}T_{an+b}F
-
\frac{M_N}{N-M_N}\E_{n\in[M_N]}T_{an+b}F.
$$
Both averages on the right hand side converge to $0$ in $L^2(\mu)$ by the
forward implication in part~\eqref{I:aper1}, while the two coefficients
remain bounded for fixed $\delta$. This proves \eqref{E:abzero'}.
		\end{proof}
	We next  state another vanishing  property of aperiodic functions, in a stronger form  than the one needed to prove part~\eqref{I:DecompositionGeneral2} in  \cref{T:DecompositionGeneral}; this stronger form
	 will be used later in the proof of \cref{P:Characteristic2}.
	\begin{proposition}\label{P:aperiodicP0}
		Let  $(X,\CX,\mu,T_n)$ be a multiplicative action and $F\in X_a$. Suppose also that
$R_1,R_2$ are  rational polynomials that factor linearly and  $R_1$ is not of the form  $c\, R^r$ for any $c\in \Q_+$,  rational polynomial $R$, and    $r\geq 2$. Let also  $K_N$, $N\in \N$,  be convex subsets of $\R^2_+$,  and $\alpha,\beta,t \in \R$. Then
		\begin{equation}\label{E:R1R2zero}
			\lim_{N\to\infty}\norm{\E_{m,n\in[N]}\,  {\bf 1}_{K_N}(m,n)\cdot e(m\alpha+n\beta)\cdot (R_2(m,n))^{it} \cdot T_{R_1(m,n)}F}_{L^2(\mu)}=0.
		\end{equation}
	\end{proposition}
\begin{remark}
	For $t=0$, \eqref{E:R1R2zero} also holds for
	$R_1=c\prod_j L_j^{k_j}$, where $c\in\Q\setminus\{0\}$,
	the $L_j$ are pairwise independent integer-coefficient linear forms,
	and $k_j\in\Z$ satisfy $\gcd_j|k_j|=1$, on convex averaging sets
	where $R_1$ is defined and positive. The conclusion also holds on
	fixed affine grids. These extensions follow from the same spectral
	argument, using the even extension of multiplicative functions in
	\cite[Lemma~9.6]{FH17} and expanding congruence indicators into
	additive characters.
\end{remark}
\begin{proof}
	Suppose first that $t=0$.
	Using \eqref{E:spectralid} and the fact that the spectral measure $\sigma_F$ of $F$ is supported on $\CM_a$, it suffices to show that for every aperiodic multiplicative function $f\colon \N\to \S^1$ we have
	\begin{equation}\label{E:fPzero}
	\lim_{N\to\infty} \E_{m,n\in[N]}\,
	  {\bf 1}_{K_N}(m,n)\cdot e(m\alpha+n\beta)\cdot f(R_1(m,n))=0.
	\end{equation}
	We have
	$
R_1(m,n)=c\, 	\prod_{j=1}^sL^{k_j}_j(m,n)
$
for some  $c\in \Q_+, s\in \N$, non-zero $k_1,\ldots, k_s\in \Z$, and pairwise independent linear forms $L_1,\ldots, L_s$ with non-negative coefficients.

We first  claim that not all the multiplicative functions  $f^{k_1},\ldots, f^{k_s}$ are  pretentious. Indeed, if this was the case, then  \eqref{E:triangle} gives that
$f^d$ is pretentious, where $d:=\gcd(k_1,\ldots, k_s)$.
Then  $R_1= c\,R^d$,  where $R:= \prod_{j=1}^sL^{k_j/d}_j$ is a rational polynomial, and  our assumption about  $R_1$ gives that $d=1$. Hence, $f$ is pretentious, a contradiction.

We deduce that   there exists $j_0\in [s]$ such that $f^{k_{j_0}}$ is aperiodic. The required vanishing property \eqref{E:fPzero} then follows from  \cite[Theorem~2.5]{FH17} and  \cite[Lemma~9.6]{FH17}
(the latter does not include the exponential term $e(m\alpha+n\beta)$, but the same argument applies.)

Finally, we consider the case where $t$ can  be non-zero. After carrying out the previous deductions,  it suffices to show that if $f_1,\ldots, f_\ell\colon \N\to \S^1$ are completely multiplicative functions such that $f_1$ is aperiodic, $L_1,\ldots, L_\ell$ are  linear forms such that $L_1,L_j$ are independent for $j=2,\ldots, \ell$,  and $R$ is a rational polynomial that factors linearly, then
	\begin{equation}\label{E:fPzero'}
	\lim_{N\to\infty} \E_{m,n\in[N]}\,
	{\bf 1}_{K_N}(m,n)\cdot e(m\alpha+n\beta)\cdot (R(m,n))^{it} \prod_{j=1}^\ell f_j(L_j(m,n))=0.
\end{equation}
To see this, first  suppose  that $R= c'\, L^k$ for some linear form $L$ and $c'\in \Q_+, k\in \Z$. If  $L$ is a rational multiple of $L_1$, then after replacing $f_1$ by
$f_1\cdot n^{ikt}$, which is aperiodic by \cref{T:DH},  we get that our assumptions are still satisfied and the term $(L(m,n))^{ikt}$ is absorbed in $f_1$. We can therefore assume that  $L_1,L$ are independent, so  we can work with the $\ell+1$ multiplicative functions,
$f_1,\ldots, f_\ell, f_{\ell+1}$, where $f_{\ell+1}(n):=n^{ikt}$, which  are evaluated at the linear forms $L_1,\ldots, L_\ell, L$.
Since  $L_1$ is independent of the other $\ell$ linear forms we get the necessary vanishing property.  In general, $R=c'\, R_1^{k_1}\cdots R_s^{k_s}$ where $R_1,\ldots, R_s$ are pairwise independent linear forms and $c'\in \Q_+,k_1,\ldots, k_s\in \Z$. At most one of the linear forms  $R_1,\ldots, R_s$ can   be a rational multiple of $L_1$, in which case it can be absorbed by $f_1$ as before, and the other forms can be handled by extending the product using additional multiplicative functions, again, exactly as before.
\end{proof}
	
	\subsection{Proof of   \cref{T:DecompositionGeneral}}\label{S:ProofThmB}
	Let $F\in L^\infty(\mu)$. Using \cref{P:Decomposition} we get that there exist $F_p,F_a\in L^2(\mu)$ such that $F_p\in X_p, F_a\in X_a$ and  $F=F_p+F_a$. We also get that $F_p$ is the orthogonal projection of $F$ onto $X_p$. 
	Moreover, by \cref{C:factor} the space $X_p\cap L^\infty(\mu)$ is a conjugation closed subalgebra, so  it defines a factor $\CX_p$ and $F_p=\E(F|\CX_p)$. In particular we have that $F_p\in L^\infty(\mu)$, hence $F_a\in L^\infty(\mu)$.
	Properties~\eqref{I:DecompositionGeneral1} and
	\eqref{I:DecompositionGeneral2} follow from
	\eqref{E:concegeneral2A} in \cref{P:Concentrationgeneral}, applied to
	$F_p$, and \eqref{E:R1R2zero} in \cref{P:aperiodicP0}, respectively,
	completing the proof.

	\section{Mixed seminorms and related inverse theorem}\label{S:mixedseminorms}
	 Our goal in this section is to define the mixed seminorms, show that for finitely generated multiplicative actions they vanish on the subspace $X_a$, and use this fact to
complete the  proof of   Theorem~\ref{T:Decomposition}. 	
For a proof sketch  of the vanishing property see \cref{SS:proofsketchDecomposition}.
	\subsection{Mixed seminorms for multiplicative actions}\label{SS:mixeddefinition}
	We begin by recalling the definition of the Gowers norms of a
	finite sequence on $\Z_N=\Z/(N\Z)$.
	 \begin{definition}[Gowers norms on $\Z_N$~\cite{G01}]
	 	Let $N\in \N$  and $a\colon \Z_N\to \C$. For $s\in \N$ the \emph{Gowers norm} $\norm a_{U^s(\Z_N)}$  is defined inductively as follows:    We let
	 	$$
	 	\norm a_{U^1(\Z_N)}:=|\E_{n\in\Z_N}\, a(n)|,
	 	$$
	 	and for every $s\in \N$
	 	\begin{equation}
	 		\label{eq:def-gowers}
	 		\norm a_{U^{s+1}(\Z_N)}^{2^{s+1}}:=\E_{h\in\Z_N}\norm{a\cdot \overline a_h}_{U^s(\Z_N)}^{2^s},
	 	\end{equation}
	where  $a_h(n):=a(n+h)$ for $h,n\in \Z_N$.
	 \end{definition}
	 For example, we have
	 \begin{equation}\label{E:U2}
	 \norm a_{U^2(\Z_N)}^4=\E_{n,h_1,h_2\in\Z_N}\,  a(n)\, \overline{a}(n+h_1)\, \overline{a}(n+h_2)\, a(n+h_1+h_2),
	 \end{equation}
	 and for  $s\geq 3$ we get a  similar  formula with $2^s$ terms in the product.

	 \begin{definition}\label{D:mixedsem}
	 	Given a multiplicative action $(X,\CX,\mu,T_n)$ and $F\in L^\infty(\mu)$, for $s\in \N$  we define the {\em mixed seminorm of $F$  of order $s$} as
	 	\begin{equation}\label{E:defmixed}
	 	\nnorm{F}^{2^s}_{U^s}:=\limsup_{N\to \infty} \int \norm{F(T_nx)}^{2^s}_{U^s(\Z_N)}\, d\mu,
	 	\end{equation}
	 where for $N\in \N$,
	 % we denote by $\tilde{N}$ the smallest prime greater %than $N$,
	$x\in X$, we assume that  $(F(T_nx))_{n\in[N]}$ is extended periodically on $\Z_N$.
	 \end{definition}
	 \begin{remarks}
	$\bullet$
	  We can also define the closely related seminorms
	 	\begin{equation}\label{E:Us*}
	 		\nnorm{F}^{2^s}_{*,U^s}:=\limsup_{N\to \infty} \int \norm{F(T_nx)}^{2^s}_{U^s[N]}\, d\mu,
	 	\end{equation}
 	where $\norm{\cdot }_{U^s[N]}$ are as in \cite[Section~2.1.1]{FH17}.
	 	It follows from the proof of  \cite[Lemma~A.4]{FH17} that there exists $K=K(s)\in \N$  such that for every sequence  $a\colon \N\to \U$ we have
	 	$$
	  \norm{a}_{U^s(\Z_N)}^{K}\ll_s	\norm{a}_{U^s[N]} +o_N(1)\ll_s \norm{a}_{U^s(\Z_N)}^{\frac{1}{K}},
	 	$$
	 	where $o_N(1)$ is a quantity that does not depend on $a(n)$ and converges to $0$ as $N\to \infty$.
	 %	We also have a similar estimate with the roles of
	 %	$U^s[N]$ and $U^s(\Z_N)$ reversed.
	 As a result, we have
	 $		\nnorm{F}_{U^s}=0  \Leftrightarrow  \nnorm{F}_{*,U^s}=0$.

$\bullet$ If in \eqref{E:defmixed} we place the  limsup inside the integral, then we get  seminorms that are too strong for our purposes. For example,  $F\in X_a$ would not imply that $\nnorm{F}_{U^1}=0$,  since this would  imply that $\lim_{N\to\infty}\E_{n\in[N]} F(T_nx)=0$ for almost every  $x\in X$, which is known to be false   for some  finitely generated multiplicative actions \cite[Theorem~1.2]{Loy23}.
	 \end{remarks}	
	  Since  the Gowers norms $\norm{\cdot}_{U^{s}(\Z_N)}$ satisfy the triangle inequality and are increasing in $s\in \N$,
	 %the estimate $\norm{\cdot}_{U^s[N]}\leq c_s\cdot %\norm{\cdot}_{U^{s+1}[N]}$
	 we easily deduce that the mixed seminorms also enjoy similar properties, i.e.,
	 \begin{equation}\label{E:increase}
	 	\nnorm{F+G}_{U^s}\leq \nnorm{F}_{U^s}+\nnorm{G}_{U^s} \quad \text{and} \quad 	\nnorm{F}_{U^s}\leq   \nnorm{F}_{U^{s+1}}
	 \end{equation}
	 for all $s\in \N$ and $F,G\in L^\infty(\mu)$.
	
	 To get a sense of how the mixed seminorms look like, we  note  that \eqref{E:defmixed} gives
	 \begin{equation}\label{E:GowersU2}
	\nnorm{F}^4_{U^2}= \limsup_{N\to \infty} \Big(\E_{n,h_1,h_2\in \Z_N} \int T_nF\cdot  T_{n+h_1}\overline{F}\cdot T_{n+h_2}\overline{F}\cdot T_{n+h_1+h_2}F\, d\mu\Big).
	 \end{equation}
  In \eqref{E:GowersU2}, every subscript is represented in $[N]$ modulo $N$.
 For an ergodic multiplicative action, its integral can equivalently be
 written as the $L^2$ limit of multiplicative F{\o}lner averages of the
 same product. This combines addition in the cube with multiplication
 in the iterates and motivates the term ``mixed seminorms''.
%It turns out that if  we replace the average over %$\Z_N$ with an average over $[N]$ we end up with %equivalent seminorms.
%If we do this, then
%	Note that for ergodic multiplicative actions %$(X,\CX,\mu,T_n)$, for almost every $x\in X$, the %right side in \eqref{E:U2}  is equal to
%	   $$
%	 \limsup_{N\to \infty} \big(\E_{n,h_1,h_2\in \Z_N} %\, \E_{k\in \Phi}\,     T_{kn}F\cdot T_{k(n+h_1)} %\overline{F}\cdot T_{k(n+h_2)} \overline{F}\cdot %T_{k(n+h_1+h_2)}F\big),
%	 $$
%	  where $\Phi=(\Phi_N)_{N\in\N}$ is a %multiplicative  F\o lner  sequence and $\E_{k\in %\Phi}:=\lim_{N\to\infty}\E_{k\in\Phi_N}$   where  %convergence is taken in $L^2(\mu)$.
 %	 This last expression uses a mixture of addition %and multiplication on the iterates of the action  as %well as our averaging schemes, which motivates our %terminology of ``mixed seminorms''.
 	
 Lastly, we note  that the  seminorms  $\nnorm{\cdot }_{U^s}$  differ sharply from   the  Host-Kra seminorms of order $s$ in
 	the multiplicative setting, which we denote here by $\nnorm{\cdot}_{s,\times}$.
 	 This can already be seen when $s=2$. For an ergodic multiplicative action $(X,\CX,\mu,T_n)$  we have
 	  \begin{equation}\label{E:HK}
 	 \nnorm{F}^4_{2,\times}= \lim_{N\to \infty} \E_{h_1,h_2\in \Phi_N} \int F\cdot  T_{h_1}\overline{F}\cdot T_{ h_2}\overline{F}\cdot T_{ h_1\cdot h_2}F\, d\mu,
 	 \end{equation}
 	 where $(\Phi_N)_{N\in \N}$ is a multiplicative  F\o lner sequence in $\N$ and the limit can be shown to exist. These seminorms are quite different from  the mixed seminorms,  satisfy substantially different inverse theorems, and do not play a role in our study here.

	\subsection{Inverse theorem for sequences with multiplicative structure}
In \cite[Theorem~2.5]{FH17} it is proved that a  bounded multiplicative function  is Gowers uniform if and only if it is aperiodic.  Our next goal is to adapt the proof of this  result to cover a wider variety of sequences with weaker multiplicative structure, such as sequences of the form $(F(T_nx))$ when the multiplicative action $T_n$ is finitely generated and $F\in L^\infty(\mu)$.
This is crucial for our subsequent proof of the inverse theorem for mixed seminorms.

	\begin{proposition}\label{T:FH}
	Let $N_k\to \infty$ and   $P_0$ be a subset of the primes so that $\sum_{p\in P_0} \frac{1}{p}=\infty$. Suppose that the sequence
		 $a\colon \N\to \U$   satisfies
		 \begin{equation}\label{E:apqn}
		  a(pn)=c_{p,p'} \,   a(p'n) \quad \text{ for all } p,p'\in P_0, \, n\in \N,
		 \end{equation}
		for some $ c_{p,p'}\in \U$.  
		 Then the following properties are equivalent:
		 \begin{enumerate}
		 	\item \label{I:FH1} $\lim_{k\to\infty}\E_{n\in[\alpha N_k]} \,  a(qn+r)=0$ for every  $q\in \N,r\in \Z_+$,
		 	$\alpha\in \Q\cap (0,1/q]$;\footnote{The rational truncations allow us to remove the slowly varying phase in the proof.}
		 	
		 	\item  \label{I:FH2}  $\lim_{k\to\infty} \norm{a}_{U^2(\Z_{N_k})}=0$;
		 	
		 	\item \label{I:FH3} $\lim_{k\to\infty} \norm{a}_{U^s(\Z_{N_k})}=0$ for every  $s\in \N$.
		 \end{enumerate}
	\end{proposition}
%	\begin{remark}
%	For general subsequences $(N_k)$,  	if  we only know that  %$\eqref{I:FH1}$ holds for $\alpha=1$, we cannot show that    %$\eqref{I:FH2}$ holds.
	%	\end{remark}
	The bulk of the proof of \cref{T:FH} is already contained in the proof of   \cite[Theorem~2.5]{FH17},  and we do not plan to repeat it, but we will explain some small changes we need to make next. We  will use the following variant of the
	 Daboussi-K\'atai orthogonality criterion~\cite{DD82,K86}  that  can  be proved  as in~\cite{K86};  for completeness we sketch the argument. 
	 
	 	\begin{lemma}\label{L:Kataigeneral}
	 	Let $N_k\to \infty$ and   $w_k\colon \N\to \U$, $k\in \N$,  be   sequences that satisfy
	 	$$
	 	\lim_{k\to\infty} \E_{n\in [N_k/p]}\,   w_k(pn)\cdot \overline{w_k(p'n)} =0
	 	$$
	 	for all distinct $p,p'\in P_0$ with $p'<p$, where   $P_0$  is a subset of the primes that satisfies
	 	$\sum_{p\in P_0}\frac{1}{p}=\infty$.
	 	Let also $a\colon \N\to \U$ be  such that \eqref{E:apqn} holds 	for some $ c_{p,p'}\in \U$. Then
	 	$$
	 	\lim_{k\to\infty}\E_{n\in[N_k]}\, w_k(n)\cdot a(n)=0.
	 	$$
	 \end{lemma}
	 \begin{proof}
Let $\CP\subset P_0$ be a finite non-empty set  and 
$$
H:=\sum_{p\in\CP}\frac1p,
\qquad
\omega(n):=\sum_{p\in\CP}\one_{p\mid n}.
$$
Fix $p_*\in P_0$, and let
$b(m):=a(p_*m)$ and $\gamma_p:=c_{p,p_*}$. By \eqref{E:apqn} we have 
$a(pm)=\gamma_p b(m)$ for all $p\in\CP$ and $m\in\N$.
Let 
$$
S_k:=\sum_{n\leq N_k}a(n)w_k(n),
\qquad
W_k(m):=\sum_{\substack{p\in\CP\\pm\leq N_k}}\gamma_p w_k(pm).
$$
Since
$$
\E_{n\in[N_k]}|\omega(n)-H|^2
=
\sum_{p\in\CP}\frac{1}{p}\left(1-\frac1p\right)+o(1)
\leq H+o(1),
$$
using the Cauchy-Schwarz  inequality we get 
$$
\Big|HS_k-\sum_{m\leq N_k}b(m)W_k(m)\Big|
=
\Big|\sum_{n\leq N_k}(H-\omega(n))a(n)w_k(n)\Big|
\leq N_k\sqrt{H}+o(N_k).
$$
If we expand $\sum_m|W_k(m)|^2$, the diagonal terms contribute at
most $N_kH$. Each off-diagonal term is $o(N_k)$ by the hypothesis,
since its summation range is
$m\leq N_k/\max(p,p')$. Therefore,
$$
\frac{1}{N_k}\sum_{m\leq N_k}|W_k(m)|^2\leq H+o(1).
$$
As $|b|\leq1$, another application of the  Cauchy-Schwarz inequality gives 
$$
\limsup_{k\to\infty}\frac{|S_k|}{N_k}\leq\frac2{\sqrt H}.
$$
All limits above are taken with $\CP$ fixed. Letting $\CP$ increase
so that $H\to\infty$ completes the proof.	 	
	 \end{proof}
	 
	\begin{proof}[Proof of \cref{T:FH}]
		The implication	 $\eqref{I:FH2} \implies \eqref{I:FH1}$ can be shown as in \cite[Lemma~A.6]{FH17}.
		
	The implication	 $\eqref{I:FH2} \implies \eqref{I:FH3}$ is the most difficult and can be
	proved by repeating the argument used to establish \cite[Theorem~2.5]{FH17}  without essential changes. We briefly sketch the beginning of the argument.
	  Using the  inverse theorem of Green-Tao-Ziegler~\cite{GTZ12}
	  	 for the Gowers norms $U^s(\Z_N)$, we deduce  that if a sequence $a\colon \N\to \U$ satisfies
	$\limsup_{k\to \infty} \norm{a}_{U^s(\Z_{N_k})}>0$, then there exist  a subsequence $(N'_k)$ of $(N_k)$, $c>0$, an $(s-1)$-step nilmanifold $X=G/\Gamma$,  $F\in C(X)$,  and $b_k\in G$, $k\in \N$,  such that
	\begin{equation}\label{E:aFbk1}
		|\E_{n\in[N'_k]} \, a(n)\cdot F(b_k^n\Gamma)|>c \quad \text{for every } k\in \N.
	\end{equation}
	We can also assume that $F$ is a  so-called non-trivial vertical nil-character; for example  if $X=\T$,  then $F(x)=e(kx)$ for some non-zero $k\in \Z$.
	Using   \cref{L:Kataigeneral} (we make essential use of   \eqref{E:apqn} here), we deduce that there
	are $p_0,q_0\in P_0$ with $q_0<p_0$  such that
	\begin{equation}\label{E:aFbk2}
		\limsup_{k\to\infty} |\E_{n\in[N'_k/p_0]} \, F(b_k^{p_0n}\Gamma)\cdot \overline{F}(b_k^{q_0n}\Gamma)|>0.
	\end{equation}
	From this point on, following
	 the (rather intricate) argument used to prove  \cite[Theorem~2.5]{FH17},  we deduce from  \eqref{E:aFbk1} and  \eqref{E:aFbk2}
that  $\limsup_{k\to\infty} \norm{a}_{U^2(\Z_{N'_k})}>0$.

	The implication	 $\eqref{I:FH3} \implies \eqref{I:FH2}$ is obvious.
	
Finally, we prove  the implication	 $\eqref{I:FH1} \implies \eqref{I:FH2}$, by slightly modifying the argument in the proof of \cite[Lemma~9.1]{FH17}.
	Arguing by contradiction,  suppose that \eqref{I:FH1} holds but  $\limsup_{k\to\infty} \norm{a}_{U^2(\Z_{N_k})}>0$. We use  the inverse theorem for the $U^2(\Z_N)$ norms and   \cref{L:Kataigeneral},  and argue as 	in the proof of the implication	 $\eqref{I:FH2} \implies \eqref{I:FH3}$. We let   $F(x):=e(x)$  in \eqref{E:aFbk1}, which takes the form
	\begin{equation}\label{E:aFbk'}
| \E_{n\in [N'_k]}\, a(n)\cdot e(n\beta_k)|\geq c>0 \quad \text{for every } k\in \N,
	\end{equation}
	for some  $\beta_k\in [0,1)$, $k\in \N$, $c>0$, and subsequence $(N'_k)$ of $(N_k)$.
	Then  \eqref{E:aFbk2} gives
	$$
	\limsup_{k\to\infty}| \E_{n\in [N'_k/p_0]}\,  e((p_0-q_0)n\beta_{k})|>0
$$
for some $p_0,q_0\in P_0$ with $q_0<p_0$.
	We easily deduce from this that, after passing to a further subsequence
	and choosing representatives of $\beta_k$ modulo $1$, there exist
	$\beta_0\in\Q\cap[0,1)$ and $C>0$ such that
	$$
	|\beta_k-\beta_0|\leq C/N'_k
	$$
	for every $k\in\N$. Relabelling this subsequence $(N'_k)$ as $(N_k)$
	and inserting this back into \eqref{E:aFbk'}, we get
	\begin{equation}\label{E:beta0}
		\limsup_{k\to\infty}
		\left|
		\E_{n\in[N_k]} a(n)\cdot e(n\beta_0)\cdot e(n\gamma_k)
		\right|
		\geq c,
	\end{equation}
	where $\gamma_k:=\beta_k-\beta_0$ satisfies
	$$
	|\gamma_k|\leq C/N_k
	$$
	for every $k\in\N$.
	
	Next, we  eliminate the term $e(n\gamma_k)$, by changing our averaging range $[N_k]$ to $[\alpha N_k]$ for a suitable rational $\alpha\in (0,1]$.
	Let $L\in \N$ be large enough so  that
	$$
	2\pi C/L^2\leq c/2L,
	$$
	i.e., $L$ satisfies  $L\geq 4\pi C/c$.
	We partition the interval $[N_k]$ into subintervals of the form
	$$
	I_j:=((j-1)N_k/L, jN_k/L] \quad \text{ for } j=1,\ldots,L.
	$$
	 Then \eqref{E:beta0} implies that there exist $j_k\in [L]$, $k\in \N$,  such that
\begin{equation}\label{E:beta0'}
	\limsup_{k\to\infty} |\E_{n\in [N_k]} \, {\bf 1}_{I_{j_k}}(n)\cdot  a(n)\cdot e(n \beta_0)\cdot e(n\gamma_k)|\geq c/L.
\end{equation}
	For $n,n'\in I_{j_k}$, we have  	
	$|e(n\gamma_k)-e(n'\gamma_k)|\leq 2\pi C/L$, $k\in \N$, so  if $n_k$ is the left end of the interval $I_{j_k}$, we have
	\begin{equation}\label{E:beta0''}
		\limsup_{k\to\infty} \E_{n\in [N_k]} \, {\bf 1}_{I_{j_k}}(n)|e(n \gamma_k)-e(n_k\gamma_k)|\leq 2\pi C/L^2.
	\end{equation}
	We deduce  from \eqref{E:beta0'}, \eqref{E:beta0''},  and the choice of $L$, that
$$
	\limsup_{k\to\infty} |\E_{n\in [N_k]} \, {\bf 1}_{I_{j_k}}(n)\cdot  a(n)\cdot e(n \beta_0)|>0.
$$
This immediately implies that
\begin{equation}\label{E:alpha}
\limsup_{k\to\infty} |\E_{n\in [\alpha_k\, N_k]} \, a(n)\cdot e(n \beta_0)|>0,
\end{equation}
where $\alpha_k$ is either $(j_k-1)/L$ or $j_k/L$. Since $\alpha_k$ takes
values on a finite set (namely, the set $\{j/L\colon j\in [L]\}$), we deduce that  for some $\alpha \in \Q\cap (0,1]$ we have
\begin{equation}\label{E:alpha'}
	\limsup_{k\to\infty} |\E_{n\in [\alpha\, N_k]} \, a(n)\cdot e(n \beta_0)|>0.
\end{equation}
%This  easily  contradicts   our assumption \eqref{I:FH1}.  Indeed,
On the other hand, if we write $\beta_0=k_0/l_0$ for some $k_0\in \Z_+$, $l_0\in \N$,  then we can use our hypothesis \eqref{I:FH1}, for $\alpha/l_0$ in place of $\alpha$, $l_0$ in place of $q$, and $l$ instead  of $r$,   to get that
$$
\lim_{k\to\infty} \E_{n\in [\alpha/l_0\, N_k]} \, a(l_0n+l)\cdot e((l_0n+l) \beta_0)=e(l\beta_0)\cdot
\lim_{k\to\infty} \E_{n\in [\alpha/l_0\, N_k]} \, a(l_0n+l)=0
$$
 for  $l=0,\ldots, l_0-1.$
%since $e((l_0n+l) \beta_0)=e(l\beta_0)$ for every $n\in \N$.
Averaging over $l\in \{0,\ldots, l_0-1\}$ contradicts  \eqref{E:alpha'}.%, completing the proof.
	\end{proof}

\subsection{Inverse theorem for mixed seminorms}\label{SS:mixedinverse}
The next result gives an inverse theorem for the mixed seminorms defined in \cref{SS:mixeddefinition}. It  will be crucial  for the proof of \cref{T:MainA} and the decomposition result in \cref{T:Decomposition}.
	\begin{theorem}\label{T:Inverse}
		Let $(X,\CX, \mu,T_n)$ be a finitely generated multiplicative action and $F\in L^\infty(\mu)$.
		Then the following properties are equivalent:
		\begin{enumerate}
			\item \label{I:inv1} $F\in X_a$ (or $F\bot X_p$);
			
			\item  \label{I:inv2} 	$\lim_{N\to\infty}\norm{\E_{n\in[N]}\,  T_{qn+r}F}_{L^2(\mu)}=0$
			 for every   $q\in \N$, $r\in\Z_+$;
			
			\item \label{I:inv3}   $\nnorm{F}_{U^2}=0$;
			
			\item \label{I:inv4}   $\nnorm{F}_{U^s}=0$  for every $s\geq 2$.
		\end{enumerate}
	\end{theorem}
	\begin{proof}
We know from \cref{P:aperiodic0} that \eqref{I:inv1} and \eqref{I:inv2} are equivalent.
It also follows from \eqref{E:increase} that \eqref{I:inv4} implies \eqref{I:inv3}.

We prove that   \eqref{I:inv3} implies \eqref{I:inv2}. Note that for fixed $q\in \N, r\in \Z_+,$ and $a\colon \N\to \U$,  we have (see e.g. \cite[Lemma~A.6]{FH17})
$$
|\E_{n\in[N]} \, a(qn+r)|\ll_{q} \norm{a}_{U^2(\Z_{qN+r})}.
$$
It follows that
$$
\norm{\E_{n\in[N]}\,  T_{qn+r}F}^2_{L^2(\mu)}\ll_{q}
 \int \norm{F(T_nx)}_{U^2(\Z_{qN+r})}^2\, d\mu.
 $$
 Hence,
 $$
 \limsup_{N\to\infty}\norm{\E_{n\in[N]}\,  T_{qn+r}F}_{L^2(\mu)}\ll_{q} \nnorm{F}_{U^2}.
 $$
 We deduce that \eqref{I:inv3} implies \eqref{I:inv2}.

  So all that remains is to show 
that   \eqref{I:inv2} implies  \eqref{I:inv4},
  which  is by far the most difficult task, and our main  ingredient is  \cref{T:FH}. Arguing by contradiction, let $s\geq 2$ and  suppose  that
 \eqref{I:inv2} holds, but not  \eqref{I:inv4}.
 By homogeneity, we can assume that   $\norm{F}_\infty\leq 1$.

 Since  $\sum_{p\in \P}\frac{1}{p}=\infty$ and  the action is finitely generated, there exists $P_0\subset \P$ such that
 $\sum_{p\in P_0}\frac{1}{p}=\infty$ and $T_p$ is constant for $p\in P_0$,
 hence
 \begin{equation}\label{E:pp0}
 T_{pn}=T_{p'n} \quad \text{for all } p,p'\in P_0,\,  n\in \N.
 \end{equation}
 Since $\nnorm{F}_{U^s}>0$,
there exist $c>0$ and  $N_k\to \infty $ such that
\begin{equation}\label{E:UNk}
\int \norm{F(T_nx)}_{U^s(\Z_{N_k})}\, d\mu\geq c>0  \quad \text{for every } k\in \N.
\end{equation}
Furthermore, since \eqref{I:inv2} holds, and since mean convergence implies pointwise convergence along a subsequence, using a diagonal argument we get that   there exists a  subsequence $(N_k')$ of $(N_k)$ such that for a.e. $x\in X$ we have
\begin{equation}\label{E:aqnr}
\lim_{k\to\infty}\E_{n\in[\alpha N_k']}\,  F(T_{qn+r}x)=0 \quad \text{ for every } q\in \N, r\in \Z_+, \alpha \in \Q\cap (0,1].
\end{equation}
We deduce from \eqref{E:UNk} using  Fatou's lemma that
$$
\int \limsup_{k\to\infty} \norm{F(T_nx)}_{U^s(\Z_{N'_k})}\, d\mu\geq c>0.
$$
Hence, there exists
a measurable subset  $E$ of $X$ with positive measure such that
 \begin{equation}\label{E:Usa0}
 \limsup_{k\to\infty} \norm{F(T_nx)}_{U^s(\Z_{N'_k})}>0 \quad \text{for every } x\in E.
 \end{equation}
After discarding a null subset of $E$, for $x\in E$ let
$a_x\colon\N\to\U$ be defined by $a_x(n):=F(T_nx)$, $n\in\N$.
By \eqref{E:pp0}, we have 
$$
a_x(pn)=a_x(p'n)
\quad\text{for all }p,p'\in P_0,\ n\in\N.
$$
Therefore, the assumptions of \cref{T:FH} are satisfied for $a_x$
and $P_0$, with $c_{p,p'}=1$.
Using the contrapositive of the implication
$\eqref{I:FH1}\implies\eqref{I:FH3}$, \eqref{E:Usa0} gives that
for every $x\in E$,
 there exist $q_x\in \N, r_x\in \Z_+$, and $\alpha_x\in \Q\cap (0,1]$,  such that
\begin{equation}\label{E:qxrx}
\limsup_{k\to\infty} |\E_{n\in[\alpha_x N'_k]} \, F(T_{q_xn+r_x}x)|>0 \quad   x\in E.
\end{equation}
Since we have countably many possibilities for $\alpha_x,q_x,r_x$  (we do not claim that these functions of $x$ are measurable), using countable subadditivity
we get that   there exist $q_0\in \N, r_0\in \Z_+$, and $\alpha_0\in \Q\cap (0,1]$, such that the measurable set
\begin{equation}\label{E:qr0}
	E_0:=\{x\in E\colon  \limsup_{k\to\infty}| \E_{n\in [\alpha_0N'_k]}\, F(T_{q_0n+r_0}x)|>0\}
\end{equation}
has positive measure; if it does not, then \eqref{E:qxrx}  would imply that $E$ is contained in a countable union of sets with measure zero, which cannot happen since $E$ has positive measure.  This contradicts \eqref{E:aqnr}, and completes the proof.
 \end{proof}

	\subsection{Proof of \cref{T:Decomposition}}
The starting point of the proof is the same as in the proof of \cref{T:DecompositionGeneral} in \cref{S:ProofThmB}.
Properties \eqref{I:Decomposition1},
\eqref{I:Decomposition2},
\eqref{I:Decomposition3}  of \cref{T:Decomposition} follow from
\eqref{E:concentrationa} in \cref{P:Concentrationfg}, \cref{T:Inverse},
and \cref{P:aperiodicP0},
respectively.

	\section{Pairwise independent linear  forms - Proof   of \cref{T:MainA}}\label{S:ProofThmA}
	%Proof   of \cref{T:MainA}}
Our goal in this section is to prove \cref{T:MainA}. For a proof sketch see \cref{SS:proofsketchA}.
 	\subsection{Characteristic factors} We start with a result that gives convenient characteristic factors for the averages in \eqref{E:MainA0}.
 	 \begin{proposition}\label{P:Characteristic1}
 	Let  $(X,\CX,\mu,T_{1,n},\ldots, T_{\ell,n})$ be a multiplicative action and
 	$L_1,\ldots, L_\ell$ be linear forms with non-negative coefficients  such that $L_1,L_j$ are independent for $j=2,\ldots, \ell$. Suppose $F_1,\ldots, F_\ell \in L^\infty(\mu)$ and $\nnorm{F_1}_{U^s(T_1)}=0$ where $s:=\max(\ell-1,2)$. Then for any $2$-dimensional grid $\Lambda$ we have
 	\begin{equation}\label{E:cflinear}
 		\lim_{N\to\infty}	\E_{m,n\in  [N]}\, {\bf 1}_\Lambda(m,n)\cdot T_{1,L_1(m,n)}F_1\cdots T_{\ell,L_\ell(m,n)}F_\ell=0
 	\end{equation}
 	in $L^2(\mu)$.
 \end{proposition}
 \begin{proof}
 	We assume throughout that $\norm{F_j}_\infty\leq 1$, for $j\in [\ell]$.
 	Using \cite[Lemma~3.4]{FH16} there exist  $l=l(L_1,\ldots, L_\ell)\in \N$, $c=c(l)\in (0,1/2]$, and prime numbers $\tilde{N}\in [lN,2lN]$, $N\in\N$,    such that  for any sequences $a_1,\ldots, a_\ell\colon \N\to \U$ we have
 	$$
 	|\E_{m,n\in [N]} \, {\bf 1}_\Lambda(m,n)\cdot a_1(L_1(m,n))\cdots a_\ell(L_\ell(m,n))|\ll_{\Lambda,l}
 	%%\norm{a_1\cdot {\bf 1}_{[N]}}_{U^s(\Z_{lN})}^c+o_N(1)
 	 \norm{a_1}_{U^s(\Z_{\tilde{N}})}^c +o_N(1),
 	$$
 	where the $o_N(1)$ term does not depend on the sequences $a_1,\ldots, a_\ell$.\footnote{The estimate in \cite[Lemma~3.4]{FH16} is given when $\Lambda=\Z^2$. To treat the general case, we express ${\bf 1}_\Lambda(m,n)$ as a linear combination of sequences of the form $e(m\alpha +n\beta)$, where $\alpha,\beta\in \Q$, and note that the argument in \cite[Lemma~3.4]{FH16} can also be used to treat such weights.}

 For $x\in X$, using  this estimate for
 	$
 	a_{j,x}(n):=F_j(T_{j,n}x), $ $n\in \N,$  $j\in [\ell]$,
 	we deduce that
 	$$
 	|\E_{m,n\in [N]} \,  {\bf 1}_\Lambda(m,n)\cdot F_1(T_{1,L_1(m,n)}x)\cdots F_\ell(T_{\ell,L_\ell(m,n)}x)|^2\ll_{\Lambda,l}
 	\norm{F_1(T_{1,n}x)}_{U^s(\Z_{\tilde{N}})}^{2c} +o_N(1).
 	$$
 	Integrating with respect to $\mu$ and  using that $2c\in (0,1]$, we get
 	\begin{multline*}
 	\norm{ \	\E_{m,n\in[N]}\, {\bf 1}_\Lambda(m,n)\cdot T_{1,L_1(m,n)}F_1\cdots T_{\ell,L_\ell(m,n)}F_\ell }_{L^2(\mu)} \ll_{\Lambda,l}  \\ \norm{\norm{F_1(T_{1,n}x)}_{U^s(\Z_{\tilde{N}})}}^c_{L^2(\mu)}+o_N(1).
 	\end{multline*}
 	Since $\nnorm{F_1}_{U^s(T_1)}=0$,  the limit as $N\to\infty$ of the right side is $0$. The result follows.
 \end{proof}

\subsection{Proof of part~\eqref{I:MainA1}  of \cref{T:MainA}}\label{SS:A1}	
	Suppose first that all actions $(X,\CX,\mu,T_{j,n})$, $j\in [\ell]$, are aperiodic. To show the mean convergence to the product of the integrals, it suffices to show that if $\int F_j\, d\mu=0$ for some $j\in [\ell]$, then the averages \eqref{E:MainA0} converge to $0$ in $L^2(\mu)$. Let $j\in [\ell]$. Since the action $(X,\CX,\mu,T_{j,n})$ is aperiodic  and $\int F_j\, d\mu=0$, we have $\lim_{N\to\infty}\E_{n\in[N]}\,  T_{j,qn+r}F_j=0$ in $L^2(\mu)$ for all $q\in \N$, $r\in \Z_+$.  Since the actions are finitely generated, \cref{T:Inverse}  implies that $\nnorm{F_j}_{U^s(T_j)}=0$ for all $s\in \N$,  and \cref{P:Characteristic1} gives that  the averages \eqref{E:MainA0} converge to $0$ in $L^2(\mu)$
	 (we used here that the forms $L_1,\ldots, L_\ell$ are pairwise independent).
	
Suppose now that 	$(X,\CX,\mu,T_{j,n})$, $j\in [\ell]$, are general finitely generated multiplicative actions. For $j\in [\ell]$, 	using the decomposition result of \cref{T:Decomposition},
	we  write
		$$
		F_j=F_{j,p}+F_{j,a},
		$$
		where $F_{j,p}\in X_{j,p}$ and  $\nnorm{F_{j,a}}_{U^s(T_j)}=0$ for every $s\in \N$.
		Using this, and 		since the linear forms $L_1,\ldots, L_\ell$ are pairwise independent, we deduce from  \cref{P:Characteristic1}  that the limiting behavior of the averages \eqref{E:MainA0} is the same as  that of the averages
	\begin{equation}\label{E:AmnL}
	\E_{m,n\in[N]}\,  A(m,n), \quad A(m,n):= {\bf 1}_{\Lambda}(m,n)\cdot T_{1,L_1(m,n)}F_{1,p}\cdots T_{\ell,L_\ell(m,n)}F_{\ell,p},	
	\end{equation}
	in the sense that the difference of the two averages converges to $0$ in $L^2(\mu)$ as $N\to \infty$.
	So it suffices to show that the averages in \eqref{E:AmnL} converge in $L^2(\mu)$.
	Let $Q_K,S_K$ be as in \eqref{E:defQKSK}
	and    $S_{K; L_1,\ldots, L_\ell}$  be as in \eqref{E:defSKL}.
	We have
	$$
	\lim_{K\to\infty}\limsup_{N\to\infty}\norm{\E_{m,n\in[N]}\, A(m,n)-	\E_{a,b\in [Q_K]}
		\E_{m,n\in[N/Q_K]}\, A(Q_Km+a,Q_Kn+b)
	}_{L^2(\mu)}=0.
	$$
	It follows from this and  \eqref{E:SKQKL} that
	\begin{equation}\label{E:AAKN}
	\lim_{K\to\infty}\limsup_{N\to\infty}\norm{\E_{m,n\in[N]}\, A(m,n)-	A_{K,N}}_{L^2(\mu)}=0,
	\end{equation}
	where
	\begin{equation}\label{E:AKN}
A_{K,N}:=\E_{(a,b)\in S_{K;L_1,\ldots, L_\ell}}\E_{m,n\in[N/Q_K]}\, A(Q_Km+a,Q_Kn+b).
\end{equation}
Since $(a,b)\in S_{K;L_1,\ldots, L_\ell}$ implies that $p^K\nmid L_j(a,b)$ for every prime $p\leq K$, and
$1\leq L_j(a,b)\leq C Q_K$ for some fixed $C>0$, and since $F_{j,p}\in X_{j,p}$ for $j\in [\ell]$,
by  \eqref{E:fconcfingen''}  in  \cref{P:Concentrationfg}   and \cref{L:lN} we have
that for every $Q\in \N, r\in \Z_+,$ and $F\in X_{j,p},$ the limit
$
A_{j,Q,r}(F):=\lim_{N\to\infty}\E_{n\in [N]}\, T_{j,Qn+r}F
$
exists in $L^2(\mu)$, and
%\begin{multline*}
%\lim_{K\to\infty} 	\limsup_{N\to\infty} \max_{(a,b)\in %S_{K;L_1,\ldots,L_\ell}}\\ %\E_{n\in[N/Q_K]}\norm{T_{j,Q_Kn+L_j(a,b)}F_{j,p}-
%	A_{j,Q_K,L_j(a,b)}(F_{j,p})}_{L^2(\mu)}=0, \quad j\in [\ell].
%\end{multline*}
%Using \cite[Lemma~3.1]{FKM23} we deduce that
$$
\lim_{K\to\infty} 	\limsup_{N\to\infty} \max_{(a,b)\in S_{K;L_j}} \E_{m,n\in[N]}\norm{T_{j,L_j(Q_Km+a,Q_Kn+b)}F_{j,p}-	A_{j,Q_K,L_j(a,b)}(F_{j,p})}_{L^2(\mu)}=0
$$
for $j\in [\ell]$.
%%where $l_j$ is the sum  of the coefficients of $L_j$.
Note also that for all sufficiently large $K$  we have
$$
{\bf 1}_{\Lambda}(Q_Km+a,Q_Kn+b)={\bf 1}_{\Lambda}(a,b) \quad \text{for all } m,n\in \N.
$$

Using these identities,  \eqref{E:AKN}, and the form of $A(m,n)$ given in \eqref{E:AmnL}, we deduce  that
	\begin{equation}\label{E:AjKN}
	\lim_{K\to\infty}\limsup_{N\to\infty}\norm{A_{K,N}-
		A_K}_{L^2(\mu)}=0,
\end{equation}
where
$$
A_K:=\E_{(a,b)\in S_{K;L_1,\ldots, L_\ell}}\, {\bf 1}_{\Lambda}(a,b)\cdot  \prod_{j=1}^\ell A_{j,Q_K,L_j(a,b)}(F_{j,p}), \quad K\in \N.
$$
Combining \eqref{E:AAKN} and \eqref{E:AjKN}, we deduce
that  the sequence $(\E_{m,n\in[N]}\, A(m,n))_{N\in\N}$  is Cauchy and therefore converges in $L^2(\mu)$.

\subsection{Proof of part~\eqref{I:MainA2} of \cref{T:MainA}}\label{SS:A2}	
  Let $\varepsilon>0$,  $F:={\bf 1}_A$,
  $Q\in \N$, which  will be determined later, and $m_0,n_0\in \Z$ be as in the statement of   part~\eqref{I:MainA2} of \cref{T:MainA}. It suffices to show that
  (the limit exists by  part~\eqref{I:MainA1} of \cref{T:MainA})
  \begin{equation}\label{E:lowerneededA-}
  	\lim_{N\to\infty}\E_{m,n\in [N]}\int F \cdot T_{1,L_1(Qm+m_0,Qn+n_0)}F\cdots T_{\ell,L_\ell(Qm+m_0,Qn+n_0)}F\, d\mu
  	\geq \Big(\int F\, d\mu\Big)^{\ell+1}	-\varepsilon.
  \end{equation}

For $j\in[\ell]$, using the decomposition result of \cref{T:Decomposition}, we write
$$
F=F_{j,p}+F_{j,a},
$$
where $F_{j,p}=\E(F|\CX_{j,p})\in X_{j,p}$ and  $\nnorm{F_{j,a}}_{U^s(T_j)}=0$ for every $s\in \N$. Using \cref{P:Characteristic1}  for appropriate $\Lambda$,  we get that the limit in
\eqref{E:lowerneededA-}
 is equal to
\begin{equation}\label{E:lowerneededA}
\lim_{N\to\infty}\E_{m,n\in [N]}\int F \cdot T_{1,L_1(Qm+m_0,Qn+n_0)}F_{1,p}\cdots T_{\ell,L_\ell(Qm+m_0,Qn+n_0)}F_{\ell,p}\, d\mu.	
\end{equation}
 So it remains to show that for suitable $Q\in \N$,  the last limit is at least
  $(\int F\, d\mu)^{\ell+1}	-\varepsilon$.

\smallskip

{\bf Case 1.} Suppose that $L_j(m_0,n_0)=1$ for $j=1,\ldots, \ell$.
Using property~\eqref{I:Decomposition1} of \cref{T:Decomposition} we get that
$$
\lim_{K\to\infty}\limsup_{N\to\infty} \max_{Q\in \Phi_K}\E_{m,n\in[N]}\norm{T_{j,Qn+1}F_{j,p}-F_{j,p}}_{L^2(\mu)}=0, \quad j\in [\ell].
$$
Since $L_j(m_0,n_0)=1$ for $j\in[\ell]$,
using \cref{L:lN}  for $v(n):=T_{j,Qn+1}F_{j,p}$, $v_N:=F_{j,p}$, we deduce that for $j\in[\ell]$ we have
\begin{equation}\label{E:A2Conc}
	\lim_{K\to\infty}\limsup_{N\to\infty} \max_{Q\in \Phi_K}\E_{m,n\in[N]}\norm{T_{j,L_j(Qm+m_0,Qn+n_0)}F_{j,p}-F_{j,p}}_{L^2(\mu)}=0.
\end{equation}
For $K\in \N$, let $Q_K\in \Phi_K$ be arbitrary. It follows from the above, using a telescoping argument, that  the iterated limit (note that the limit as $N\to \infty$ exists by part~\eqref{I:MainA1})
$$
\lim_{K\to\infty} \lim_{N\to\infty}
\E_{m,n\in[N]} \int F\cdot T_{1,L_1(Q_Km+m_0,Q_Kn+n_0)}F_{1,p}\cdots T_{\ell,L_\ell(Q_Km+m_0,Q_Kn+n_0)}F_{\ell,p}\
\, d\mu
$$
is equal to
$$
\int F \cdot F_{1,p}\cdots F_{\ell,p}\, d\mu=\int F\cdot   \E(F|\CX_{1,p})\cdots  \E(F|\CX_{\ell,p})\, d\mu\geq\Big(\int F\, \, d\mu\Big)^{\ell+1},
$$
where the lower bound  follows from \cref{T:Chu}.

Combining the above,  we get  that  there exists $Q\in \N$ such that the expression in \eqref{E:lowerneededA} is at least $(\int F\, \, d\mu)^{\ell+1}-\varepsilon$, and  consequently  \eqref{E:lowerneededA-} holds for this same $Q$, as requested.

\smallskip

{\bf Case 2.} Suppose that $L_1(m_0,n_0)=0$ and $L_j(m_0,n_0)=1$ for $j=2,\ldots, \ell$.
Let
\begin{equation}\label{E:tildeF0}
	\tilde{F}_{1,p}:=\lim_{N\to\infty}\E_{m,n\in[N]}\, T_{1,L_1(m,n)}F_{1,p},
\end{equation}
where  the limit is taken in $L^2(\mu)$ and  exists by part~\eqref{I:MainA1} since the action $(X,\mu,T_{1,n}) $ is finitely generated.
Since $L_1(m_0,n_0)=0$, we have that $L_1(Qm+m_0,Qn+n_0)=QL_1(m,n)$, hence  for $G\in L^\infty(\mu)$, which will be determined later,  we get by the bounded convergence theorem that for every $Q\in \N$ we have
%%\begin{equation}\label{E:limG}
	$$
\lim_{N\to\infty} \E_{m,n\in[N]}\int T_{1,L_1(Qm+m_0,Qn+n_0)}F_{1,p}\cdot G\, d\mu = \int T_{1,Q}\tilde{F}_{1,p}\cdot G \, d\mu.
$$
%%\end{equation}
Moreover, we get by  \cref{L:multaverid} that ($\CI_{T_1}$ is as in \eqref{E:I})
$$
\lim_{K\to\infty}\E_{Q\in \Phi_K}\int T_{1,Q}\tilde{F}_{1,p}\cdot G \, d\mu =\int  \E(\tilde{F}_{1,p}|\CI_{T_1})\cdot G\, d\mu =
\int  \E(F|\CI_{T_1})\cdot G\, d\mu,
$$
where we used that
\begin{equation}\label{E:IT1}
\E(\tilde{F}_{1,p}|\CI_{T_1})=\E(F_{1,p}|\CI_{T_1})=\E(F|\CI_{T_1}).
\end{equation}
To see that the first equality holds, we use that $f_n\to f$ in $L^2(\mu)$ implies that $\E(f_n|\CI_{T_1})\to \E(f|\CI_{T_1})$ in $L^2(\mu)$,  the form of
$\tilde{F}_{1,p}$ given in \eqref{E:tildeF0}, and the fact that  $\E(T_{1,L_1(m,n)}F_{1,p}|\CI_{T_1})=\E(F_{1,p}|\CI_{T_1})$ for every $m,n\in \N$.  For  the second equality, we use that  $F_{1,p}=\E(F|\CX_{1,p})$ and $\CI_{T_1}\subset \CX_{1,p}$.
Combining the previous three identities, we get
	\begin{equation}\label{E:A12}
\lim_{K\to\infty}\E_{Q\in \Phi_K}\lim_{N\to\infty} \E_{m,n\in[N]}\int T_{1,L_1(Qm+m_0,Qn+n_0)}F_{1,p}\cdot G\, d\mu =\int  \E(F|\CI_{T_1})\cdot G\, d\mu.
\end{equation}

Using  \eqref{E:A2Conc} for $j=2,\ldots, \ell$ (here we use that $L_j(m_0,n_0)=1$ for $j=2,\ldots, \ell$), \eqref{E:A12} for $G:=F\cdot F_{2,p}\cdots F_{\ell,p}=F\cdot  \E(F|\CX_{2,p})\cdots  \E(F|\CX_{\ell,p})$,  and a telescoping argument,   we get that  the iterated limit
(note that the limit below as $N\to \infty$ exists by part~\eqref{I:MainA1}, so the average over $Q$ and the limit over $N$ below can be exchanged)
\begin{equation}\label{E:FFjp}
\lim_{K\to\infty} \E_{Q\in \Phi_K}\lim_{N\to\infty}
\E_{m,n\in[N]} \int F\cdot T_{1,L_1(Qm+m_0,Qn+n_0)}F_{1,p}\cdots T_{\ell,L_\ell(Qm+m_0,Qn+n_0)}F_{\ell,p}\
 \, d\mu
\end{equation}
 is equal to
$$
\int F\cdot  \E(F|\CI_{T_1}) \cdot \E(F|\CX_{2,p})\cdots  \E(F|\CX_{\ell,p})\, d\mu\geq \Big(\int F\, d\mu\Big)^{\ell+1},
$$
where the lower bound follows again from \cref{T:Chu}.

Combining the above,  we get that there exists $Q\in \N$ such that \eqref{E:lowerneededA-} holds, as required.

	\section{Two rational polynomials - Proof of \cref{T:MainB}}\label{S:ProofThmB1}
	%	Proof of \cref{T:MainB}}
Our goal in this section is to prove \cref{T:MainB}. For a proof sketch, see \cref{SS:proofsketchB}.
	
	\subsection{Characteristic factors} Our first goal is to obtain convenient  characteristic factors for the averages in \eqref{E:R1R2Converge}.
		We will use the following %% $2$-dimensional
		 variant of the Daboussi-K\'atai orthogonality criterion, whose proof is analogous to that
of 		 \cref{L:Kataigeneral},
%		  given in~\cite{K86}, 
so we will omit it. 
		% (see also \cite[Lemma~2.14]{Cha24}):
	\begin{lemma}\label{L:Katai}
		For $N\in \N$ let $(A_N(m,n))_{m,n\in\N}$ be a $1$-bounded sequence  in an inner product space  $H$, $C>0$, and  $P_0\subset \P$  be such that  $\sum_{p\in P_0}\frac{1}{p}=\infty$. Suppose that
	$$
	\lim_{N\to\infty}
	\E_{\substack{m\in[CN/\max(p,p')]\\
			n\in[CN/\max(q,q')]}}
	\big\langle
	A_N(pm,qn),A_N(p'm,q'n)
	\big\rangle
	=0
	$$
	for all $p,q,p',q'\in P_0$ such that
	$
	p/q\neq p'/q'.
	$
	Then
	$$
	\lim_{N\to\infty}
	\norm{
	\E_{m,n\in[CN]} A_N(m,n)}=0.
	$$
	\end{lemma}
	The next result is a $2$-dimensional generalization
	of the identity
	$
	{\bf 1}_{l\cdot \Z}(n)=\E_{q\in[l]}\, e(nq/l).
	$
	\begin{lemma}\label{L:RAZA}
 		Let $A$ be an invertible  $2\times 2$ matrix with integer coefficients and
		$$
		Z_A:=\{(q_1,q_2)\in (\Q\cap [0,1))^2\colon (q_1,q_2) \cdot A\in \Z^2\}, \quad
		R_A:=A\cdot \Z^2.
		$$
(Note that $Z_A$ is a finite set.)		Then
		$$
	{\bf 1}_{R_A}(m,n)=	\E_{(q_1,q_2)\in Z_A}\, e(mq_1+nq_2).
		$$  	
	\end{lemma}
\begin{proof}
	For convenience we let
	$$
	\bar{m}:=(m,n)^\top \quad \text{and} \quad  \bar{q}:=(q_1,q_2),
	$$
	where $v^\top$ denotes the transpose of a vector $v$.
	The claimed identity takes the form
	\begin{equation}\label{E:RA}
{\bf 1}_{R_A}(\bar{m})=	\E_{\bar{q}\in Z_A}\, e(\bar{q}\cdot \bar{m}).
\end{equation}

Suppose that $\bar{m}\in R_A$. Then $\bar{m}=A\cdot \bar{k}$ for some $\bar{k}\in \Z^2$ and the right  side in \eqref{E:RA} is
	$$
	\E_{\bar{q}\in Z_A}e(\bar{q}\cdot A\cdot \bar{k})=1,
$$
since    $\bar{q}\cdot A\in \Z^2$ for $\bar{q}\in Z_A$, and consequently $\bar{q}\cdot A\cdot \bar{k}\in \Z$.

Suppose now that $\bar{m}\not\in R_A$, or equivalently, that  $A^{-1}\cdot \bar{m}\not \in \Z^2$. We claim that  there exists $\bar{q}_0\in Z_A$ such that  $\bar{q}_0\cdot \bar{m} \notin \Z$. Indeed, if this is not the case, since $Z_A=\Z^2\cdot A^{-1} \!\! \pmod{1}$, we  have that $\bar{k}\cdot A^{-1}\cdot \bar{m}\in \Z$  for every $\bar{k}\in \Z^2$. Letting $\bar{k}:=(1,0)$ and $(0,1)$, we get  $\text{I}_2\cdot A^{-1}\cdot \bar{m}\in \Z^2$, where $I_2$ is the identity matrix. Hence, $A^{-1}\cdot \bar{m}\in \Z^2$, a contradiction.

 So let $\bar{q}_0\in Z_A$ be such that  $\bar{q}_0\cdot \bar{m} \notin \Z$.
 Then
$$
e(\bar{q}_0\cdot \bar{m})\cdot \E_{\bar{q}\in Z_A}e(\bar{q}\cdot \bar{m})=
\E_{\bar{q}\in Z_A}e((\bar{q}+\bar{q}_0)\cdot \bar{m})=\E_{\bar{q}\in Z_A}e(\bar{q}\cdot \bar{m}),
$$
since $Z_A+\bar{q}_0=Z_A \pmod{1}$.
Since $e(\bar{q}_0\cdot \bar{m}) \neq 1$, we deduce that
$$
\E_{\bar{q}\in Z_A}e(\bar{q}\cdot \bar{m})=0.
$$
This completes the proof.
\end{proof}
We plan to use the following consequence:
\begin{corollary}\label{C:RAZA}
Let $\Lambda$ be  a $2$-dimensional grid, $A$ be an invertible  $2\times 2$ matrix with integer coefficients,  $Z_A, R_A$ be as in \cref{L:RAZA}, and for $N\in \N$ let  $A_N:=A([1,N]\times [1,N])$ and $\Lambda_N:=\Lambda \cap ([N]\times [N])$.  Then for all $a,b,N\in \N$ there exist an invertible $2\times 2$ matrix $A'$, convex subsets $A_{a,b,N}$ of $\R_+^2$, and $m_0,n_0\in \Z$ (depending only on $A,\Lambda$), such that
 $$
{\bf 1}_{A(\Lambda_N)}(am,bn)=
{\bf 1}_{A_{a,b,N}}(m,n)\cdot \E_{(q_1,q_2)\in Z_{A'}}\, e((am-m_0)q_1+(bn-n_0)q_2), \quad m,n\in\N.
$$	
\end{corollary}
\begin{proof}
Suppose that   $\Lambda$ has the form  $\Lambda= D\cdot \Z^2+v$, where $D$ is a diagonal $2\times 2$ matrix with entries in $\N$,  and $v\in \Z_+^2$. Then 		
$A(\Lambda)=A' \cdot \Z^2+v'$, where $A':=AD$ and $v':=Av=(m_0,n_0)$ for some $m_0,n_0\in\Z$.
Using  this and	 \cref{L:RAZA} we get
\begin{equation}\label{E:ALA}
{\bf 1}_{A(\Lambda)}(m,n)=		{\bf 1}_{A'\cdot \Z^2}(m-m_0,n-n_0)=\E_{(q_1,q_2)\in Z_{A'}}\, e((m-m_0)q_1+(n-n_0)q_2).
\end{equation}
Finally, note that since 	$\Lambda \cap ([N]\times [N])=\Lambda \cap ([1,N]\times [1,N])$  and $A$ is injective, we have
$A(\Lambda_N)=A(\Lambda)\cap A_N$. Combining this with \eqref{E:ALA} gives the asserted identity with  $A_{a,b,N} :=\{(t,s)\in \R_+^2\colon (at,bs)\in A_N\}$, which are convex subsets of $\R_+^2$.
\end{proof}
We will also use the following elementary lemma, which will later allow us to verify some of the hypotheses of    \cref{P:aperiodicP0} while proving \cref{P:Characteristic2}
below:

\begin{lemma}\label{L:Sr}
 Let
$
R(m,n)=c\prod_{j=1}^s L_j(m,n)^{k_j},
$
where $c\in\Q_+^*$, $k_j\in\Z$, and $L_j\in\Z[m,n]$ are non-trivial
 linear forms. Suppose that $R$ takes a positive value at some point in $\N^2$ and is not of the form
$
c'\,m^k n^l S^r
$
for any $c'\in\Q_+^*$, $k,l\in\Z$, $S\in\Q(m,n)$, and $r\geq2$.
If $a,b,a',b'\in\N$ satisfy $a/b\neq a'/b'$, then
$
R(am,bn)/R(a'm,b'n)
$
is not of the form $c'S^r$ for any $c'\in\Q_+^*$,
$S\in\Q(m,n)$, and $r\geq2$.							
\end{lemma}
\begin{proof}
Let $d$ be the degree of homogeneity of $R$. After writing
$R(m,n)=n^dR(m/n)$, where $R(n):=R(n,1)$, we get that it suffices to show
the following. Let $R(n)$ be a rational function that takes a positive value at some positive rational point, factors as a product
of linear terms, and is not of the form $c\,n^kS^r(n)$ for any
$c\in\Q_+^*$, $k\in\Z$, rational function $S$, and $r\geq2$. Let also
$a\in\Q$ with $a>1$. Then $\widetilde R(n):=R(an)/R(n)$ is not of the
form $c\,S^r$ for any $c\in\Q_+^*$, rational function $S$, and $r\geq2$.
To see this,  note that our assumptions imply that
$R(n)=\alpha\,n^{k_0}\prod_{j=1}^{\ell}(n-\alpha_j)^{k_j}$, where
$k_0\in\Z$, 
$\alpha\in\Q\setminus\{0\}$, $\alpha_1,\ldots,\alpha_\ell\in\Q\setminus\{0\}$
are distinct, and $k_1,\ldots,k_\ell$ are non-zero integers.
Fix an integer $r\geq2$. If all $k_j$ were multiples of $r$, then
$R(n)=\alpha n^{k_0}S(n)^r$ for some $S\in\Q(n)$.
For even $r$, positivity at a positive point forces $\alpha>0$;
for odd $r$, its sign can be absorbed into $S$.
Either case contradicts the hypothesis. Thus, 
$k_1,\ldots,k_\ell$ are non-zero integers, not all of them multiples of
$r$. We can assume that
$|\alpha_1|=\min\{|\alpha_j|\colon r\nmid k_j,\ j\in[\ell]\}$.
Consider the multiplicity of the factor $n-\alpha_1/a$ in
$\widetilde R(n)=R(an)/R(n)$. It is $k_1-k_j$ if
$\alpha_j=\alpha_1/a$ for some $j\in[\ell]$, and is $k_1$ otherwise.
In the former case, since
$|\alpha_j|=|\alpha_1|/a<|\alpha_1|$, the choice of $\alpha_1$ gives
$r\mid k_j$. Hence, in either case, the multiplicity of
$n-\alpha_1/a$ in $\widetilde R$ is not divisible by $r$. It follows that
$\widetilde R$ is not of the form $c\,S^r$ for any $c\in\Q_+^*$ and
rational function $S$, as required.
\end{proof}
We are now ready to prove our main result regarding  characteristic factors.
\begin{proposition}\label{P:Characteristic2}
Let $(X,\CX,\mu,T_{1,n}, T_{2,n})$ be a  finitely generated multiplicative action
and $F_1,F_2\in L^\infty(\mu)$ such that  $F_1\in X_{1,a}$ or $F_2\in X_{2,a}$.
Let		 $R_1$ be  a rational polynomial that factors linearly, $L_1,L_2$ be    independent linear forms,
and
$$
R_2(m,n)= c\,  L_1(m,n)^k\cdot L_2(m,n)^l,
$$
where  $c\in \Q_+,k,l\in \Z$. Suppose  that  $R_1$ is not of the form $  c'L_1^{k'}L_2^{l'} R^r$ for any  $c'\in \Q^*_+,k',l'\in \Z$, rational polynomial $R$, and $r\geq 2$,  and    $R_2$ is not of the form $c' R^r$ for any   $c'\in \Q_+$,  rational polynomial $R$, and $r\geq 2$. Then for any $2$-dimensional grid $\Lambda$ we have
\begin{equation}\label{E:Pzero}
\lim_{N\to\infty}  \E_{m,n\in  [N]}\, {\bf 1}_\Lambda(m,n)\cdot T_{1,R_1(m,n)}F_1\cdot T_{2,R_2(m,n)}F_2=0
\end{equation}
in $L^2(\mu)$.
Moreover, if we only assume that the action $(X,\CX,\mu,T_{2,n})$ is finitely generated
(and  the action $(X,\CX,\mu,T_{1,n})$ is arbitrary) and $F_1\in X_{1,a}$, then  \eqref{E:Pzero} still holds.\footnote{It is possible to show  that if $F_2\in X_{2,a}$, then  \eqref{E:Pzero} holds, but the argument is much more cumbersome in this case, and since it is not needed in subsequent applications we will skip it.}
\end{proposition}
\begin{remarks}
$\bullet$	For general multiplicative actions,  $F_1\in X_{1,a}$ does not imply that \eqref{E:Pzero} holds. Indeed, take $R_1(m,n):=(m+n)/n$, $R_2(m,n):=m/n$,  $T_n\colon \T\to\T$ defined by  $T_nx:=nx\pmod{1}$, $\mu:=m_\T$,  $T_{1,n}=T_{2,n}=T_n$,  and $F_1(x)=\overline{F_2}(x):=e(x)$, which have mean value $0$. One can easily verify that $X_p$ is trivial for this system, hence $F_1,F_2\in X_a$. However, for $\Lambda=\Z^2$ the averages in \eqref{E:Pzero} are equal to  $F_1$ for every $N\in \N$.
		
$\bullet$	In the proof of part~\eqref{I:MainB3} of \cref{T:MainB} below we will need a variant of
\eqref{E:Pzero} where $m,n$ are also restricted to the set $S_{\delta,R_1}$ defined in \eqref{E:SdeltaR}.
	 The proof below applies with these additional restrictions. Under the
linear change of variables in Step~1 the sectors remain convex, and the
intersections of their inverse images under the dilations in Step~2 are
finite unions of convex sets. Hence, \cref{P:aperiodicP0}, with $t=0$,
gives the required vanishing in Step~3. The boundary rays contain
$O(N)$ lattice points in $[N]^2$ and therefore give a negligible error.
For the shifted expressions used in the proof of
part~\eqref{I:MainB3} of \cref{T:MainB}, we first make the affine
substitution
$
(u,v)=(Qm+m_0,Qn+n_0),
$
with $Q$ fixed. Each truncated sector is then mapped to a convex set,
and the congruence conditions are handled in the same way as the grid
conditions. Finally, the normalization by
$|S_{\delta,R_1}\cap[N]^2|$ is harmless, since by \cref{L:SdR} this
cardinality is bounded below by a positive multiple of $N^2$ for fixed
$\delta$.
		\end{remarks}
	\begin{proof}
		We first show that if $F_1\in X_{1,a}$ then \eqref{E:Pzero} holds. This part of the argument works equally well for arbitrary actions $(X,\CX,\mu,T_{1,n})$, so we only have to assume that the action $(X,\CX,\mu,T_{2,n})$ is finitely generated.   We carry out the proof in several steps.
		
		\smallskip
{\bf Step 1} (Reduction to $R_2(m,n)=m^k\, n^l$). We first  reduce to a statement where instead of $R_2(m,n)$ we have a monomial  of the form $m^k\, n^l$ for some $k,l\in \Z$. This preparatory step is necessary  to get a simplification when we use \cref{L:Katai} later. We can assume that  $R_1=c_1\, \prod_{j=1}^sL^{k_j}_j$ for some $c_1\in \Q_+$, $s\geq 3$, $k_1,k_2\in \Z$, non-zero $k_3,\ldots, k_s\in \Z$, and non-trivial  linear forms $L_j(m,n)=a_jm+b_jn$, where $a_j,b_j\in \Z_+$,  not both zero, $j\in [s]$,   such that $L_1,L_2$ are independent and  none of the forms $L_3,\ldots, L_s$ is a rational  multiple of $L_1$ or $L_2$. We will perform the change of variables
$\tilde{m}=a_1m+b_1n$, $\tilde{n}=a_2m+b_2n$. More precisely, we
define the matrix  $
A:=\begin{pmatrix}  a_1\quad b_1 \\ a_2\quad b_2
	\end{pmatrix}
$ and the corresponding linear transformation $L (m,n)^\bot:=A(m,n)^\bot$ ($v^\bot$ denotes the transpose of a vector $v$), and note that  since the linear forms $L_1,L_2$ are independent, $A$ is invertible
and $(m,n)^\bot=A^{-1}\cdot (\tilde{m},\tilde{n})^\bot$.  Let $d:=|\det(A)|>0$.
We have
	\begin{equation}\label{E:Pzerosum}
 \sum_{m,n\in [N]}\, {\bf 1}_{\Lambda}(m,n)\cdot  T_{1,R_1(m,n)}F_1\cdot T_{2,R_2(m,n)}F_2=
\sum_{\tilde{m},\tilde{n}\in L(\Lambda_N)}\, T_{1,\tilde{R}_1(\tilde{m},\tilde{n})}\tilde{F}_1\cdot T_{2,\tilde{R}_2(\tilde{m},\tilde{n})}\tilde{F}_2,
\end{equation}	
	where
	$$
 \tilde{F}_1:=T_{1,c_1\, d^{-(k_1+\cdots+k_s)}}F_1,\quad
\tilde{F}_2:=T_{2,c}F_2, \quad \Lambda_N:=\Lambda\cap ([N]\times [N]),
	$$
$$
\tilde{R}_1(\tilde{m},\tilde{n}):=\prod_{j=1}^s{\tilde{L}^{k_j}}_j(\tilde{m},\tilde{n}), \, \,
\tilde{R}_2(\tilde{m},\tilde{n}):=\tilde{m}^k\tilde{n}^l,
$$
$$
\tilde{L}_j(\tilde{m},\tilde{n}):=d\, (a_j,b_j)A^{-1}(\tilde{m},\tilde{n})^\bot, \, j\in [s].
$$
Note that $\tilde{L}_1,\ldots,\tilde{L}_s$ have integer coefficients
(this is why we multiplied by $d$), and the relevant non-power assumptions
on $R_1,R_2$ carry over to $\tilde{R}_1,\tilde{R}_2$. Although the
transformed linear forms may have signed coefficients, they are positive
on the transformed support, and in Step~3 we use the signed-form extension
following \cref{P:aperiodicP0}.
Let $C\geq 1$ be large enough so that  $L(\Lambda_N)\subset [CN]\times [CN]$ for every $N\in\N$.
%Since  $\lim_{N\to\infty} | %L(\Lambda_N)|/(CN)^2>0$,
We deduce from \eqref{E:Pzerosum} that in order to establish \eqref{E:Pzero}
it suffices to show that
%$$
%\lim_{N\to\infty} \E_{\tilde{m},\tilde{n}\in [CN]}\, %{\bf 1}_{ L([N]\times [N])}(\tilde{m},\tilde{n})\cdot  %T_{1,\tilde{R}_1(\tilde{m},\tilde{n})}F_1\cdot %T_{2,\tilde{R}_2(\tilde{m},\tilde{n})}F_2=0
%$$
$$
\lim_{N\to\infty} \E_{m,n\in [CN]}\, {\bf 1}_{ L(\Lambda_N)}(m,n)\cdot  T_{1,\tilde{R}_1(m,n)}\tilde{F}_1\cdot T_{2,\tilde{R}_2(m,n)}\tilde{F}_2=0
\quad \text{in } L^2(\mu).
$$

\smallskip
		
	{\bf Step 2} (Applying the orthogonality criterion). In Steps 2-3 we suppress the
	tildes.
		Since any function in $L^\infty(\mu)$ is a linear combination of two functions in $L^\infty(\mu)$ with norm pointwise equal to $1$, 	we can assume that $|F_2(x)|=1$ for all $x\in X$.
		Since the action $(X,\CX,\mu,T_{2,n})$ is finitely generated, there exist  a subset $P_0$ of the primes with $\sum_{p\in P_0}\frac{1}{p}=\infty$ and $p_0\in P_0$, such that
		\begin{equation}\label{E:P0}
			T_{2,p}=T_{2,p_0} \quad \text{for all } p\in P_0.
		\end{equation}
Using Lemma~\ref{L:Katai}, it suffices to show that for all $p,q,p',q'\in P_0$ with
$p/q\neq p'/q'$ the following limit vanishes %The terms are zero unless
%$m\leq CN/\max(p,p')$ and $n\leq %CN/\max(q,q')$, so changing the
%normalization from this rectangle to $[CN]^2$ %is harmless for fixed primes
\begin{multline*}
\lim_{N\to\infty}  \E_{m,n\in[CN]} \,
{\bf 1}_{L(\Lambda_N)}(pm,qn)\cdot {\bf 1}_{ L(\Lambda_N)}(p'm,q'n)
\\
\int  T_{1,R_1(pm,qn)}F_1\cdot  T_{1,R_1(p'm,q'n)}\overline{F_1}\cdot  T_{2,p^kq^lm^kn^l}F_2\cdot T_{2, (p')^{k} (q')^lm^kn^l} \overline{F_2}
\, d\mu=0.
\end{multline*}
Using \eqref{E:P0} and the fact that $T_{2,n}$ is a multiplicative action, we get that $T_{2,p^kq^lm^kn^l}F_2=T_{2, (p')^k(q')^lm^kn^l}F_2$ for all $p,q,p',q'\in P_0$, $m,n\in\N$. Combining this
with our assumption 	$|F_2(x)|=1$ for all $x\in X$, we deduce that 	
$$
T_{2,p^kq^lm^kn^l}F_2\cdot T_{2, (p')^k(q')^lm^kn^l} \overline{F_2}=1\quad \text{for all } p,q,p',q'\in P_0, \, m,n\in \N.
$$

Thus, it suffices to show that 	if  $p,q,p',q'\in P_0$   satisfy $p/q\neq p'/q'$, then
$$
\lim_{N\to\infty}  \E_{m,n\in[CN]} \, {\bf 1}_{S_{A,N,p,q,p',q'}}(m,n)
\cdot \int  T_{1,R_1(pm,qn)}F_1\cdot  T_{1,R_1(p'm,q'n)}\overline{F_1}
\, d\mu=0,
$$
where 	
$$
S_{A,N,p,q,p',q'}:=\{(m,n)\in \N^2 \colon (pm,qn) \in L(\Lambda_N) \text{ and }  (p'm,q'n) \in  L(\Lambda_N)  \}.
$$
Equivalently, it suffices to show that
$$
\lim_{N\to\infty}  \E_{m,n\in[CN]} \, {\bf 1}_{S_{A,N,p,q,p',q'}}(m,n)
\cdot\int  T_{1,R_1(pm,qn)/R_1(p'm,q'n)}F_1\cdot  \overline{F_1}
\, d\mu=0.
$$
Using \cref{C:RAZA}  and following  the notation used there, we get
\begin{multline*}
{\bf 1}_{S_{A,N,p,q,p',q'}}(m,n)={\bf 1}_{A_{p,q,p',q',N}}(m,n)\cdot
\\ \E_{(q_1,q_2)\in Z_{A'}}\, e((pm-m_0)q_1+(qn-n_0)q_2) \cdot \E_{(q_1,q_2)\in Z_{A'}}\, e((p'm-m_0)q_1+(q'n-n_0)q_2),
\end{multline*}
for some $m_0,n_0\in \Z_+$ and convex subsets   $A_{p,q,p',q',N}$ of  $\R_+^2$.  So it  suffices to show that if $K_N$, $N\in \N$, are convex subsets of $\R_+^2$ contained in the intersection of the two supports, where both iterates are positive,  and $\alpha,\beta\in \Q$, then
\begin{equation}\label{E:neededR}
\lim_{N\to\infty}  \E_{m,n\in[CN]} \, {\bf 1}_{K_N}(m,n)\cdot e(m\alpha+n\beta)
\cdot\int  T_{1,R_1(pm,qn)/R_1(p'm,q'n)}F_1\cdot  \overline{F_1}
\, d\mu=0.
\end{equation}

{\bf Step 3} (End of proof when $F_1\in X_{1,a}$).
Since $p,q,p',q'\in P_0$ satisfy $p/q\neq p'/q'$  and $R_1(m,n)$ is not of the form $c\, m^kn^lR^r(m,n)$ for any    $c\in \Q_+, k,l\in \Z$, rational function $R\in \Q(m,n)$, and  $r\geq 2$, we deduce from \cref{L:Sr} that  $R_1(pm,qn)/R_1(p'm,q'n)$ is not of the form $ c\, R^r(m,n)$ for any non-zero $c\in \Q$, rational function $R\in \Q(m,n)$, and integer $r\geq 2$.
Using this and since $F_1\in X_{1,a}$, we get that \eqref{E:neededR}  follows from \cref{P:aperiodicP0} and its signed-form extension, which holds
for general multiplicative actions (not necessarily finitely generated).

\smallskip

{\bf Step 4} (End of proof when $F_2\in X_{2,a}$).	We return to the original coordinates and write
$R_1=c_1\prod_{j=1}^s L_j^{k_j}$.
	Finally, we show that if  $F_2\in X_{2,a}$, then \eqref{E:Pzero} holds. As we just showed, \eqref{E:Pzero}
	holds if $F_1\in X_{1,a}$. Since $X_{1,p}=X_{1,a}^\bot$, it suffices to show that \eqref{E:Pzero} holds if
	  $F_1\in X_{1,p}$.
	  In this case, we use the concentration estimates of  part~\eqref{I:concseclinfinA} of \cref{P:ConcRfg} below
	  and argue as in \cref{SS:A1}. We deduce that    it suffices to show that for every sufficiently large $K\in \N$ we have
	$$
	\lim_{N\to\infty}\E_{(a,b)\in S_{K;L_1,\ldots, L_s}} A_{1,K,a,b}\cdot {\bf 1}_\Lambda(a,b)\cdot
	\E_{m,n\in[N/Q_K]}\,T_{2,R_2(Q_Km+a,Q_Kn+b)}F_2=0
	$$
in $L^2(\mu)$, 	where (it follows from \cref{P:convergelinear} that  the limit below exists)
$$
A_{1,K,a,b}:=\lim_{N\to\infty}\E_{n_1,\ldots,n_s\in[N]}\, 
T_{1,c_1\prod_{j=1}^s(Q_Kn_j+L_j(a,b))^{k_j}}F_1.
$$
	$Q_K$  and  $S_{K; L_1,\ldots, L_s}$  are as in \eqref{E:defQKSK} and  \eqref{E:defSKL}, respectively, and we used that for all sufficiently  large  $K\in \N$ we have
	 ${\bf 1}_{\Lambda}(Q_Km+a,Q_Kn+b)= {\bf 1}_{\Lambda}(a,b)$ for all $m,n\in\N$.
	So it suffices to show that for every $Q\in \N$ and $a,b\in \Z_+$ we have
		$$
	\lim_{N\to\infty}
	\E_{m,n\in[N]}\, T_{2,R_2(Qm+a,Qn+b)}F_2=0.
	$$	
	 Since $R_2$ is not of the form $ c\, R^r(m,n)$ for any non-zero
	 $c\in \Q_+$, rational polynomial $R$,
	 and  $r\geq 2$,
	 this follows from \cref{P:aperiodicP0} and the fixed affine-grid extension stated after it. 	This completes the proof.
	\end{proof}

\subsection{Concentration results-Finitely generated case}
We give some concentration results tailored to our needs; part~\eqref{I:concseclinfinA} is used for convergence results  (and was already used in the proof of Step~4 of \cref{P:Characteristic2}) and part~\eqref{I:concseclinfinB} is used for recurrence results.
\begin{proposition}\label{P:ConcRfg}
	Let $(X,\CX, \mu,T_n)$ be a finitely generated multiplicative action, $F \in X_p$,
	$L_1,\ldots, L_\ell\in \Z[m,n]$ be  non-trivial linear forms with non-negative coefficients, $c\in \Q^*_+,k_1,\ldots, k_\ell\in \Z$, and
		$
		R(m,n):=c\, \prod_{j=1}^\ell L_j^{k_j}(m,n).
	$
	For $Q,N\in \N$ and $a,b\in \Z_+$, let
	\begin{equation}\label{E:AQNrj}
		A_{Q,N,a,b}(F):= \E_{n_1,\ldots, n_\ell\in [N]}\,T_{ c\, \prod_{j=1}^\ell (Qn_j+L_j(a,b))^{k_j}} F.
	\end{equation}
	Then the following properties hold:
	\begin{enumerate}
		\item \label{I:concseclinfinA}
		The limit
		$$
		A_{Q,a, b}(F):= \lim_{N\to\infty}A_{Q,N,a,b}(F)
		$$
		%%\end{equation}
		exists for every $Q\in \N$, $a,b\in \Z_+,$  and
		\begin{equation}\label{E:concsevlin}		
			\lim_{K\to\infty} 	\limsup_{N\to\infty} \max_{(a,b)\in S_{K;L_1,\ldots, L_\ell}}
			\E_{m,n\in[N]} \norm{  T_{R(Q_Km+a,Q_Kn+b)}F-  A_{Q_K,a,b}(F) }_{L^2(\mu)}=0,
			\end{equation}
		where  $Q_K$  and  $S_{K; L_1,\ldots, L_\ell}$  are as in \eqref{E:defQKSK} and  \eqref{E:defSKL}, respectively.
		
		\item  \label{I:concseclinfinB} If $a,b\in \Z_+$ are such that $R(a,b)>0$, we have
		\begin{equation}\label{E:concsevlinQ}
			\lim_{K\to\infty} 	\limsup_{N\to\infty} \max_{Q\in \Phi_K}
			\E_{m,n\in [N]} \norm{  T_{R(Qm+a,Qn+b)}F-  T_{R(a,b)}F }_{L^2(\mu)}=0.
		\end{equation}
	%	Moreover, if $R_1,R_2$ are rational %polynomials that factor linearly, and  %$a,b,a',b'\in \Z$ are such that $R_1(a,b)\cdot %R_2(a',b')>0$, then   for every $k\in \Z$ we have
	%	\begin{multline}\label{E:concsevlinQ2}
	%		\lim_{K\to\infty} 	\limsup_{N\to\infty} %\max_{Q\in \Phi_K} 	\E_{m,n\in [N]}\\
	%		\norm{  T_{R_1(Qm+a,Qn+b)\cdot %R_2(Q^2m+kQ+a',Q^2n+b')}F-  T_{R_1(a,b)\cdot %R_2(a',b')}F }_{L^2(\mu)}=0.
	%	\end{multline}
	\end{enumerate}
\end{proposition}
\begin{proof}
	We prove \eqref{I:concseclinfinA}.
	The existence of the  limit	$\lim_{N\to\infty}A_{Q,N,a,b}(F)$
	follows immediately from \cref{P:convergelinear}
	and we also get from this result that  the next two limits exist in $L^2(\mu)$  and we have the identity
	\begin{equation}\label{E:iterated}
		\lim_{N\to\infty}A_{Q,N,a,b}(F)=\lim_{N_1\to\infty}\E_{n_1\in [N_1]}\cdots \lim_{N_\ell\to\infty}\E_{n_\ell\in [N_\ell]}\,T_{ c\, \prod_{j=1}^\ell (Qn_j+L_j(a,b))^{k_j}} F.
	\end{equation}

Next, we establish \eqref{E:concsevlin}. Suppose first that $\ell=1$.
For the action $S_n:=T_{n^{k_1}}$, the function $F$ is still pretentious,
and the extended residue range following \cref{P:Concentrationfg} applies to
$L_1(a,b)\leq C Q_K$. Hence, using \eqref{E:fconcfingen''} of
\cref{P:Concentrationfg} with $T_cF$ in place of $F$, we get
	$$
	\lim_{K\to\infty} 	\limsup_{N\to\infty} \max_{(a,b)\in S_{K;L_1}} \\
	\E_{n\in[N]} \norm{  T_{c\, (Q_Kn+L_1(a,b))^{k_1}}F-  A_{Q_K,a,b}(F) }_{L^2(\mu)}=0.
	$$
	Using \cref{L:lN} for $v_{K,a,b}(n):= T_{c\, (Q_Kn+L_1(a,b))^{k_1}}F$, $n\in\N$,
	we deduce that
	$$
	\lim_{K\to\infty} 	\limsup_{N\to\infty} \max_{(a,b)\in S_{K;L_1}} \\
	\E_{m,n\in[N]} \norm{  T_{c\, (Q_KL_1(m,n)+L_1(a,b))^{k_1}}F-  A_{Q_K,a,b}(F) }_{L^2(\mu)}=0.
	$$	
	Since $Q_KL_1(m,n)+L_1(a,b)=L_1(Q_Km+a,Q_Kn+b)$ we get that \eqref{I:concseclinfinA} holds for $\ell=1$. The general case follows  from the $\ell=1$ case, using a telescoping argument, and  \eqref{E:iterated}.
	
	We prove \eqref{I:concseclinfinB}.
	Using \eqref{E:spectralid}, \cref{L:finspectral},  and Fatou's lemma twice, it suffices to show that for any pretentious finitely generated completely multiplicative function $f\colon \N\to \S^1$ we have
	\begin{equation} \label{E:concsevlinQ'}
		\lim_{K\to\infty} 	\limsup_{N\to\infty} \max_{Q\in \Phi_K}
		\E_{m,n\in [N]}  |f(R(Qm+a,Qn+b))- f(R(a,b)) |=0.
	\end{equation}
	Since $f$ is pretentious and finitely generated, we have that $f\sim \chi$ for some Dirichlet character $\chi$ with period $q\in \N$
	(see for example \cite[Lemma~B.3]{Cha24}).
	Since $R=c\,\prod_{j=1}^\ell L_j^{k_j}$ and  $L_j(a,b)>0$ for each occurring factor, because $a,b\geq0$ and $R(a,b)$ is
	defined and nonzero, we get by using \eqref{E:fconcfingenr} of \cref{C:concentrationfg}
	with  $L_j(a,b)$ in place of $b$ and \cref{L:lN},	that
	\begin{equation}\label{E:KNQK}
	\lim_{K\to\infty} 	\limsup_{N\to\infty} \max_{Q\in \Phi_K} \E_{m,n\in[N]}\big|f(L_j(Qm+a,Qn+b))- %\epsilon_j\, 
	f(L_j(a,b)) \big|=0, \quad j\in [\ell],
	\end{equation}
	%where $\epsilon_j=1$, unless $L_j(a,b)<0$ and %$\chi(q-1)=-1$, in which case $\epsilon_j=-1$.
	%Using this, and since $R(a,b)>0$ implies that
	%$\sum_{j\in [\ell]\colon L_j(a,b)<0}k_j$ is  %even, we have $\prod_{j=1}^\ell %\epsilon_j^{k_j}=1$. 
	Keeping in mind that $R=c\,\prod_{j=1}^\ell L_j^{k_j}$ and  combining the previous facts with a telescoping argument, we get  that \eqref{E:concsevlinQ'} holds.
%	The same argument allows us to prove   %\eqref{E:concsevlinQ2}. Indeed,  %\cref{C:concentrationfg} and \cref{L:lN}  give that %\eqref{E:KNQK} holds if we replace $L_j(Qm+a,Qn+b)$ %by  $L_j(Q^2m+kQ+a,Q^2n+b)$ for any $k\in \Z$ and %$j\in[\ell]$, and we can complete the proof as %before. 	
\end{proof}

	\subsection{Proof  of part~\eqref{I:MainB1} of \cref{T:MainB}}\label{SS:B1}
		Recall that $R_2=c\, L_1^k\cdot L_2^l$ for some  independent linear forms $L_1,L_2$ and non-zero $c\in \Q^*_+,k,l\in \Z$. Let
			$
		R_1=c_1\, \prod_{j=1}^s L_j^{k_j},
		$
		where   $c_1\in \Q_+,k_j\in \Z,$  and $L_j$ are non-trivial linear forms for $j=1,\ldots, s$.

	%	Suppose first that $R_2=c\,R^r$ for some %$c\in \Q_+$, rational polynomial $R$, and %$r\geq 2$ that  we assume to be maximal; so $R$ %is not of the form $ c'\,(R')^{r'}$ for any  %$c'\in \Q_+$, rational polynomial $R'$, and %$r'\geq 2$.
	%	 We let $T'_{2,n}:=T_{2,n^r}$, $n\in \N$. %Then
	%	$(X,\CX,\mu,T'_{2,n})$ is also a finitely %generated multiplicative action, and the %averages in \eqref{E:R1R2Converge} take the form
	%	$$
	%	\E_{m,n\in [N]}\, {\bf 1}_\Lambda(m,n)\cdot %T_{1,R_1(m,n)}F_1\cdot T'_{2,R(m,n)}F'_2,
	%	$$
	%	where $F'_2:=T_cF_2$.
	%	Note also that our hypothesis for $R_1,R_2$, %transfers to  the rational polynomials $R_1,R$.
	%	Using the same reduction for the rational polynomial  $R_2$ %if needed,

By assumption,  the rational polynomial $R_2$ is not of the form
	$c'R^r$ for any $c'\in\Q_+^*$, rational polynomial $R$, and $r\geq2$.
	Using  \cref{P:Characteristic2} (our additional assumption about $R_2$ is needed here) and 	arguing as in \cref{SS:A1}, we get that it suffices to prove mean convergence for the averages
	$$
	\E_{m,n\in[N]} \, A(m,n)
	$$
	where
	$$
	A(m,n):={\bf 1}_\Lambda(m,n)\cdot T_{1,R_1(m,n)}F_{1,p}\cdot T_{2,R_2(m,n)}F_{2,p}, \quad  m,n\in \N.
	$$
	Moreover, we have
	\begin{equation}\label{E:AAKN'}
		\lim_{K\to\infty}\limsup_{N\to\infty}\norm{\E_{m,n\in[N]}\, A(m,n)-	A_{K,N}}_{L^2(\mu)}=0,
	\end{equation}
	where
	\begin{equation}\label{E:AKN'}
	A_{K,N}:=\E_{(a,b)\in S_{K;L_1,\ldots, L_s}}\E_{m,n\in[N/Q_K]}\,  A(Q_Km+a,Q_Kn+b),
\end{equation}
and  $Q_K$  and  $S_{K; L_1,\ldots, L_s}$  are as in \eqref{E:defQKSK} and  \eqref{E:defSKL}, respectively.
Note also that for sufficiently large  $K\in \N$ we
have ${\bf 1}_{\Lambda}(Q_Km+a,Q_Kn+b)={\bf 1}_{\Lambda}(a,b)$ for all $m,n\in\N$.

Using part~\eqref{I:concseclinfinA} of \cref{P:ConcRfg}
  we get  that  the limits
\begin{equation}\label{E:AK1}
	A_{1,K,a,b}:=	\lim_{N\to\infty} \E_{n_1,\ldots, n_s\in [N]}\,  T_{1,c_1\prod_{j=1}^s(Q_Kn_j+L_j(a,b))^{k_j}}F_{1,p}
	\end{equation}
	and
	\begin{equation}\label{E:AK2}
		A_{2,K,a,b}:=	\lim_{N\to\infty}\E_{n_1, n_2\in [N]}\ T_{2,c\, (Q_Kn_1+L_1(a,b))^k\cdot (Q_Kn_2+L_2(a,b))^l}F_{2,p}
	\end{equation}
	exist in $L^2(\mu)$, and 	
	\begin{equation}\label{E:AjKN'}
	\lim_{K\to\infty}\limsup_{N\to\infty}\norm{A_{K,N}- A_K}_{L^2(\mu)}=0,
\end{equation}
where
$$
A_K:=
\E_{(a,b)\in S_{K;L_1,\ldots, L_s}}\, {\bf 1}_{\Lambda}(a,b)\cdot  A_{1,K, a,b}\cdot  A_{2,K, a,b}.
$$
Combining \eqref{E:AAKN'} and \eqref{E:AjKN'}, we deduce
that  the sequence $(\E_{m,n\in[N]}\, A(m,n))_{N\in\N}$  is Cauchy and therefore converges in $L^2(\mu)$.

	\subsection{Proof  of part~\eqref{I:MainB2} of \cref{T:MainB}}\label{SS:B2}
We keep the notation and observations from the proof of part~(i). 
%The only additional observation one has to make,  is %that if $R_2=c\, R^r$, then our assumption
%	$R_2(m_0,n_0)=1$ implies   %$R_2(m,n)=(R(m,n)/R(m_0,n_0))^r$ and the rational %polynomial  $R'(m,n):=R(m,n)/R(m_0,n_0)$ also %satisfies $R'(m_0,n_0)=1$.

 Let $\varepsilon>0$ and  $F:={\bf 1}_A$.
 Let $Q\in \N$, which  will be determined later, and $m_0,n_0\in \Z_+$ be as in the statement of   part~\eqref{I:MainB2} of \cref{T:MainB}.
 It suffices to show that
  \begin{equation}\label{E:lowerneeded}
\lim_{N\to\infty} \E_{m,n\in[N]}\int F \cdot T_{1,R_1(Qm+m_0,Qn+n_0)}F\cdot T_{2,R_2(Qm+m_0,Qn+n_0)}F\, d\mu\geq \Big(\int F\, d\mu\Big)^3	-\varepsilon,
 \end{equation}
 where the limit exists by  part~\eqref{I:MainB1} (this is one of the reasons why it helps to have the $2$-dimensional grid $\Lambda$ in  part~\eqref{I:MainB1}).

 Using the decomposition result of \cref{T:Decomposition} for $j=1,2,$ we write
$$
F=F_{j,p}+F_{j,a},
$$
where $F_{j,p}=\E(F|\CX_{j,p})\in X_{j,p}$ and  $F_{j,a}\in X_{j,a}$.  Using \cref{P:Characteristic2} we get that \eqref{E:lowerneeded} would follow if we show that
 \begin{equation}\label{E:lowerneeded'}
\lim_{N\to\infty} \E_{m,n\in[N]}\int F \cdot T_{1,R_1(Qm+m_0,Qn+n_0)}F_{1,p}\cdot T_{2,R_2(Qm+m_0,Qn+n_0)}F_{2,p}\, d\mu\geq \Big(\int F\, d\mu\Big)^3	-\varepsilon.
\end{equation}
We consider two cases.

\smallskip

{\bf Case 1.} Suppose that  $R_j(m_0,n_0)=1$ for $j=1,2$.  Using part~\eqref{I:concseclinfinB} of   \cref{P:ConcRfg}
	 we get that
\begin{equation}\label{E:B2Conc}
\lim_{K\to\infty}\limsup_{N\to\infty} \max_{Q\in \Phi_K}\E_{m,n\in[N]}\norm{T_{j,R_j(Qm+m_0,Qn+n_0)}F_{j,p}-F_{j,p}}_{L^2(\mu)}=0, \quad j=1,2.
\end{equation}
For $K\in \N$, let $Q_K\in \Phi_K$ be arbitrary. It follows from the above using a telescoping argument, that  the iterated limit (the limit below  as $N\to \infty$ exists  by part~\eqref{I:MainB1})
$$
\lim_{K\to\infty} \lim_{N\to\infty} \E_{m,n\in[N]}\int F \cdot T_{1,R_1(Q_Km+m_0,Q_Kn+n_0)}F_{1,p}\cdot T_{2,R_2(Q_Km+m_0,Q_Kn+n_0)}F_{2,p}\, d\mu
$$
 is equal to
 $$
 \int F \cdot F_{1,p}\cdot F_{2,p}\, d\mu=\int F \cdot \E(F|\CX_{1,p}) \cdot \E(F|\CX_{2,p})\, d\mu\geq \Big(\int F\, \, d\mu\Big)^3,
 $$
 where the lower bound follows from \cref{T:Chu}.

Combining the above,  we deduce that there exists $Q\in \N$ such that
\eqref{E:lowerneeded'} holds.

\smallskip
	
	{\bf Case 2.} Suppose that $(m_0,n_0)$ is a simple zero of $R_1$ and $R_2(m_0,n_0)=1$. Then
	$$
	R_1(m,n)=L_1(m,n)\cdot \tilde{R}_1(m,n)
	$$
	for some linear form $L_1$ that satisfies $L_1(m_0,n_0)=0$ and rational polynomial $\tilde{R}_1$ that factors linearly and satisfies $\tilde{R}_1(m_0,n_0)\neq 0$.

%	\smallskip
	
%	{\bf Case 2a.} Suppose first  that
%	$\tilde{R}_1(m_0,n_0)>0$.
%	$$
%	R_1(Qm+m_0,n+n_0)=Q\, L_1(m,n)\cdot \tilde{R}_1(Qm+m_0,Qn+n_0)
%	$$
Since every linear factor of $R_1$ has non-negative coefficients and
$m_0,n_0\in \Z_+$, it takes a non-negative value at $(m_0,n_0)$; hence every
remaining factor, being nonzero there, is strictly positive. Put $r_0:=\widetilde R_1(m_0,n_0)>0$.
	Using \eqref{E:concsevlinQ} in   \cref{P:ConcRfg} we get
	%  and since $r_0:=\tilde{R}_1(m_0,n_0)$ is positive, we get that
$$
\lim_{K\to\infty}\limsup_{N\to\infty} \max_{Q\in \Phi_K}\E_{m,n\in[N]}\norm{T_{1,\tilde{R}_1(Qm+m_0,Qn+n_0)}F_{1,p}-T_{1,r_0}F_{1,p}}_{L^2(\mu)}=0.
$$

	Let
	\begin{equation}\label{E:tildeF}
		\tilde{F}_{1,p}:=\lim_{N\to\infty}\E_{m,n\in[N]}\, T_{1,r_0L_1(m,n)}F_{1,p},
	\end{equation}
	where  the limit exists by part~\eqref{I:MainA1}  of \cref{T:MainA} since the action $(X,\CX, \mu,T_{1,n}) $ is finitely generated.
Since  $L_1(Qm+m_0,Qn+n_0)=QL_1(m,n)$,   for $G\in L^\infty(\mu)$, which   will be determined later,  we get, using the bounded convergence theorem, that  for every $Q\in \N$ we have
	%%\begin{equation}\label{E:limG}
	$$
	\lim_{N\to\infty} \E_{m,n\in[N]}\int T_{1,r_0L_1(Qm+m_0,Qn+n_0)}F_{1,p}\cdot G\, d\mu = \int T_{1,Q}\tilde{F}_{1,p}\cdot G \, d\mu.
	$$
	%%\end{equation}
	Moreover, we get by  \cref{L:multaverid} that
	$$
	\lim_{K\to\infty}\E_{Q\in \Phi_K}\int T_{1,Q}\tilde{F}_{1,p}\cdot G \, d\mu  =\int  \E(\tilde{F}_{1,p}|\CI_{T_1})\cdot G\, d\mu =
	\int  \E(F|\CI_{T_1})\cdot G\, d\mu,
	$$
	where $(\Phi_K)_{K\in\N}$ is a multiplicative  F\o lner  sequence in $\N$, and we used that
	$$
	\E(\tilde{F}_{1,p}|\CI_{T_1})=\E(F_{1,p}|\CI_{T_1})=\E(F|\CI_{T_1}),
	$$
	where the comments immediately after
	\eqref{E:IT1}  justify  these equalities.
	
	Applying the isometry $T_{1,QL_1(m,n)}$ to the difference in the preceding
	concentration estimate and combining the above, we get that
		\begin{equation}\label{E:B210}
		\lim_{K\to\infty} \E_{Q\in \Phi_K}\lim_{N\to\infty}
	\E_{m,n\in[N]} \int  T_{1,R_1(Qm+m_0,Qn+n_0)}F_{1,p}\cdot G\, d\mu=	\int  \E(F|\CI_{T_1})\cdot G\, d\mu.
	\end{equation}
		Using first  \eqref{E:B2Conc} for $j=2$ (here we use that $R_2(m_0,n_0)=1$ and  that $(X,\CX, \mu,T_{2,n})$ is finitely generated)  and   then  \eqref{E:B210} for $G:=F\cdot F_{2,p}$,  we deduce that the next limit (as $K\to \infty$) exists
	(the limit below as $N\to \infty$ exists by part~\eqref{I:MainB1} for every  $Q\in \N$)
$$
\lim_{K\to\infty} \E_{Q\in \Phi_K}\lim_{N\to\infty}
\E_{m,n\in[N]} \int F\cdot T_{1,R_1(Qm+m_0,Qn+n_0)}F_{1,p}\cdot T_{2,R_2(Qm+m_0,Qn+n_0)}F_{2,p}
\, d\mu,
$$
and
is equal to
	$$
 \int F \cdot\E(F|\CI_{T_1})\cdot F_{2,p}\, d\mu=	\int F\cdot  \E(F|\CI_{T_1}) \cdot \E(F|\CX_{2,p})\, d\mu\geq \Big(\int F\, d\mu\Big)^{3},
	$$
	where the lower bound again follows  from  \cref{T:Chu}.
	
	Combining the above,  we deduce that there exists $Q\in \N$ such that
	\eqref{E:lowerneeded'} holds.
	
%	\smallskip
	%	{\bf Case 2b.}
%	 Suppose now that  $\tilde{R}_1(m_0,n_0)<0$. In this case, %if we average over the grid $\{(Qm+m_0,Qn+n_0)\colon m,n\in %\N\}$ we cannot use \eqref{E:concsevlinQ} in   %\cref{P:ConcRfg}.  To overcome this problem, we change our %averaging grid in a way that will  be described shortly and use
%	 \eqref{E:concsevlinQ2} in   \cref{P:ConcRfg} instead.
	
%	 We have  $L_1(1,0)\neq 0$ or $L_1(0,1)\neq 0$.  We assume that the former is true, the other case can be treated similarly.
%	  Suppose also that
%	 $L_1(1,0)>0$, the argument is similar if  $L_1(1,0)<0$.
%	 We repeat the argument in Case 2a,  but  we average over %the grid  $\{(Q^2m-Q+m_0,Q^2n+n_0)\colon m,n\in\N\}$.
%	 Note that since $L_1(m_0,n_0)=0$  we have
%	$$
%	R_1(Q^2m-Q+m_0,Q^2n+n_0)=QL_1(Qm-1,Qn)\cdot 	%\tilde{R_1}(Q^2m-Q+m_0,Q^2n+n_0),
%	$$
%	and, crucially, $r_0:=L_1(-1,0)\cdot \tilde{R_1}(m_0,n_0)$ %is positive.
%	 	Using   \eqref{E:concsevlinQ2} in   \cref{P:ConcRfg}, %(with $L_1$  and $\tilde{R}_1$ instead of $R_1,R_2$, %respectively, and $a:=-1$, $b:=0$, $a':=m_0$, $b':=n_0$),  we %get that
%	 $$
%	 \lim_{K\to\infty}\limsup_{N\to\infty} \max_{Q\in %\Phi_K}\E_{m,n\in[N]}\norm{T_{1,R_1(Q^2m-Q+m_0,Q^2n+n_0)}F_{1,p%}-T_{1,r_0 Q}F_{1,p}}_{L^2(\mu)}=0,
%	 $$
%	and  since $R_2(m_0,n_0)=1$, we also get that
%	  $$
%	 \lim_{K\to\infty}\limsup_{N\to\infty} \max_{Q\in %\Phi_K}\E_{m,n\in[N]}\norm{T_{1,R_2(Q^2m-Q+m_0,Q^2n+n_0)}F_{2,p%}-F_{2,p}}_{L^2(\mu)}=0.
%	 $$
%	  The rest of the argument is identical to the case %$\tilde{R}_1(m_0,n_0)>0$, so we omit it.

	  	\subsection{Proof  of part~\eqref{I:MainB3} of \cref{T:MainB}}
The concentration result we need uses averages on  sets defined as in the next lemma.

\begin{lemma}\label{L:SdR}
	Let $R$ be a non-constant rational polynomial that factors linearly and
	has degree $0$. Suppose that
	$R(a,b)=1$ for some $a,b\in \Z_+$. 
	Then for every $\delta>0$ the set
	\begin{equation}\label{E:SdeltaR}
		S_{\delta,R}:=
		\big\{(m,n)\in\N^2:\ |\log R(m,n)|\leq\delta\big\}
			\end{equation}
	has positive lower density. Moreover, up to finitely many boundary rays,
	it is the set of lattice points in a finite union of convex sectors.
\end{lemma}
\begin{proof}
 We have $
	R(m,n)=r(m/n)$, where 
	 $r(s):=R(s,1)$
	is positive and non-constant on $(0,\infty)$. Each of the
	equations $r(s)=e^\delta$ and $r(s)=e^{-\delta}$ has finitely many
	solutions, so
	$
	e^{-\delta}\leq r(s)\leq e^\delta
	$
	describes finitely many intervals and isolated points in the slope
	variable. The assumption $R(a,b)=1$ and continuity imply that one of these 
	intervals has non-empty interior (possibly with endpoint $0$ or
	$\infty$). The corresponding sector has positive area in $(0,1]^2$,
	and arguing exactly as in the proof of \cite[Lemma~6.4]{FKM24},
	Riemann sums give that $S_{\delta,R}$ has positive lower density.
\end{proof}

	Here is the precise statement of the concentration result:
	 \begin{proposition}\label{P:ConcRgeneral}
	 	Let $(X,\CX, \mu,T_n)$ be a general multiplicative action and $F \in X_p$. Let also $R$ be a rational polynomial that factors linearly and has degree $0$,  
	 	 and let $a,b\in\Z_+$ be such that
	 	$R(a,b)=1$.
	 Then
	 	\begin{equation}\label{E:concsevlingenB}
	 	\lim_{\delta\to 0^+}	\limsup_{K\to\infty} 	\limsup_{N \to\infty} \max_{Q\in \Phi_K}
	 \E_{(m,n)\in S_{\delta,R,N}} \norm{  T_{R(Qm+a,Qn+b)}F-  F}_{L^2(\mu)}=0,	
	 	\end{equation}
		where    $\Phi_K$ is as 
		in \eqref{E:PhiK}, $S_{\delta,R}$ is as in \eqref{E:SdeltaR}, and
		$S_{\delta,R,N}:=S_{\delta,R}\cap [N]^2$,  $\delta>0$, $N\in \N$.
	 \end{proposition}
 \begin{remark}
 In contrast to  the finitely generated case (see part \eqref{I:concseclinfinB} of  \cref{P:ConcRfg}), 	if $R(a,b)\neq 1$, we cannot infer concentration at $T_{R(a,b)}F$, even if $R(a,b)$ is positive. The reason is that, for example, if  $f(n):=n^{it}$, then  $f(R(Qm+a,Qn+b))$ concentrates, using the averaging  in \eqref{E:concsevlingenB}, at $1$ rather than at $f(R(a,b))$.
 \end{remark}
	 \begin{proof}
	 	The case $R\equiv1$ is immediate, with $S_{\delta,R}=\N^2$.
	 	Otherwise, we combine proportional factors and omit zero exponents.
	Our assumption gives that $R$ has the form
	$	R:=  c\, \prod_{j=1}^sL_j^{k_j},$
	where  	$L_1,\ldots, L_s\in \Z[m,n]$ are non-trivial linear forms with non-negative coefficients, $c\in \Q^*_+$,  and  $k_1,\ldots, k_s\in \Z$ that satisfy $\sum_{j=1}^s k_j=0.$
	
	 	Using \eqref{E:spectralid}, Fatou's lemma several times, and since $\sigma_F$ is supported on $\CM_p$,  it suffices to show that for any pretentious completely multiplicative function $f\colon \N\to \S^1$ we have
	 		\begin{equation}\label{E:concsevlingenB'}
	 		\lim_{\delta\to 0^+}	\limsup_{K\to\infty} 	\limsup_{N \to\infty} \max_{Q\in \Phi_K}
	 		\E_{(m,n)\in S_{\delta,R,N}} | f(R(Qm+a,Qn+b))-  1|=0.\footnote{If  $R(a,b)$ is positive but not necessarily $1$, we  get concentration at
	 			$f(R(a,b))\, (R(a,b))^{-it}$, provided $S_{\delta,R}$ has positive density for every $\delta>0$.}	
	 	\end{equation}
Henceforth, we will assume that $f\sim \chi \cdot n^{it}$ for some $t\in \R$ and Dirichlet character $\chi$ with period $q$. Since $a,b\geq0$ and $R(a,b)=1$, all occurring factors satisfy
$L_j(a,b)>0$.

	%satisfying $\chi(q-1)=-1$; the case where $\chi(q-1)=1$ %can be treated similarly.

We first  claim  that for every occurring linear form $L$   and every $k\in \Z$ we have
\begin{multline}\label{E:fkj}
\limsup_{K\to\infty} 	\limsup_{N \to\infty} \max_{Q\in \Phi_K} \E_{m,n\in  [N]}
| f^k(L(Qm+a,Qn+b))-\\ %(\text{sign}(L(a,b)))^k\, 
f^k(L(a,b))\,(QL(a,b)^{-1}L(m,n))^{ikt}\,\exp(k F_N(f,K))|=0,
\end{multline}
where $F_N(f,K)$ is as in \eqref{E:FNQ}.
The case  $k=1$  follows by combining  \eqref{E:fconc} in  \cref{T:concestlinear}, for $L(a,b)$ instead of $b$, with
\cref{L:lN},  for $v_n:=f(Qn+L(a,b))\, n^{-it}$ and $v_N$  a constant multiple of $\exp( F_N(f,K))$.
For general $k\in \Z$ we use  the $k=1$ case and the elementary estimate  $|a^k-b^k|\ll_{k,C}|a-b|$, which holds if $C^{-1}\leq |a|, |b|\leq C$. 	 For negative $k$, the required lower bound on
$|\exp(F_N(f,K))|$ follows from \eqref{E:=1} for all sufficiently large $K$.	 
	 	
Since $R=c\, \prod_{j=1}^s L_j^{k_j}$, we deduce from \eqref{E:fkj} and the multiplicativity of $f$ that
\begin{multline}\label{E:fkjR}
\limsup_{K\to\infty} 	\limsup_{N \to\infty} \max_{Q\in \Phi_K} \E_{m,n\in  [N]}
| f(R(Qm+a,Qn+b))-\\ f(c)\prod_{j=1}^s %(\text{sign}(L_j(a,b)))^{k_j} 
f^{k_j}
(L_j(a,b))\, (Q L_j(a,b)^{-1}L_j(m,n))^{ik_jt} \exp(k_j F_N(f,K))|=0.
\end{multline}
Note that the subtracted term is equal to
 $$
  %\text{sign}(R(a,b))
  f(R(a,b))\, Q^{i\sum_{j=1}^s k_j\, t} \, ( R(a,b)^{-1} \, R(m,n))^{it}  \exp\Big(\sum_{j=1}^sk_j \ F_N(f,K)\Big).
 $$
 Using that $R(a,b)=1$ to eliminate the  terms involving $R(a,b)$  and  our assumption
$\sum_{j=1}^sk_j=0$ to  eliminate the $Q$'s  and the exponentials, we deduce from \eqref{E:fkjR} that
$$
	\limsup_{K\to\infty} 	\limsup_{N \to\infty} \max_{Q\in \Phi_K} \E_{m,n\in  [N]}
	| f(R(Qm+a,Qn+b))-(R(m,n))^{it}|=0.
$$
 Since by \cref{L:SdR} the set $S_{\delta,R}$ has positive lower density, we can restrict our averaging to this set, hence
	$$
		\limsup_{K\to\infty} 	\limsup_{N \to\infty} \max_{Q\in \Phi_K} \E_{(m,n)\in  S_{\delta,R,N}}
		| f(R(Qm+a,Qn+b))-(R(m,n))^{it}|=0
	$$
holds	for all $\delta>0$. 
Since \eqref{E:SdeltaR} implies that
$$
|R(m,n)^{it}-1|
=
|e^{it\log R(m,n)}-1|
\leq |t|\,|\log R(m,n)|
\leq |t|\delta
$$
for all $(m,n)\in S_{\delta,R}$, we deduce that \eqref{E:concsevlingenB'} holds, as requested.
	 \end{proof}

	\begin{proof}[Proof  of part~\eqref{I:MainB3} of \cref{T:MainB}]
Recall that our assumptions imply that $R_2$ is not of the form
$c\,R^r$ for any $c\in\Q_+^*$, rational polynomial $R$, and $r\geq2$.

Let $\varepsilon>0$, $F:={\bf 1}_A$, and
$F_{j,p}:=\E(F\mid \CX_{j,p})$, $j=1,2$.
Let also $m_0,n_0\in\Z_+$ be such that
$R_j(m_0,n_0)=1$ for $j=1,2$, $S_{\delta,R_1}$ be as in
\eqref{E:SdeltaR}, and
$S_{\delta,R_1,N}:=S_{\delta,R_1}\cap[N]^2$,
$\delta>0$, $N\in\N$.
By \cref{L:SdR}, the set $S_{\delta,R_1}$ has positive lower density
for every $\delta>0$.
Using
\cref{P:Characteristic2}  for appropriate $\Lambda$ (see the second remark, which  also covers the variant with $(m,n)$ restricted to $S_{\delta,R_1}$) and arguing as in the proof of part~\eqref{I:MainB2} of \cref{T:MainB},
it suffices to show that there exist $\delta>0$ and $Q\in \N$  such that
%%that  the limit
\begin{equation}\label{E:lowerneededgen}
	\liminf_{N\to\infty} \E_{(m,n)\in S_{\delta,R_1,N}}\int F \cdot T_{1,R_1(Qm+m_0,Qn+n_0)}F_{1,p}\cdot T_{2,R_2(Qm+m_0,Qn+n_0)}F\, d\mu
	%\geq \Big( \int F\, d\mu\Big)^3-\varepsilon.
\end{equation}
is at least $(\int F\, d\mu)^3	-\varepsilon$ (note that only the second function is pretentious).

We will use that
	\begin{equation}\label{E:concthmB3}
	\lim_{\delta\to 0^+}	\limsup_{K\to\infty} 	\limsup_{N \to\infty} \max_{Q\in \Phi_K}\E_{(m,n)\in S_{\delta,R_1,N}}\norm{T_{j,R_j(Qm+m_0,Qn+n_0)}F_{j,p}-F_{j,p}}_{L^2(\mu)}=0
	\end{equation}
	for $j=1,2$. For $j=1$ this follows from our assumption
	$R_1(m_0,n_0)=1$ and \cref{P:ConcRgeneral}
	(the assumption $\deg(R_1)=0$ is crucial here).
	  For $j=2$  it follows using that   $R_2(m_0,n_0)=1$ and \eqref{E:concsevlinQ} of   \cref{P:ConcRfg}; note that we can restrict our averaging  to the set $S_{\delta,R_1}$ since it has  positive lower density by \cref{L:SdR}.

We use \eqref{E:concthmB3} for $j=1$ and for $K\in \N$ we  let $Q_K\in \Phi_K$ be arbitrary.  We see that   to prove the required lower bound for  \eqref{E:lowerneededgen}, it suffices to show that
	\begin{equation}\label{E:lowerneededgen'}
		\liminf_{\delta\to 0^+}\liminf_{K\to\infty} 	\liminf_{N\to\infty} \E_{(m,n)\in S_{\delta,R_1,N}}\int F \cdot F_{1,p}\cdot T_{2,R_2(Q_Km+m_0,Q_Kn+n_0)}F\, d\mu\geq \Big( \int F\, d\mu\Big)^3.
	\end{equation}
	For fixed $\delta>0$, decompose $S_{\delta,R_1}$ into finitely many
	convex sectors, ignoring boundary rays of density zero. For fixed
	sufficiently large $K$, the substitution
	$$
	(u,v):=(Q_Km+m_0,Q_Kn+n_0)
	$$
	takes each truncated sector to a convex set. Expressing the resulting
	congruence indicators as finite sums of additive characters,
	\cref{P:aperiodicP0} applies to
	$F-F_{2,p}\in X_{2,a}$. The change of normalization is harmless for
	fixed $\delta$ and $K$. Thus, the left side in \eqref{E:lowerneededgen'}  does not  change if we replace
	$T_{2,R_2(Q_Km+m_0,Q_Kn+n_0)}F$ with $T_{2,R_2(Q_Km+m_0,Q_Kn+n_0)}F_{2,p}$ (recall that $R_2$ is not of the form $c\, R^r$ for any $c\in \Q_+$,  rational polynomial $R$, and $r\geq 2$).
	Using  this and  \eqref{E:concthmB3} for $j=2$, we get that the left side in \eqref{E:lowerneededgen'} is equal to
	$$
	\int F \cdot F_{1,p}\cdot F_{2,p}\, d\mu =\int F \cdot \E(F|\CX_{1,p}) \cdot \E(F|\CX_{2,p})\, d\mu \geq \Big(\int F\, \, d\mu\Big)^3,
	$$
	where the lower bound follows from \cref{T:Chu}.
	Combining the above,  we deduce that \eqref{E:lowerneededgen'} holds, as requested.
\end{proof}
	\section{Further directions}\label{S:Problems}
	\subsection{Two  conjectures}\label{SS:conj} The methodology developed in  this article addresses some of the potential difficulties in a more systematic study of multiple recurrence and convergence problems of multiplicative actions. We provide a list of problems that seem to be logical next steps in this endeavor starting with two conjectures that vastly extend the scope of  Theorems~\ref{T:MainA} and \ref{T:MainB} when we deal with a single multiplicative action.
	\begin{conjecture}[Mean convergence]\label{Con1}
		Let $(X,\CX, \mu,T_n)$ be a  finitely generated  multiplicative action   and $R_1,\ldots, R_\ell$ be  rational polynomials that factor linearly. Then for all $F_1,\ldots, F_\ell\in L^\infty(\mu)$ the averages
		$$
		\E_{m,n\in [N]} \, T_{R_1(m,n)} F_1\cdots T_{R_\ell(m,n)}F_\ell
		$$
		converge in $L^2(\mu)$ as $N\to\infty$.
		Furthermore, if all the rational polynomials have degree $0$ and are
		multiplicatively independent modulo constants, that is,
		$\prod_{j=1}^{\ell}R_j^{k_j}$ is nonconstant for every
		$(k_1,\ldots,k_\ell)\in\Z^\ell\setminus\{0\}$,
		then the conclusion holds for all multiplicative actions.
	\end{conjecture}
	For general multiplicative actions, the assumption that all rational polynomials have degree $0$ is necessary for convergence, see the remarks following \cref{T:MainA}.  The multiplicative independence assumption for general actions cannot be
	omitted. A counterexample for the pair $m/n,n/m$ arises from the quadratic
	phase $a(n):=e\big(\frac12\sum_p\alpha_pv_p(n)^2\big)$ on $(\mathbb N,\times)$.
	Indeed, note that for every $k,m,n\in\mathbb N$, we have
	$$
	\overline{a}(kmn)^2a(km^2)a(kn^2)
	=e\Big(\sum_p\alpha_p\big(v_p(m)-v_p(n)\big)^2\Big).
	$$
	Take $\alpha_p=1/\log\log p$ for $p\ge p_0$ and zero otherwise.
It is not hard to show that 	the averages of the right-hand side over $m,n\in[N]$ differ from
	$\big(\E_{n\in[N]}f(n)\big)^2$ by $O(p_0^{-1})$, uniformly in $N$,
	where $f$ is completely multiplicative with $f(p)=e(\alpha_p)$;
	the discrepancy comes only from prime-square divisibility or common
	prime factors. 
	As in example~\eqref{I:Ex3} in \cref{SS:pretentious}, the means of $f$
	have a nonvanishing oscillatory main term. Hence, for sufficiently large
	$p_0$, the averages
	$$
	\E_{m,n\in[N]}
	\big(
	\lim_{K\to\infty}\E_{k\in\Phi_K}
	\overline{a}(kmn)^2a(km^2)a(kn^2)
	\big)
	$$
	do not converge as $N\to\infty$, where $(\Phi_K)$ is any multiplicative
	F{\o}lner sequence in $\mathbb N$. Via the multiplicative Furstenberg correspondence principle, this
	obstruction yields a counterexample to mean convergence.
	%; and in fact the system
	%can be realized as a multiplicative action by %measure-preserving
	%affine transformations of a compact abelian group.

%	$X:=\mathbb T^{\P}\times\mathbb T$, equipped with %Haar measure, by setting
%	$$
%	T_p\big((x_q)_{q\in\P},y\big)
%	:=\Big((x_q+\alpha_p{\bf 1}_{q=p})_{q\in\P},
%	y+x_p+\tfrac12\alpha_p\Big),
%	\qquad p\in\P.
%	$$
	
When  $\ell=1$, for finitely generated actions, the conjecture follows from  part~\eqref{I:MainB1} of  \cref{T:MainB} (which in turn follows  easily from the machinery developed in \cite{FH17, FKM23}), and for general actions
it follows by modifying the proof of \cite[Theorem~1.5]{FH16} in a straightforward way.
These arguments depend on the use of a representation result of Bochner-Herglotz, which only helps when $\ell=1$.
\begin{conjecture}[Multiple recurrence]\label{Con2}
	Let $(X,\CX, \mu,T_n)$ be a finitely generated multiplicative action
	and $R_1,\ldots, R_\ell$ be  rational polynomials that factor linearly. Suppose that there exist  $m_0,n_0\in \Z_+$,   such that
	$R_j(m_0,n_0)=1$ for $j=2,\ldots, \ell$
	and either $R_1(m_0,n_0)=1$ or $(m_0,n_0)$ is a simple zero of $R_1$.
	Then for every  $A\in \CX$ with $\mu(A)>0$ we have
	$$
	\liminf_{N\to \infty}	\E_{m,n\in [N]} \, \mu(A\cap T^{-1}_{R_1(m,n)}A\cap \cdots \cap T^{-1}_{R_\ell(m,n)}A)>0.\footnote{ For finitely generated actions, it is possible to get a  partial result using the Green-Tao theorem.
		For example, let
		$
		R_1(m,n)=m(m+n)$, $R_2(m,n)=(m+2n)(m+3n),
		$
		and let $(T_n)$ be a finitely generated multiplicative action. Since
		$\{T_p:p\in\P\}$ is finite, for some transformation $S$ the set
		$P_S:=\{p\in\P:T_p=S\}$ has positive relative upper density in the
		primes. By \cite{GT08}, there exist $u,h\in\N$ such that
		$u,u+h,u+2h,u+3h\in P_S$. By Poincar\'e recurrence, choose
		$r\in\N$ such that $\mu(A\cap S^{-2r}A)>0$, and set
		$m:=u^r$, $n:=u^{r-1}h$. Then
		$T_{m+jn}=S^r$ for $j=0,1,2,3$, and hence
		$
		T_{R_1(m,n)}=T_{R_2(m,n)}=S^{2r}.
		$
		Thus
		$
		\mu\big(A\cap T_{R_1(m,n)}^{-1}A\cap T_{R_2(m,n)}^{-1}A\big)=\mu(A\cap S^{-2r}A)>0.
		$
		This argument only gives existence of a recurrence pair, and does not
		give the positive lower bound for the averages asserted in the
		conjecture. It also uses finite generation in an essential way and does
		not apply to infinitely generated actions.}
	$$
	Furthermore, if all the  rational polynomials
	have degree $0$ and
	$R_j(m_0,n_0)=1$ for $j=1,\ldots, \ell$  for some  $m_0,n_0\in \Z_+$, then the conclusion holds for all multiplicative actions.

	%\footnote{
	%	To ensure that actions by dilations by $k$-th %powers do not pose obstructions to multiple %recurrence,  we must guarantee  that if    %$c_1(R_1(m,n))^k+\cdots+ c_\ell(R_\ell(m,n))^k=0, %m,n\in\N$, then  $c_1+\cdots+c_\ell=0$. This %follows from our assumption $R_j(m_0,n_0)=1$ for %$j=1,\ldots,\ell$.}
\end{conjecture}
To illustrate the scope of this conjecture, note that it predicts that for any multiplicative action we have the multiple recurrence property 
	$$
	\mu( T^{-1}_{m^2-n^2} A\cap T^{-1}_{m^2-4n^2} A\cap T^{-1}_{m^2-9n^2}A)>0
	$$
	for some $m,n\in\N$ with $m>3n$.
Indeed, after factoring out  $T^{-1}_{m^2-n^2}$ and making  the substitution $m\mapsto m+3n$, one sees that \cref{Con2} applies for  $R_1(m,n):=(m+n)(m+5n)(m+2n)^{-1}(m+4n)^{-1}$
and  $R_2(m,n):=m(m+6n) (m+2n)^{-1}(m+4n)^{-1}$ and $m_0:=1, n_0:=0$.

When $\ell=1$, for finitely generated actions,  the conjecture follows from part~\eqref{I:MainB2} of \cref{T:MainB}, and the part of the conjecture that refers to general actions follows for  $\ell=1$  by modifying the proof of \cite[Theorem~2.2]{FKM23}
in a straightforward way.

To understand the necessity of some of the assumptions made in \cref{Con2}, we note that by considering suitable multiplicative rotations, we get  finitely generated multiplicative actions $(X,\CX,\mu,T_n)$ and sets $A\in \CX$ such that $\mu(A)>0$ and  
	\begin{enumerate}
	%%	\item \label{I:count1}	
	%%	$\mu(T^{-1}_{n} A\cap T^{-1}_{2n} A)=0$ for every $n\in \N$;
	
		\item \label{I:count2} $\mu(A\cap T^{-1}_{2n^2}A)=0$ for every $n\in \N$
		(\cite[Example~3.11]{DLMS23});

		\item \label{I:count3}
		$\mu(T^{-1}_{n} A\cap T^{-1}_{2n} A)=0$ for every $n\in \N$ and
		$\mu(A\cap T^{-1}_{m/n} A\cap T^{-1}_{(3m-n)/n}A)=0$
		for every $m,n\in \N$ with $3m>n$.
	\end{enumerate}	
%	It is easier to describe the examples in the setting of the integers, %where $T_n$ corresponds to the map $x\mapsto nx$ for $x\in \N$. In %cases \eqref{I:count1}  and \eqref{I:count2},  take $A:=\{n\in \N\colon %f(n)=1\}$, where  $f$ is the completely multiplicative function defined %by $f(2):=-1$ and $f(p)=1$ for $p\neq 2$. In case \eqref{I:count3}  take %$A$ of positive multiplicative density such that  the equation $x+y=3z$ %has no solutions in $A$, for example, take $p>5$ to be a prime and let %$A$ be the set of integers whose first non-zero digit in the base $p$ %expansion is $1$ (this is the $1$-level set of some modified Dirichlet %character).
%	In case \eqref{I:count1} the polynomials $n, 2n$ are not independent. 
In case \eqref{I:count2}  the polynomial $R(n):=2n^2$ vanishes at $0$ with a multiplicity greater than $1$. Finally, in case \eqref{I:count3}  we cannot find $m_0,n_0\in \Z$
	that satisfy   the assumption of  \cref{Con2}.
\subsection{Inverse theorem for mixed seminorms and applications}
 \cref{T:Inverse} gives a simple inverse theorem
for the mixed seminorms $\nnorm{\cdot}_{U^s}$, which covers all finitely generated multiplicative actions.
A  similar inverse theorem fails for general multiplicative actions, in particular,  none of the properties \eqref{I:inv1}-\eqref{I:inv3} of \cref{T:Inverse} implies property~\eqref{I:inv4}. To see this,  consider for $k\in \N$ the multiplicative action of dilations by
$k$-th powers on $\T$ (see \cref{SS:examples})
and let $F(x):=e(x)$, $x\in \T$.
For $k=1$ we have  $F\in X_a$ but $\nnorm{F}_{U^2}\neq 0$,
and for $k=2$ we have $\nnorm{F}_{U^2}=0$ but $\nnorm{F}_{U^3}\neq 0$.
So the next problem comes naturally.
\begin{problem}\label{Pr:Inverse}
	Find an inverse theorem and a decomposition result for the mixed seminorms $\nnorm{\cdot}_{U^s}$ that work for arbitrary multiplicative actions.
\end{problem}
Optimally,  given a multiplicative action $(X,\CX,\mu,T_n)$, we want to characterize the smallest factor $\CZ$ of the system such that if $F\in L^\infty(\mu)$ satisfies $\E(F| \CZ)=0$, then
$\nnorm{F}_{U^s}=0$.

In view of \cref{P:Characteristic1}, a solution to \cref{Pr:Inverse} is likely
to lead to progress on natural multiple recurrence and convergence problems
for arbitrary multiplicative actions, such as the following special case
of \cref{Con1}. It also provides a 
useful test of whether \cref{Pr:Inverse} has been solved in a sufficiently effective form.
  
\begin{problem}\label{Pr:linear}
	Let $(X,\CX,\mu,T_n)$ be a multiplicative action. Show that for all $F_0,F_1,\ldots, F_\ell\in L^\infty(\mu)$ the following limit exists
	$$
	\lim_{N\to\infty}	\E_{m,n\in [N]}\,\int T_mF_0\cdot   T_{m+n}F_1\cdots T_{m+\ell n}F_\ell\, d\mu.
	$$
\end{problem}
In the case of finitely generated actions the conclusion follows from \cref{T:MainA} and  it  holds even for several not necessarily commuting  multiplicative actions. 
On the other hand, the conclusion fails  if we remove the integrals and ask for mean convergence of the resulting averages; indeed, for a multiplicative rotation by $n^i$ even  the single term averages  $\E_{m,n\in [N]}\,  T_{(m+n)}F_1$ do not converge in $L^2(\mu)$.

\subsection{Density regularity for some homogeneous quadratic equations}
Next, we record some problems related to the density regularity (in the sense of \cref{D:prdr}) of   equations of the form \eqref{E:xyz}. For example, consider  the equation $x^2-y^2=xz$,
which is satisfied when $x=km^2, y=kmn, z=k(m^2-n^2)$, $k,m,n\in \N$.
Thus, to prove that this equation is density regular, it suffices to show that for any  multiplicative action
$(X,\CX, \mu,T_n)$ and
$A\in \CX$ with $\mu(A)>0$, we have
\begin{equation}\label{E:Fpos}
	\liminf_{N\to\infty}\E_{m,n\in [N], m>n} \, \int T_{m^2}F\cdot  T_{mn}F\cdot T_{m^2-n^2}F\, d\mu>0,
\end{equation}
where  $F:={\bf 1}_A$.
By \cref{T:DecompositionGeneral}, we have  $F=F_{p}+F_a$, where $F_p=\E(F| \CX_p)$ and $F_a\in X_a$. Using  \eqref{E:concsevlingenB} in \cref{P:ConcRgeneral} and by adjusting (non-trivially) the argument in \cref{SS:B2},  we can probably show  that \eqref{E:Fpos} holds  if we replace $F$ by $F_p$. So a key  remaining obstacle to  proving \eqref{E:Fpos} is to answer the next question:
\begin{problem}\label{Pr:charfact}
Let $(X,\CX, \mu,T_n)$ be a multiplicative action and $F_1,F_2, F_3\in L^\infty(\mu)$ be such that $F_2\in X_a$ or $F_3\in X_a$. Is it true that
\begin{equation}\label{E:mmmn}
\lim_{N\to\infty}\E_{m,n\in [N], m>n} \,  \int T_{m^2}F_1\cdot T_{mn}F_2\cdot T_{m^2-n^2}F_3\, d\mu=0 \, ?
\end{equation}
\end{problem}
\cref{P:Characteristic2} covers the case where the multiplicative action $(X,\CX, \mu,T_n)$ is finitely generated.  In the case of infinitely generated actions,
it is easy to verify that  for each $k\in \N$  the multiplicative
actions of dilations by
$k$-th powers on $\T$ (see \cref{SS:examples}), or products of such actions,  
do not pose obstructions to the vanishing property \eqref{E:mmmn}, since if  
$P_1(m^2)+P_2(mn)+P_3(m^2-n^2)=0$ for all $m,n\in \N$ and some $P_1,P_2,P_3\in \Z[t]$ with zero constant terms, then $P_1=P_2=P_3=0$.

\subsection{Controlling ``Pythagorean averages'' by mixed seminorms}
It is natural to explore if methods from  ergodic theory can
be used to prove partition regularity for Pythagorean triples, i.e., for the equation $x^2+y^2=z^2$, which is satisfied when $x=2mn, y=m^2-n^2, z=m^2+n^2$. We are naturally led to study
the limiting behavior of the averages in \eqref{E:Pythagorean} below and  to study the following problem:
\begin{problem}\label{Pr:Pythagorean}
Let $(X,\CX,\mu,T_n)$ be a multiplicative action and $F_1,F_2,F_3\in L^\infty(\mu)$ be such that $\nnorm{F_j}_{U^3}=0$ for $j=1$ or $2$. Is it true that
\begin{equation}\label{E:Pythagorean}
	\lim_{N\to\infty}\E_{m,n\in [N], m>n} \, \int  T_{mn}F_1\cdot T_{m^2-n^2}F_2\cdot T_{m^2+n^2}F_3\, d\mu=0\, ?
\end{equation}
%%	in $L^2(\mu)$\, ?
\end{problem}
We remark that the assumption $\nnorm{F_j}_{U^3}=0$ cannot be replaced by $F_j\in X_a$ or $\nnorm{F_j}_{U^2}=0$.   To see this, consider the
action of dilations by squares on $\T$ (see \cref{SS:examples}), and let $F_1(x):=e(4x)$, $F_2(x):=e(x)$, $F_3(x)=e(-x)$. Then $F_j\in X_a$ and $\nnorm{F_j}_{U^2}=0$ for $j=1,2,3,$ but all  the integrals in \eqref{E:Pythagorean} are $1$ for $m,n\in \N$.

Using a variant of  the Daboussi-K\'atai orthogonality criterion (see \cref{L:Katai}) for the Gaussian integers,  it is not hard to see that the needed seminorm control in \cref{Pr:Pythagorean}
	would follow from a similar seminorm control for averages of the form
$$
\E_{m,n\in[N]} \int T_{L_1(m,n)\cdot L_2(m,n)}F_1\cdot T_{L_3(m,n)\cdot L_4(m,n)}F_2\cdot T_{L_5(m,n)\cdot L_6(m,n)}F_3\cdot T_{L_7(m,n)\cdot L_8(m,n)}F_4\, d\mu,
$$
where $L_1(m,n),\ldots, L_8(m,n)$ are pairwise independent linear forms. We do not know how to handle this problem even for finitely generated multiplicative actions, in which case
$F$ has vanishing mixed seminorms if and only if
 $F\in X_a$ (see \cref{T:Inverse}).
 In this regard, we state the following simpler model problem:
\begin{problem}\label{Pr:6linear}
Let $(X,\CX,\mu,T_n)$ be a finitely generated multiplicative  action and let $F_j\in L^\infty(\mu)$, $j\in [3]$,  be such that $F_j\in X_a$ for some $j\in[3]$. Is it true that
$$
\lim_{N\to\infty}\E_{m,n\in [N], m>2n} \, \int T_{mn}F_1\cdot   T_{m^2-n^2}F_2\cdot T_{m^2-4n^2}F_3\, d\mu=0\,  ?
$$	
%	in $L^2(\mu)$?
\end{problem}
This problem  resembles the one  addressed in  \cref{P:Characteristic2}, which allows us to deal  with 
  three quadratic polynomials that factor linearly, but only when two of them  have a common linear factor. 
% six linear forms, but only when two  are equal and appear in different iterates.

\subsection{Characteristic factor larger than $X_p$.}
Finally, we record two  problems for finitely generated multiplicative actions
%that avoids the use of addition on the iterates and
where the  factor $X_p$ does not control the limiting behavior of the averages we aim to study.
\begin{problem}\label{Pr:NoAddition}
	Let $(X,\CX, \mu,T_n)$ be a finitely generated multiplicative action.  Show that  the limit
	$$
	\lim_{N\to\infty} \E_{m,n\in [N]} \, T_mF_1\cdot T_nF_2\cdot T_{m+n} F_3\cdot T_{mn}F_4
	$$
 exists in $L^2(\mu)$ for all 	 $F_1,F_2,F_3,F_4\in L^\infty(\mu)$, and
	$$
	\lim_{N\to \infty} \E_{m,n\in [N]}\,  \mu(A\cap T^{-1}_{m}A\cap T^{-1}_{n}A\cap T^{-1}_{m+n}A\cap T^{-1}_{mn}A)>0
	$$
	for all $A\in \CX$ with $\mu(A)>0$.
	\end{problem}
	It follows from work of Bowen and Sabok \cite{BS24}, which was later extended by Alweiss \cite{Al23}, that the patterns $\{km,kn,k(m+n),kmn\colon k,m,n\in \N\}$ are partition regular. However, the  recurrence part of \cref{Pr:NoAddition}  relates more to density regularity  and fails for general multiplicative actions.
	 See \cite[Theorem~1.4]{DLMS23}  for  related multiple recurrence results.
	
	Finally, we mention the following problem, which seems deceptively simple:
	\begin{problem}
	 Let $(X,\CX, \mu,T_n,S_n)$ be a finitely generated multiplicative action of commuting measure preserving transformations.
Show that the limit
	$$
	\lim_{N\to\infty} \E_{n\in [N]} \, T_nF_1\cdot S_nF_2
	$$
exists in $L^2(\mu)$ for all $F_1,F_2\in L^\infty(\mu)$. Similarly for more than two commuting actions.
\end{problem}

\end{document}